\documentclass[11pt,reqno,UTF8,twoside]{amsart}
\usepackage{mathrsfs}
\usepackage{amsfonts,amssymb,amsmath,amsthm}
\usepackage{cite}
\usepackage[colorlinks=true,citecolor=red]{hyperref}
\usepackage{titletoc}
\usepackage{geometry}
\usepackage{bm}
\usepackage{indentfirst}
\usepackage{graphicx}
\usepackage{float} 
\usepackage{booktabs}
\usepackage{longtable}
\usepackage{subfigure}
\usepackage{stmaryrd}
\usepackage{enumerate}
\usepackage{tikz}
\usepackage{ulem}
\usepackage{color,soul}
\usepackage{setspace}
\usepackage{stmaryrd}
\usepackage[pagewise]{lineno} 
\allowdisplaybreaks[2]
\usepackage{appendix}
\usepackage{cases}
\usepackage{todonotes}
\usepackage{enumitem}

\numberwithin{equation}{section}
\newtheorem{theorem}{Theorem}[section]
\newtheorem{lemma}[theorem]{Lemma}

\newtheorem{proposition}[theorem]{Proposition}
\newtheorem*{maintheorem}{Main Theorem}
\newtheorem*{motivation}{Motivation}

\newtheorem{remark}[theorem]{Remark}
\newtheorem{definition}[theorem]{Definition}
\newtheorem{example}[theorem]{Example}

\newcommand{\N}{\mathbb{N}}
\newcommand{\R}{\mathbb{R}}

\newcommand{\E}{\mathbb{E}}

\newcommand{\Pb}{\mathbb{P}}

\newcommand{\supp}{\operatorname{supp}}

\newcommand{\runum}[1]{\romannumeral #1}
\hypersetup{linkcolor=blue}

    \makeatletter
    \@namedef{subjclassname@2020}{\textup{2020} Mathematics Subject
	Classification}
    \makeatother

\begin{document}

     \title[ Mixing for Allen--Cahn equation ]
	{{\Large E{\MakeLowercase{xponential mixing for the stochastic} A{\MakeLowercase{llen--}C{\MakeLowercase{ahn equation}\\
    \vspace{1mm}
    \MakeLowercase{ with localized white noise}}}}
    }}

    \author[Z. Liu, S. Xiang, Z. Zhang]{ {\small Z\MakeLowercase{iyu} L\MakeLowercase{iu}, S\MakeLowercase{hengquan} X\MakeLowercase{iang}, Z\MakeLowercase{hifei} Z\MakeLowercase{hang}}}
    
    \address[Ziyu Liu]{School of Mathematics and Physics, University of Science and Technology Beijing, 100083, Beijing, China.}
    \email{ziyu@ustb.edu.cn}

    \address[Shengquan  Xiang]{School of Mathematical Sciences, Peking University, 100871, Beijing, China.}
    \email{shengquan.xiang@math.pku.edu.cn}

    \address[Zhifei  Zhang]{School of Mathematical Sciences, Peking University, 100871, Beijing, China.}
    \email{zfzhang@math.pku.edu.cn}

    \begin{abstract}
     This paper studies the 1D stochastic Allen--Cahn equation on a bounded domain driven by localized white noise. We prove that the associated Markov process admits a unique invariant measure and is exponential mixing. The main challenge lies in the interaction between localized nature of the noise and non-trivial global dynamics of the system. To overcome this, our approach relies on two ingredients from PDE control theory: stabilization for the linearized system and global steady-state controllability for the nonlinear equation. The stabilization result is derived using the weak observability and Fenchel--Rockafellar duality, while the global controllability relies on quasi-static deformations combined with global dynamics.

    \end{abstract}

	

    \maketitle
    \setcounter{tocdepth}{1}
    \tableofcontents

    \section{Introduction}\label{Sec 1}

    \subsection{Backgrounds and motivations}\label{Sec 1.1}

    Ergodicity and mixing for SPDEs, crucial for understanding the statistical behavior of  physical models, have achieved significant progress in recent decades. By now, it is well understood that when all determining modes of the underlying PDE are directly affected by noise; see e.g. \cite{EMS-01, KS-02, BKL-02, FM-95, Mattingly-02, MY-02}. On the other hand, the situation is much less clear when the random force does not transmit directly to the determining modes of the system.

    In recent years, two main directions have emerged in addressing this problem. The first concerns SPDEs driven by white noise that is highly degenerate in Fourier space. This setting is first studied by Hairer and Mattingly \cite{HM-06, HM-08, HM-11b}, where ergodicity and exponential mixing are established for the 2D Navier--Stokes system. Building on these works and the mixing theory in \cite{BBPS-22c, BBPS-22d}, exponential mixing for passive scalars with highly degenerate noise is recently established in \cite{CR-25}. See also \cite{FGRT-15, KNS-20,PZZ-24,NZZ-24} for other results in the context of highly degenerate noise.
    
    The second direction, first investigated by Shirikyan \cite{Shi-15, Shi-21}, focuses on Navier--Stokes perturbed by physically localized noise, where the noise is bounded in time. More recently, the authors and their coauthors established a mixing criterion for hyperbolic/dispersive systems; as applications, exponential mixing for nonlinear wave equations and nonlinear Schrödinger equations with localized noise are obtained in \cite{LWXZZ-24, CXZZ-25,CLXZ-24,CX-26}. 
   
    \vspace{0.6em}

    In contrast, there is a gaping lack of results for systems with physically degenerate white noise, where the white-in-time forcing is transmitted to the system only through a small region of the spatial domain.  The corresponding questions, such as the ergodicity for the Navier--Stokes equations in this setting, are to our knowledge, quite challenging and remain open. 
    
    \begin{motivation} 
        Turbulence and statistical properties of Navier--Stokes equations with localized white noise.
    \end{motivation}    
    
    \noindent As a first step toward understanding such phenomena, we investigate in this work the case of the 1D Allen--Cahn equation and establish exponential mixing.  More specifically, this paper aims to provide a systematic study of the mixing properties for SPDEs driven by localized white noise.  To tackle this problem, the novelty of our approach lies in bringing together the global dynamics of underlying system, PDE control theory and stochastic analysis in a natural and general way.  In particular, since the localized nature of the random perturbation is intrinsically connected with the control theory,  the corresponding control problems therefore play a crucial role in our study.

    \vspace{0.3em}
    
    The Allen--Cahn equation was introduced as a gradient-flow model for phase separation, associated with the Ginzburg--Landau energy
    \begin{equation*}
        E_{\varepsilon}(u)=\int_D\left(\frac{\varepsilon}{2}|\nabla u|^2+\frac{1}{\varepsilon}F(u)\right)dx,\quad F(u)=\frac{1}{4}(u^2-1)^2.
    \end{equation*}
    Here the small parameter $\varepsilon$ represents the thickness of diffuse interfaces, and in the sharp-interface limit $\varepsilon\rightarrow0$, solutions converge to motion by mean curvature. For analytical convenience the equation is often written in the form
    \begin{equation*}
        \partial_tu-\nu\Delta u+u^3-\lambda u=0,
    \end{equation*}
    where the classical scaling corresponds to $\nu=\varepsilon$, $\lambda=\varepsilon^{-1}$. The physically relevant regime is therefore the high-contrast setting $\lambda\gg\nu$, in which solutions exhibit metastable structures separated by thin transition layers. In this paper we work in the $(\nu,\lambda)$-formulation and study the ergodic behavior of the system under localized stochastic perturbations.

    \subsection{Main result}\label{Sec 1.2}
    The stochastic Allen--Cahn equation with localized white noise under consideration is given by
    \begin{equation}\label{AC equation}
    \begin{cases}
         \partial_tu-\nu \partial_x^2u+u^3-\lambda u=\frac{dW}{dt}(t,x),\quad x\in (0,\pi),\;t>0,\\
         u|_{x=0,\pi}=0,\\
        u(0,\cdot)=u_0(\cdot).
    \end{cases}
    \end{equation}
    Here $\nu,\lambda>0$ are two arbitrarily given parameters. The localized random force $W(t,x)$ is a white noise of the form  
    \begin{equation}\label{eq noise W}
        W(t,x)=\sum_{j\in\N^+}b_j\beta_j(t)\psi_j(x).
    \end{equation}
        Here $b_j$ are real numbers, $\{\beta_j(t)\}_{j\in\N^+}$ is a sequence of independent standard Brownian motions defined on a filtered probability space $(\Omega,\mathcal{F}, (\mathcal{F}_t)_{t\geq 0},\Pb)$, and $\{\psi_j\}_{j\in\N^+}$ are continuous functions {\it supported in a sub-interval} $[a,b]\subset [0,\pi]$.

    We consider the equation \eqref{AC equation} in the space of 
    \begin{equation*}
    H= L^2(0,\pi),
    \end{equation*}
    endowed with the usual $L^2$-norm $\|\cdot\|$. The eigenvectors of $-\partial_{x}^2$ on $(0,\pi)$  with Dirichlet boundary conditions are denoted by $\{e_j(x)=\sqrt{2/\pi}\sin(j x)\}_{j\in\N^+}\subset H^1_0(0,\pi)$, which form an orthonormal basis of $L^2(0, \pi)$.  The corresponding eigenvalues are  $j^2$. The Hilbert space $L^2(a,b)$ can be viewed as  a subset of $H$ by extending the functions in $L^2(a,b)$ by zero outside $(a,b)$. We consider that $\{\psi_j\}_{j\in\N^+}\subset H_0^1(a,b)$ forms an orthonormal basis on $L^2(a,b)$. Without loss of generality, let us fix any $0\leq a<b\leq \pi$ and  take 
    \begin{equation}\label{eq psi-def}
        \psi_j(x)=\begin{cases}
            \sqrt{\frac{2}{b-a}}\sin \left(j\pi\frac{x-a}{b-a}\right)\quad &\text{for }x\in[a,b],\\
            0\quad &\text{otherwise}.
        \end{cases}
    \end{equation}

    Our main result concerning exponential mixing  is contained in the following.

    \begin{maintheorem}\label{thm AC mixing}

    For the equation given by \eqref{AC equation}-\eqref{eq psi-def} and any $B_0>0$, there exists an integer $N_*=N_*(B_0)\in \N^+$ such that the following holds. If the sequence $\{b_j\}_{j\in\N^+}$ satisfies
    \begin{equation*}
       \sum_{j\in\N^+} j^2b_j^2\leq B_0\quad \text{and}\quad b_j\neq 0\quad \text{for }1\leq j\leq N_*,
    \end{equation*}
    then the Markov process associated to \eqref{AC equation} admits a unique invariant measure $\mu_*$.  Moreover,  there exist constants $C,\alpha,\gamma>0$ such that
    \begin{equation}\label{eq exponential mixing}
    \left|\E f(u(t,u_0))-\int_{H}fd\mu_*\right|\leq  Ce^{-\alpha t}\exp\left(\gamma\|u_0\|^2\right)\|f\|_{\gamma} 
    \end{equation}  
    for any $ f\in \mathcal{O}_\gamma$, $u_0\in H$ and $t\geq 0$, where
    \begin{equation*}
        \mathcal{O}_\gamma:=\left\{f\in C^1(H):\|f\|_{\gamma}<\infty\right\},\quad \|f\|_{\gamma}:=\sup_{u_0\in H}\left(\exp\left(-\gamma\|u_0\|^2\right)(|f(u_0)|
        +\|\nabla f(u_0)\|)\right).
    \end{equation*}
    \end{maintheorem}

    For simplicity, this paper addresses mixing properties for the 1D semilinear parabolic equations with Dirichlet boundary condition. Nevertheless, the approach could be adapted for other equations with localized white noise. Moreover, it could also be further applied to accommodate other noise models,  such as boundary white noise or Lévy-type noise.  We  also note that the  mixing results and the underlying methods can further be applied to investigate other statistical properties of the system, e.g. large deviations.

    \subsection{Obstructions and ingredients} \label{Sec 1.3}  To establish the Main Theorem, our strategy combines a general probabilistic framework for mixing with two  ingredients from the PDE control theory. Specifically, a stabilization result for the linearized system is obtained via weak observability and the Fenchel--Rockafellar duality, while the global steady-state controllability is achieved by employing quasi-static deformations together with the global dynamics of the unforced system.
    
    \vspace{0.6em}
    \noindent{1.3.0. \it Probabilistic framework.}

    \vspace{0.3em}
    The general criterion for exponential mixing that we apply is due to Hairer and Mattingly \cite{HM-06,HM-08,HM-11b}, i.e. Theorem~\ref{lemma HM}. This framework reduces mixing to verifying three key conditions: 
    \begin{itemize}
        \item[(\runum{1})] the existence of a suitable Lyapunov function;
        \item[(\runum{2})] the asymptotic strong Feller property;
        \item[(\runum{3})] a weak form of irreducibility.
    \end{itemize}
    Among these, the verification of the latter two properties contributes to the main challenges in the present setting, due to both the localized nature of the noise and the non-trivial global dynamics of the system.

    \vspace{0.6em}
   
    \noindent{1.3.1. \it Two  obstructions for mixing.}

    \vspace{0.3em}

    \noindent {\bf Asymptotic strong Feller.}  The first issue to address is the asymptotic strong Feller property. Such semigroup regularities are standard tools in the study of ergodicity and mixing. In fact, recent works \cite{GLLL-24,GLLL-25} show that such conditions are in many cases necessary for mixing. Roughly speaking, the asymptotic strong Feller property can usually established via gradient estimates of the form
    \begin{equation*}
        |\nabla P_t\varphi(u_0)|\leq C(\|u_0\|)\left(\|\varphi\|_{\infty}+\delta_t\|\nabla\varphi\|_{\infty}\right)
    \end{equation*}
    for any $C^1$-smooth observable $\varphi$ on $H$. Here $C(\|u_0\|),\delta_t>0$ are constants, and $\lim_{t\rightarrow\infty} \delta_t=0$.

    In classical settings, several mechanisms are available. When the system is perturbed by non-degenerate white noise, the strong Feller property can be established; see e.g. \cite{FM-95,HZ-24}. When the determining modes of the system are forced, ergodicity can often be obtained via a change of measure through Girsanov’s theorem combined with the pathwise contractive nature of the dynamics; see \cite{EMS-01,BKL-02,Mattingly-02,KS-02,NZ-24}. In the hypoelliptic framework \cite{HM-06}, the asymptotic strong Feller property is achieved by exploiting detailed structural information on the Navier--Stokes nonlinearity from \cite{MP-06}, together with a high-low Fourier mode splitting.

    In contrast, in the present setting with physically degenerate noise, the localized random perturbation interacts with the nonlinearities in a highly nontransparent way. This rules out not only the application of Girsanov’s theorem, but also a direct use of structural information on the nonlinearity. Moreover, the high-low mode interaction requires a refined frequency analysis, owing to the internal nature of the forcing.

    \vspace{0.6em}

    \noindent {\bf Irreducibility.} The second difficulty concerns the irreducibility, which also plays an equally essential role in establishing ergodicity and mixing; see, e.g. \cite{DPZ-96,HM-11b}.  Namely, the weak irreducibility under consideration can be formulated as follows: there exists a state $u_*\in H$ such that, for any $R,\varepsilon>0$, there exists $T>0$, for any $t\geq T$,
       \begin{equation*}
        \inf_{u_0\in B_H(0,R)}P_t(u_0,B(u_*,\varepsilon))>0. \end{equation*}
    
    By support theorems, the irreducibility problem is often reduced to the analysis of the global dynamics of the unforced equation. In many situations, this reduction leads to relatively straightforward arguments; see, e.g.  \cite{EM-01,LWXZZ-24}. For the Allen--Cahn system considered here, when diffusion dominates the destabilizing linear term, i.e. $\lambda \leq \nu$, the system is globally stable, and irreducibility follows readily.

    Conversely, when $\lambda > \nu$,  the unforced Allen--Cahn equation admits multiple steady states, and the long-time behavior depends sensitively on the initial condition. In this regime,  mixing is governed by the interaction between the dynamics and the noise. For instance, the Lorenz system in \cite{CZH-21} exhibits a transition from a unique invariant measure to multiple ergodic measures.  For the Allen--Cahn equation with highly degenerate noise \cite[Section 8.4]{HM-11b}, irreducibility follows from noise acting on all spatial locations. Nevertheless, the present noise is {\it localized}, and this mechanism is no longer available. As a result, establishing irreducibility becomes more delicate.

    \vspace{0.6em}

    \noindent{1.3.2. \it Ingredients from control theory.}

    \vspace{0.3em}

    \noindent {\bf Stabilization for the linearized system.}
    To address the asymptotic strong Feller property, we invoke the technical route proposed in \cite{HM-06} using Malliavin calculus, and adapt it to our setting. The main objective is to construct suitable controls to stabilize the linearized stochastic Allen--Cahn equation, which is achieved by analyzing the invertibility of Malliavin matrix.

    Indeed, it turns out that the Malliavin matrix is naturally related to the Gramian matrix in control theory. This allows us to reformulate the question as a deterministic control problem. In this framework, the study of the Gramian matrix is equivalent to an observability analysis, by virtue of the duality between observability and controllability.

    In particular, building on the ``{\it frequency analysis}'' strategies from our earlier work \cite{LWXZZ-24,Xiang-24}, we establish an explicit low-frequency control construction together with quantitative stabilization estimates. Here the key tools  are a quantitative version of {\it weak observability inequality} and {\it Fenchel--Rockafellar duality} arguments.  Consequently, with this explicit deterministic control construction at hand, we are able to design the required control inputs to establish the asymptotic strong Feller property. See Section~\ref{Sec 3} for more details.

   \begin{figure}[th]
    \centering

\tikzset{every picture/.style={line width=0.75pt}} 

\begin{tikzpicture}[x=0.75pt,y=0.75pt,yscale=-1,xscale=1]

\draw [color={rgb, 255:red, 0; green, 0; blue, 0 }  ,draw opacity=0.2 ]   (191.4,17.3) -- (213.29,36.27) ;
\draw [color={rgb, 255:red, 0; green, 0; blue, 0 }  ,draw opacity=0.2 ]   (213.29,36.27) -- (191.4,54.3) ;

\draw [color={rgb, 255:red, 0; green, 0; blue, 0 }  ,draw opacity=0.2 ]   (417.4,36.85) -- (439.19,50.36) ;
\draw [color={rgb, 255:red, 0; green, 0; blue, 0 }  ,draw opacity=0.2 ]   (439.19,50.36) -- (417.4,63.2) ;

\draw (15,26) node [anchor=north west][inner sep=0.75pt]  [font=\normalsize] [align=left] {{\fontfamily{ptm}\selectfont Weak observability inequality}};
\draw (19.27,46.6) node [anchor=north west][inner sep=0.75pt]  [font=\normalsize] [align=left] {{\fontfamily{ptm}\selectfont Fenchel--Rockafellar duality }};
\draw (213.7,28.27) node [anchor=north west][inner sep=0.75pt]  [font=\normalsize] [align=left] {{\fontfamily{ptm}\selectfont Stabilization for linearized system}};
\draw (257.21,54.29) node [anchor=north west][inner sep=0.75pt]  [font=\normalsize] [align=left] {{\fontfamily{ptm}\selectfont Malliavin calculus }};
\draw (439.58,41.67) node [anchor=north west][inner sep=0.75pt]  [font=\normalsize] [align=left] {{\fontfamily{ptm}\selectfont Asymptotic strong Feller}};
\draw (50,6) node [anchor=north west][inner sep=0.75pt]  [font=\normalsize] [align=left] {{\fontfamily{ptm}\selectfont Carleman estimate}};

\end{tikzpicture}
    
   \vspace{-0.6em}
      \caption{Proof of the asymptotic strong Feller.}\label{figure 1}
   \end{figure}
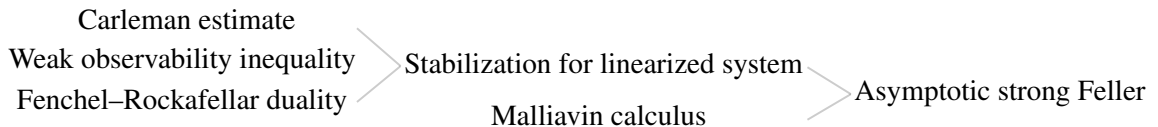

\vspace{-0.3em}

    \noindent {\bf Global steady-state controllability.}  It remains to establish irreducibility in the parameter regime $\lambda>\nu$. Roughly speaking, this means showing the existence of a reference state $u_*$ such that, with positive probability, the system can reach any small neighborhood of $u_*$.  Due to support theorems, it suffices to construct deterministic forces to drive deterministic solutions into the neighborhood of $u_*$.   To ensure positivity of the probability, a key point is that only a finite number of control trajectories is constructed, uniformly with respect to the initial state.

    Therefore, the irreducibility problem reduces to a global approximate controllability problem. Moreover, the controls must satisfy certain regularity requirements in order to belong to the Cameron--Martin space. 
    Global controllability with localized control is a central problem in PDE control theory, see \cite{CT-04,CX-25} for related results and the discussion in \cite[Section 1.2.2]{CKX-25}. In particular,  for the case of the geometric heat flow model \cite{CX-25} solved the challenging open problem for the small-time global controllability between steady states of nonlinear heat equations.

    To achieve this, our proof strategy is based on the {\it quasi-static deformation method} developed in \cite{Coron-02,CT-04}. More precisely, we first establish steady-state controllability for the deterministic equation. This is achieved by combining controlled Lyapunov theory with quasi-static deformation, and by constructing explicit feedback-type controls that connect different steady states. In the next step, we exploit the  global dynamics of the unforced equation.  The details are presented in Section~\ref{Sec 4}.

    \begin{figure}[th]
        \centering

\tikzset{every picture/.style={line width=0.75pt}} 

\begin{tikzpicture}[x=0.75pt,y=0.75pt,yscale=-1,xscale=1]

\draw [color={rgb, 255:red, 0; green, 0; blue, 0 }  ,draw opacity=0.2 ]   (152.4,12.87) -- (174.29,25.5) ;
\draw [color={rgb, 255:red, 0; green, 0; blue, 0 }  ,draw opacity=0.2 ]   (174.29,25.5) -- (152.4,37.51) ;

\draw [color={rgb, 255:red, 0; green, 0; blue, 0 }  ,draw opacity=0.2 ]   (378.14,26) -- (404.14,26) ;

\draw (2.67,3) node [anchor=north west][inner sep=0.75pt]  [font=\normalsize] [align=left] {{\fontfamily{ptm}\selectfont Quasi-static deformation }};
\draw (28.27,26.6) node [anchor=north west][inner sep=0.75pt]  [font=\normalsize] [align=left] {{\fontfamily{ptm}\selectfont Global dynamics}};
\draw (176.7,17.27) node [anchor=north west][inner sep=0.75pt]  [font=\normalsize] [align=left] {{\fontfamily{ptm}\selectfont Global steady-state controllability}};
\draw (405.58,17.3) node [anchor=north west][inner sep=0.75pt]  [font=\normalsize] [align=left] {{\fontfamily{ptm}\selectfont Irreducibility}};

\end{tikzpicture}
 \vspace{-0.6em}
        \caption{Proof of the irreducibility.}
        \label{figure 2}
    \end{figure}
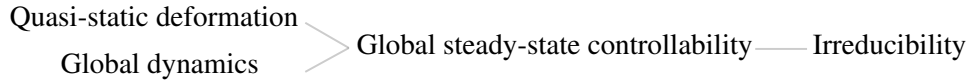
    
\vspace{-1em}    

 \subsection{Organization of the paper and guide to notation} This paper is organized as follows.  Section~\ref{Sec 2}  collects some probabilistic preliminaries for the mixing problem, together with the relevant notions from Malliavin calculus. In Section~\ref{Sec 3}, we establish the asymptotic strong Feller property for the associated Markov process, relying on quantitative stabilization estimates for the linearized Allen--Cahn system. Section~\ref{Sec 4} is devoted to proving irreducibility, which is based on the global steady-state controllability of the deterministic Allen--Cahn equation.  In Section~\ref{Sec 5}, we combine the above results to complete the proof of the Main Theorem. Finally, the Appendix gathers several technical results and auxiliary proofs used in the main text.

    \subsubsection*{A guide to notation}  We include here for the convenience of the reader a guide to commonly used notations in each part of this paper.

    \begin{itemize}[leftmargin=1em]
        \item[\tiny$\bullet$] Throughout this paper, the  settings for the Allen--Cahn equation are as follows:

         Let $H=L^2(0,\pi)$ be a Hilbert space, equipped with the usual $L^2$-inner product $\langle \cdot,\cdot\rangle$ and norm $\|\cdot\|$. We denote by $\{e_{j}\}_{j\in\N^+}$ the orthonormal basis of $H$ consisting of the eigenfunctions of the Dirichlet Laplacian, with corresponding eigenvalues $j^2$.  The noise $W=\sum_{j\in\N^+}b_j\beta_j\psi_j$ given by \eqref{eq noise W} is spacial localized in a sub-interval $(a,b)\subset[0,\pi]$. Here $b_j$ are real numbers and $\beta_j$ are independent stand Brownian motions.  The sequence $\{\psi_j\}_{j\in\N^+}$,  defined by \eqref{eq psi-def}, forms an orthogonal basis of  $L^2(a,b)$. The notation $\mathsf{P}_N$ stands for the projection in $L^2(a,b)$ onto the finite-dimensional subspace spanned by $\{\psi_j\}_{1\leq j\leq N}$. 

        \vspace{0.3em}

        Throughout, we use $u(t,u_0)=u(t,u_0,W)=u_t$\footnote{Throughout this paper, subscripts $t$ denote the value of the solution at time $t$. Partial derivatives are written as $\partial_t$, $\partial_x$ and  $\partial_x^2$.} to denote the solution of the stochastic equation \eqref{AC equation}.  The corresponding linearized equation is denoted by $\mathcal{J}_{s,t}$, satisfying \eqref{J equation} and its adjoint $\mathcal{J}_{s,t}^*$ satisfies \eqref{J* equation}. In Section \ref{Sec 3}, $z(t)$ denotes the solution of linearized equation \eqref{z equation}, while $y(t)$ stands for the solution of adjoint equation \eqref{y equation}. In Section \ref{Sec 4}, we use $w(t,u_0,\zeta)$ and $\phi(t,u_0)$ to represent the solutions of the controlled deterministic equation \eqref{AC determine equation} and unforced equation \eqref{AC unforced equation}, respectively.

        \vspace{0.6em}

        \item[\tiny$\bullet$] Section \ref{Sec 2}, i.e. the probabilistic setting, frequently uses the following notation:

        For the Markovian framework, we lay out the following notations for a Hilbert space $X$. The symbol $\mathcal{B}(X)$ denotes the Borel $\sigma$-algebra on $X$ and $\mathcal{P}(X)$ is the set of Borel probability measures on $X$. We write $B_b(X)$ for the space of bounded Borel measurable functions on $X$. The Markov semigroup associated with equation \eqref{AC equation} on $H$ are denoted by $P_t,P_t^*$.
        
        \vspace{0.3em}
         The Malliavin calculus setting invokes the following. The operator $\mathcal{D}$ denotes the Malliavin derivative, taking directions in $\mathcal{H}=L^2(\R^+;L^2(a,b))$. The random operators $\mathcal{A}_{s,t}$ on $\mathcal{H}_{s,t}=L^2(s,t;L^2(a,b))$ and its dual $\mathcal{A}_{s,t}^*$ are defined in \eqref{eq A def},\eqref{eq A* def}. We also use the truncated operators $\mathcal{A}_{s,t,N}=\mathcal{A}_{s,t}\mathsf{P}_N$, $\mathcal{A}_{s,t,N}^*=\mathsf{P}_N\mathcal{A}_{s,t}^*$, along with the truncated  Malliavin matrix $\mathcal{M}_{s,t,N}=\mathcal{A}_{s,t,N}\mathcal{A}_{s,t,N}^*$. The operators $B$,$B^*$ introduced in \eqref{eq B def},\eqref{eq B* def} are also used in this setting.

        \vspace{0.6em}

        \item[\tiny$\bullet$] Section \ref{Sec 3} applies the following symbols to study the stabilization for the linearized equation:
        
        The letter $g\in L^\infty(Q)$ stands for a potential term on the domain $Q=(s,t)\times(0,\pi)$. To align with the stochastic setting, we use the nation $\mathcal{J}_{s,t,g}$ to represent  the solution operator of equation \eqref{z equation}, and $\mathcal{J}_{s,t,g}^*$ for the solution operator of equation \eqref{y equation}. The associated operators $\mathcal{A}_{s,t,N,g}$ and $\mathcal{A}_{s,t,N,g}^*$ are given by \eqref{eq A-A*def2}. The corresponding Gramian operator is denoted by $\mathcal{G}_{s,t,N,g}=\mathcal{A}_{s,t,N,g}\mathcal{A}_{s,t,N,g}^*$.  We also set $\hat{b}_N=\max\{|b_j|^{-1}:1\leq j\leq N\}$.

         \vspace{0.6em}

        \item[\tiny$\bullet$]  Section \ref{Sec 4} uses the notations in the context of global steady-state controllability:

        The set $\mathcal{S}$ in \eqref{eq S def} denotes the collection of equilibria for the unforced system \eqref{AC unforced equation}; while $\mathcal{S}_{\rm ex}$ represents the extended equilibrium set in \eqref{eq Sh def}. We use $(\bar{y}(\tau,x),\bar{h}(\tau,x))$ to denote a $C^1$-path in $\mathcal{S}_{\rm ex}$ connecting two steady-states $\phi,\hat{\phi}\in\mathcal{S}$. Moreover, we define a family of operators $A(\tau)$ by \eqref{eq A_tau def}, with corresponding eigenfunctions $e_{j}(\tau,\cdot)$ and eigenvalues $\lambda_j(\tau)$. The projection operator onto the first $m$ modes is denoted by $\mathscr{P}(\tau)$, where $m$ is chosen by \eqref{eq m def}.

    \end{itemize}

    \vspace{0.3em}

     Let $X,Y$ be separable Banach spaces. We write $B_{X}(R)$ for the closed ball in $X$ of radius $r$ centered at zero. $\mathcal{L}(X;Y)$ stands for the space of continuous linear operators from $X$ to $Y$ with the usual operator norm.

    \vspace{0.3em} 
    Additionally, $\mathbf{1}_A(\cdot)$ denotes the indicator function on set $A$. For $a,b\in\R$, we use $a\lor b$, $a\wedge b$ to denote their maximum and minimum. For $a\in\R^+$, the notation $\lfloor a \rfloor$ and $\lceil a \rceil$ denote the greatest integer less than or equal to $a$  and the smallest integer greater than or equal to  $a$, respectively.

    \section{Probabilistic setup and Malliavin calculus}\label{Sec 2}

    In this section, we first summarize the mathematical setup of mixing problem for the stochastic Allen--Cahn equation \eqref{AC equation}. Additionally, we also introduce the Malliavin calculus setting, which will play a crucial role in the subsequent analysis.

    \subsection{Well-posedness and Markovian framework}\label{Sec 2.1}
    To begin with, let us formulate equation \eqref{AC equation} as an abstract evolution equation on a Hilbert space and define its associated Markovian framework.
    
    \vspace{0.3em}
    
    Recall that the phase space $H=L^2(0,\pi)$, endowed with the $L^2$-norm $\|\cdot\|$. We denote the associated inner product on $H$ by $\langle\cdot,\cdot\rangle$.  We write $H_0^1(0,\pi)=\{v\in H^1(0,\pi):v(0)=v(\pi)=0\}$. On the spaces $H^k(0,\pi)\cap H_0^1(0,\pi)$ for $k=1,2$, we use the equivalent spectral norms defined by
    \begin{equation*}
    \|v\|_{H^k}^2 =\sum_{j\in\N^+} j^{2k} \langle v,e_j\rangle^2,\qquad v\in H_0^1(0,\pi)\cap H^k(0,\pi) .
    \end{equation*}
    We also use the usual $H^1$-norm on the space $H^1(a,b)$, denoted by $\|\cdot\|_{H^1(a,b)}$.

    In order to rewrite \eqref{AC equation} in a functional form we introduce the following two abstract operators:
    \begin{equation*}
        A\colon D(A)=H_0^1(0,\pi)\cap H^2(0,\pi)\rightarrow H,\quad Av=-\nu\partial_x^2v,
    \end{equation*}
    and 
    \begin{equation*}
       F\colon H^1(0,\pi)\rightarrow H,\quad F(v)=v^3-\lambda v.
    \end{equation*}
    
    With these notations, equation \eqref{AC equation} is then written as an abstract stochastic evolution equation on $H$, which reads
    \begin{equation}\label{AC equation2}
        du+(Au+F(u))dt=dW,\quad u(0)=u_0\in H.
    \end{equation}
    We say that $u=u(t,u_0)$ is a solution of \eqref{AC equation2} if it is $\mathcal{F}_t$-adapted,
    \begin{equation}\label{eq sol_reg}
        u\in C([0,\infty);H)\cap L^2_{loc}([0,\infty);H^1(0,\pi))\quad a.s.,
    \end{equation}
    and $u$ satisfies \eqref{AC equation2} in the mild sense, that is,
    \begin{equation*}
        u(t)=e^{-At}u_0-\int_0^te^{-A(t-s)}F(u(s))ds+\int_0^te^{-A(t-s)}dW(s).
    \end{equation*}

    Recall that the localized noise $W$ is given by \eqref{eq noise W}. To formulate the differentiability with respect to the noise, let $\mathsf{P}_N$ be the projection in $L^2(a,b)$ onto the finite-dimensional subspace spanned by $\{\psi_j\}_{1\leq j\leq N}$ with $N\in\N^+$. We set 
    \begin{equation*}
        W_N(t,x)=\mathsf{P}_NW(t,x)=\sum_{1\leq j\leq N}b_j\beta_j(t)\psi_j(x).
    \end{equation*}
    Note that $W_N\in C([0,\infty);\mathsf P_NL^2(a,b))$ a.s., where $\mathsf{P}_NL^2(a,b):=\text{span}\{\psi_j:1\leq j\leq N\}$. Identifying $\mathsf{P}_NL^2(a,b)$ with $\R^N$, we may view $W_N$ as an element of $C([0,\infty);\R^N)$.

    \vspace{0.3em}
    
    The following proposition summarizes the basic well-posedness, regularity, and smoothness with respect to data for \eqref{AC equation}.

    \begin{proposition}\label{prop well-posed}
      For the equation given by \eqref{AC equation}-\eqref{eq psi-def}, assume that there exists $N_*\in\N^+$ such that the sequence $\{b_j\}_{j\in\N^+}$ in \eqref{eq noise W} satisfy
    \begin{equation*}
       \sum_{j\in\N^+} j^2b_j^2<\infty\quad \text{and}\quad b_j\neq 0\quad \text{for }1\leq j\leq N_*.       
    \end{equation*}
        Then for any $u_0\in H$, there exists a unique solution $u\colon[0,\infty)\times \Omega\rightarrow H$ of \eqref{AC equation} which is $\mathcal{F}_t$-adapted and satisfies \eqref{eq sol_reg}.  For any $t\geq 0$ and realization of the noise $W(\cdot,\omega)$, the map $u_0\mapsto u(t,u_0)$ is Fréchet differentiable on $H$. Meanwhile, for any $u_0\in H$,  $t\geq 0$ and $1\leq N\leq N_*$, $W_N\mapsto u(t,u_0,W)$ is Fréchet differentiable from $C([0,t];\R^{N})$ to $H$. Moreover, the solution $u$ has spatial regularity of
      \begin{equation*}
          u\in C([s,\infty);H_0^1(0,\pi))\quad a.s.\quad \forall\,s>0.
      \end{equation*}      
      Finally, the solution $u$ satisfies certain a priori moment bounds collected in Lemma \ref{lemma L bdd}.
    \end{proposition}

    Throughout this paper,  we use the notations $u(t,x)=u(t,u_0)=u(t,u_0,W)=u_t$ to denote the unique solution of equation \eqref{AC equation}, depending on the context. The meaning will be clear from the circumstances and should not cause any confusion.

    \vspace{0.3em}

     Using the well-posedness by Proposition \ref{prop well-posed}, equation \eqref{AC equation} defines a Feller family of Markov processes in $H$. The transition function is given by    
    \begin{equation*}
        P_t(u_0,A)=\Pb(u(t,u_0)\in A)\quad \text{for }A\in\mathcal{B}(H),\;u_0\in H,\;t\geq 0.
    \end{equation*}
    The corresponding Markov semigroup $P_t\colon B_b(H)\rightarrow B_b(H)$ and its dual $P^*_t\colon\mathcal{P}(H)\rightarrow \mathcal{P}(H)$ defined by
	\begin{equation*}
		P_t\varphi(u_0)= \int_{H}\varphi(\hat{u})P_t(u_0,d\hat{u}),\quad\quad P_t^*\mu(A)=\int_{H}P_t(u_0,A)\mu(du_0)
	\end{equation*}
    for $\varphi\in B_b(H)$, $\mu\in\mathcal{P}(H)$, $u_0\in H$ and $A\in\mathcal{B}(H)$. Recall that a probability measure $\mu\in \mathcal{P}(H)$ is called \textit{invariant} for $\{P_t^*\}_{t\geq 0}$  if $P_t^*\mu=\mu$ for any $t\geq 0$. Our goal is to investigate exponential mixing for the Markov process $\{u(t,u_0)\}_{t\geq 0,\;u_0\in H}$ on $H$, i.e. the Main Theorem.

    \subsection{Malliavin calculus and the Allen--Cahn}\label{Sec 2.2} 
    
    To establish the asymptotic strong Feller property, we follow the approach developed in \cite{HM-06,HM-11b,FGRT-15}, which relies on the Malliavin calculus techniques. To this end, we introduce in this subsection the associated definitions and summarize some preliminaries for the stochastic Allen--Cahn equation \eqref{AC equation}.

    \vspace{0.3em}
    
    Let us introduce the separable Hilbert space  
    \begin{equation*}
        \mathcal{H}=L^2(\R^+\times\N^+,\mathcal{B}(\R^+\times\N^+),\text{Leb}\times\#)\cong L^2(\R^+;l^2(\N^+))\cong L^2(\R^+;L^2(a,b)),
    \end{equation*}
    where Leb denotes the Lebesgue measure on $\R^+$ and $\#$ denotes the counting measure on $\N^+$. The sequence of standard independent Brownian motions $\{\beta_j(t)\}_{j\in\N^+}$ is isometry to an isonormal Gaussian process on $\mathcal{H}$. Specifically,  let $\mathsf{W}=\{\mathsf{W}(h):h\in\mathcal{H}\}$ be a centered Gaussian family of random variables such that 
    \begin{equation*}
        \E \mathsf{W}(h)\mathsf{W}(g)=\langle h,g\rangle_{\mathcal{H}}\quad \forall\, h,g\in\mathcal{H}.
    \end{equation*}
    Then $\mathsf{W}$ has a version with continuous paths such that
    \begin{equation*}
       \beta_j(t)=\mathsf{W}\left(\mathbf{1}_{[0,t]\times\{j\}}\right)\quad\text{for } j\in\N^+,\; t\geq 0.
    \end{equation*}

     The Malliavin derivative $\mathcal{D}$  associated with equation \eqref{AC equation} is therefore defined as, for each $v\in \mathcal{H}$ with $V(\cdot)=\int_0^\cdot v(t)dt$,
    \begin{equation*}
        \mathcal{D}\colon L^2(\Omega;H)\rightarrow L^2(\Omega;\mathcal{H}\otimes H),\quad \langle \mathcal{D}u_t,v\rangle_{ \mathcal{H}}=\lim\limits_{\varepsilon\rightarrow 0}\frac{u(t,u_0,W+\varepsilon V)-u(t,u_0,W)}{\varepsilon},
    \end{equation*}
    where the limit holds almost surely with respect to the Wiener measure.

    \vspace{0.6em}

    In the following notations, let $u_t= u(t,u_0,W)$ be the solution of stochastic Allen--Cahn equation \eqref{AC equation} with initial value $u_0$ and noise $W$.
    
    For any $\xi\in H$ and $s\geq 0$, let $\mathcal{J}_{s,t}\xi$ be the unique solution of the linearized equation, which reads   
    \begin{equation}\label{J equation}
    \begin{cases}
         \partial_t\mathcal{J}_{s,t}\xi-\nu \partial_x^2\mathcal{J}_{s,t}\xi+(3u_t^2-\lambda)\mathcal{J}_{s,t}\xi=0,\quad t>s,\\         
        \mathcal{J}_{s,s}\xi=\xi,
    \end{cases}
    \end{equation}
    equipped with the Dirichlet condition as in \eqref{AC equation}.\footnote{All of the Allen--Cahn equations arising in this paper, which may be positioned in various settings of stochastic problems, linearized systems and controlled systems, are supplemented by the Dirichlet condition, without any explicit mention.} 
    
    Then for any $v\in  \mathcal{H}$, it follows that
    \begin{equation}\label{eq D def}
        \langle \mathcal{D}u_t,v\rangle_{ \mathcal{H}}=\int_0^t\mathcal{J}_{s,t}Bv(s)ds,
    \end{equation}
    where the operator $B$ is defined as
    \begin{equation}\label{eq B def}
        B\colon L^2(a,b)\rightarrow L^2(a,b)\subset H,\quad Bv=\sum_{j\in \N^+}b_j\langle v,\psi_j\rangle\psi_j.
    \end{equation}

    \vspace{0.3em}

    Inspired by \eqref{eq D def},  we define the random operator $\mathcal{A}_{s,t}$ for $0\leq s<t<\infty$ by
    \begin{equation}\label{eq A def}
       \mathcal{A}_{s,t}\colon \mathcal{H}_{s,t}\rightarrow H,\quad \mathcal{A}_{s,t}v=\int_s^t\mathcal{J}_{r,t}Bv(r)dr,
    \end{equation}
    where $\mathcal{H}_{s,t}=L^2(s,t;L^2(a,b))$,  endowed with the inner product inherited from $\mathcal{H}$. 
    
    Noticing that, for any $v\in\mathcal{H}_{s,t}$,  $\mathcal{A}_{s,\cdot}v$ satisfies the following equation
    \begin{equation*}
    \begin{cases}
         \partial_r\mathcal{A}_{s,r}v-\nu \partial_x^2\mathcal{A}_{s,r}v+(3u_r^2-\lambda)\mathcal{A}_{s,r}v=Bv(r),\quad r\in(s,t),\\         
        \mathcal{A}_{s,s}v=0.
    \end{cases}
    \end{equation*}
    The adjoint operator $\mathcal{A}^*_{s,t}$ of $\mathcal{A}_{s,t}$ is given by
    \begin{equation}\label{eq A* def}
        \mathcal{A}^*_{s,t}\colon H\rightarrow \mathcal{H}_{s,t},\quad (\mathcal{A}^*_{s,t}\xi)(r)=B^*\mathcal{J}^*_{r,t}\xi \quad \text{for }\xi\in H,\;r\in[s,t].
    \end{equation}
    Here $B^*$ denotes the adjoint operator of $B$, satisfying
    \begin{equation}\label{eq B* def}
        B^*\colon H\rightarrow L^2(a,b),\quad B^*\xi=\sum_{j\in\N^+}b_j\langle\xi,\psi_j\rangle\psi_j \quad \text{for }\xi\in H.
    \end{equation}
    Meanwhile, the notation $\mathcal{J}_{s,t}^*\xi$ stands for the adjoint system of $\mathcal{J}_{s,t}\xi$, which satisfies: 
    \begin{equation}\label{J* equation}
    \begin{cases}
         \partial_s\mathcal{J}^*_{s,t}\xi+\nu \partial_x^2\mathcal{J}^*_{s,t}\xi-(3u_s^2-\lambda)\mathcal{J}^*_{s,t}\xi=0,\quad s\in(0,t),\\         
        \mathcal{J}^*_{t,t}\xi=\xi.
    \end{cases}
    \end{equation}

    Similarly, for any $N\in\N^+$ and $0\leq s<t$, define the random operator $\mathcal{A}_{s,t,N}=\mathcal{A}_{s,t}\mathsf{P}_N$ as
    \begin{equation*}
       \mathcal{A}_{s,t,N}\colon \mathcal{H}_{s,t}\rightarrow H,\quad \mathcal{A}_{s,t,N}v=\int_s^t\mathcal{J}_{r,t}B\mathsf{P}_Nv(r)dr,
    \end{equation*}
    where recall that $\mathsf{P}_N$ is the projection on $L^2(a,b)$. The associated adjoint operator $\mathcal{A}^*_{s,t,N}$ is given by
    \begin{equation*}
        \mathcal{A}^*_{s,t,N}\colon H\rightarrow \mathcal{H}_{s,t},\quad (\mathcal{A}^*_{s,t,N}\xi)(r)=\mathsf{P}_NB^*\mathcal{J}^*_{r,t}\xi \quad \text{for }\xi\in H,\;r\in[s,t].
    \end{equation*}

    \vspace{0.3em}

    Finally, we define the truncated Malliavin matrix $\mathcal{M}_{s,t,N}$ by
    \begin{equation*}
        \mathcal{M}_{s,t,N}\colon H\rightarrow H,\quad \mathcal{M}_{s,t,N}=\mathcal{A}_{s,t,N}\mathcal{A}^*_{s,t,N}.
    \end{equation*}
    In particular, one has 
    \begin{equation*}
        \mathcal{M}_{s,t,N}=\int_s^t\mathcal{J}_{r,t}B\mathsf{P}_NB^*J^*_{r,t}dr,
    \end{equation*}
    and 
    \begin{equation*}
        \langle \mathcal{M}_{s,t,N}\xi,\xi\rangle=\int_s^t\|(\mathcal{A}^*_{s,t,N}\xi)(r)\|^2dr=\sum_{1\leq j\leq N}b_j^2\int_s^t\langle\mathcal{J}^*_{r,t}\xi,\psi_j\rangle^2dr\quad \forall\,\xi\in H.
    \end{equation*}

    \section{Stabilization analysis and asymptotic strong Feller}\label{Sec 3}

    In this section, we establish the quantitative stabilization properties for the linearized system, summarized in Theorem~\ref{Thm linear-stable}. This result, combined with the Malliavin calculus, will be used to derive the asymptotic strong Feller property, i.e. Proposition \ref{prop ASF}. The latter plays a crucial role in Section~\ref{Sec 5}, where it is employed in the proof of the Main Theorem.
    
    The proof of the stabilization result, along with an outline of the argument, is presented in Section~\ref{Sec 3.1}. The verification of the asymptotic strong Feller property is provided in Section~\ref{Sec 3.2}.

    \vspace{0.6em}

    Let us consider a stabilization property associated with the following linearized system with a potential term, 
    \begin{equation}\label{z equation}
    \begin{cases}
         \partial_rz-\nu \partial_x^2 z+ gz=B\mathsf{P}_N\zeta(r,x),\quad x\in(0,\pi),\;r\in(s,t),\\
        z(s,\cdot)=z_s(\cdot)\in H.
    \end{cases}
    \end{equation}
    Here $0\leq s<t\leq s+1$,  $g\in L^\infty(Q)$ denotes a potential posed on the space-time domain $Q=(s,t)\times (0,\pi)$, and $\zeta=\zeta(r,x)\in \mathcal{H}_{s,t}=L^2(s,t;L^2(a,b))$ is a control. Recall that the operator $B$ and its dual $B^*$ are defined by \eqref{eq B def},\eqref{eq B* def}, and  the notation $\mathsf{P}_N$ stands for the projection in $L^2(a,b)$ onto the finite-dimensional subspace spanned by $\{\psi_j\}_{1\leq j\leq N}$ with $N\in\N^+$.

    \vspace{0.3em}
    The corresponding  adjoint system of \eqref{z equation} is given by
     \begin{equation}\label{y equation}
    \begin{cases}
         \partial_ry+\nu \partial_x^2y-gy=0,\quad x\in(0,\pi),\; r\in(s,t),\\
        y(t,\cdot)=y_t(\cdot)\in H.
    \end{cases}
    \end{equation}

    To continue, we further introduce the following notations. For each $g\in L^\infty(Q)$ and $0\leq s<t$, let $\mathcal{J}_{s,t,g}\colon H\rightarrow H$ be the solution of \eqref{J equation} with the random potential term $3u_t^2-\lambda$ replaced by $g$, and its dual operator is denoted by $\mathcal{J}_{s,t,g}^*$. Using these  operators, the solutions of \eqref{z equation}, \eqref{y equation} have the form of 
    \begin{equation*}
        z(t)=\mathcal{J}_{s,t,g}z_s+\int_s^t\mathcal{J}_{r,t,g}B\mathsf{P}_N\zeta(r)dr,\quad y(s)=\mathcal{J}^*_{s,t,g}y_t.
    \end{equation*}  
   Also define the control map $\mathcal{A}_{s,t,N,g}\in\mathcal{L}(\mathcal{H}_{s,t};H)$ and its dual $\mathcal{A}^*_{s,t,N,g}\in\mathcal{L}(H;\mathcal{H}_{s,t})$ by  
    \begin{equation}\label{eq A-A*def2}
        \mathcal{A}_{s,t,N,g}(\zeta)=\int_s^t\mathcal{J}_{r,t,g}B\mathsf{P}_N\zeta(r)dr,\quad (\mathcal{A}^*_{s,t,N,g}y_t)(r)=\mathsf{P}_NB^*\mathcal{J}^*_{r,t,g}y_t,
    \end{equation}
    and the Gramian operator $\mathcal{G}_{s,t,N,g}$ is thus given by 
    \begin{equation*}
        \mathcal{G}_{s,t,N,g}\colon H\rightarrow H,\quad \mathcal{G}_{s,t,N,g}=\int_s^t\mathcal{J}_{r,t,g}B\mathsf{P}_NB^*\mathcal{J}^*_{r,t,g}dr.
    \end{equation*}

    The main stabilization property is collected in the following.

    \begin{theorem}\label{Thm linear-stable} {\rm(}Deterministic stabilization along trajectory{\rm)}
    For the equation \eqref{z equation}, there exists a constant $C>0$  such that the following holds. For any $0\leq s<t\leq s+1$, $\beta>0$, $N\in\N^+$, $z_s\in H$ and $g\in L^\infty(Q)$, the solution $z$ of \eqref{z equation} with the control $\zeta\in \mathcal{H}_{s,t}$ given by 
    \begin{equation}\label{eq zeta def}
        \zeta(r)=-\mathsf{P}_N\left(B^*\mathcal{J}_{r,t,g}^*(\mathcal{G}_{s,t,N,g}+\beta)^{-1}\mathcal{J}_{s,t,g}z_s\right),\quad r\in(s,t)
    \end{equation}
    satisfies that
    \begin{equation}\label{eq z sqz1}
        \|z(t)\|\leq  C\exp\left(C((t-s)\|g\|_{L^\infty(Q)}+(t-s)^{-2})\right)\left(\hat{b}_N\beta^{1/2}+N^{-1/2}\right)\|z_s\|,
    \end{equation}
    and 
    \begin{equation}\label{eq z sqz2}
        \|\zeta\|_{\mathcal{H}}\leq  C\exp\left(C((t-s)\|g\|_{L^\infty(Q)}+(t-s)^{-2})\right)\left(\hat{b}_N+\beta^{-1/2}N^{-1/2}\right)\|z_s\|,
    \end{equation}
    where $\hat{b}_N:=\max\{|b_j|^{-1}:1\leq j\leq N\}$.
    \end{theorem}   

    \vspace{0.3em}

   Next we establish the asymptotic strong Feller property by using the Malliavin calculus.
    
    \begin{proposition}\label{prop ASF}{\rm(}Asymptotic strong Feller property{\rm)}
     For the equation given by \eqref{AC equation}-\eqref{eq psi-def} and any $B_0,\gamma,\alpha>0$, there exists a constant $N\in\N^+$ such that if the sequence $\{b_j\}_{j\in\N^+}$ satisfies 
    \begin{equation*}
       \sum_{j\in\N^+} j^2b_j^2\leq B_0\quad \text{and}\quad b_j\neq 0\quad \text{for }1\leq j\leq N,
    \end{equation*}
    then there is a constant $C>0$  such that for any Fréchet differentiable function $\varphi\colon H\rightarrow \R$, $u_0\in H$ and $t\geq 0$,
        \begin{equation}\label{eq ASF}
            \|\nabla P_t\varphi(u_0)\|\leq C\exp\left(\gamma\|u_0\|^2\right)\left(\sqrt{P_t(|\varphi|^2)(u_0)}+e^{-\alpha t}\sqrt{P_t(\|\nabla\varphi\|^2)(u_0)})\right).
        \end{equation}

    \end{proposition}

    \subsection{Stabilization for the linearized equation}\label{Sec 3.1}

    In this subsection, we present the proof of Theorem \ref{Thm linear-stable}. The strategy builds upon the frequency analysis applied in our previous work \cite[Proposition 5.1]{LWXZZ-24}, which is inspired by \cite{Xiang-23, Xiang-24}. Specifically, our approach relies on two ingredients:
    
    \begin{itemize}
        \item[(\runum{1})] A weak observability inequality based on Carleman estimates, whose quantitative estimates are crucial for verifying the asymptotic strong Feller property;

        \item[(\runum{2})] The Fenchel--Rockafellar duality, which establishes the duality between the stabilization property and the weak observability inequality.

    \end{itemize}

    \vspace{0.3em}
    
    The proof is structured into three steps, as sketched below. See Sections \ref{Sec 3.1.1}-\ref{Sec 3.1.3} for details.

    \noindent {\it Step 1: Quantitative weak observability inequality.} We first establish a weak observability inequality by invoking the result of \cite{FCZ-00}. Specifically, there exists a constant $C>0$ such that for any  $0\leq s<t\leq s+1$, $N\in\N^+$, $y_t\in H$ and $g\in L^\infty(Q)$, the following estimate holds:
    \begin{equation}\label{eq outa}
             \|y(s)\|\leq 
            C\exp\left(C((t-s)\|g\|_{L^\infty(Q)}+(t-s)^{-2})\right)\left(\|\mathsf{P}_N(\mathbf{1}_{(a,b)}y)\|_{\mathcal{H}}+N^{-1/2}\|y_t\|\right).
       \end{equation}
       
    Compared to the standard observability inequality used in the context of null controllability, inequality \eqref{eq outa} represents a weaker form due to the presence of the remainder term $C\exp\left(C((t-s)\|g\|_{L^\infty(Q)}+(t-s)^{-2})\right)N^{-1/2}\|y_t\|$. Indeed, the notion of weak observability, can be used to establish  weak forms of null controllability; see e.g. \cite{AM-22,TWX-20}.

    \vspace{0.6em}
    
    \noindent {\it Step 2: Duality for  stabilization and  weak observability.} 
    Next, we establish the equivalence of the following two properties: let  $C,\varepsilon\geq0$ be arbitrarily given,
    \begin{itemize}[leftmargin=2em]
    \item[\tiny$\bullet$] Stabilization of \eqref{z equation}: For any $z_s\in H$, there exists a control $\zeta\in \mathcal{H}_{s,t}$ such that
    \begin{equation}\label{eq outb}
            \|z(t)\|\leq \varepsilon\|z_s\|,\quad \|\zeta\|_{\mathcal{H}}\leq C\|z_s\|
    \end{equation}

    \item[\tiny$\bullet$] Observability of \eqref{y equation}: For any $y_t\in H$,
    \begin{equation}\label{eq outc}
            \|y(s)\|\leq C\|\mathsf{P}_NB^*y\|+\varepsilon\|y_t\|.
    \end{equation}
    \end{itemize}
    The duality between controllability and observability is well-known in control theory; see e.g. Coron \cite{Coron-07}. In our circumstance, the relation \eqref{eq outb} corresponds precisely to the notion of $\alpha$-null controllability considered in \cite{TWX-20}.

    Following \cite{Lions-92,TWX-20}, the proof of the equivalence is based on the Fenchel--Rockafellar duality. More precisely, the issue of stabilization  \eqref{eq outb} is formulated as a convex minimization problem:
    \begin{equation}\label{eq outd}
        \inf_{\zeta\in \mathcal{H}_{s,t}}J(\zeta)\leq\tfrac{1}{2}C^2\|z_s\|^2.
    \end{equation}
  Thus \eqref{eq outd} is equivalent to the condition that its dual problem satisfies:  
    \begin{equation}\label{eq oute}
        \sup_{w\in H}J^*(w)\geq\tfrac{1}{2}C^2\|z_s\|^2.
    \end{equation}
    Finally, we show that \eqref{eq oute} holds if and only if the weak observability inequality \eqref{eq outc} is satisfied, thereby establishing the desired duality.

    \vspace{0.6em}

    \noindent {\it Step 3: Regularized stabilization via Fenchel--Rockafellar duality.}   Although the last step establishes the duality for stabilization, the corresponding control  constructed in this way does not admit an explicit expression and may exhibit singular behavior; see Remark \ref{Rem 3.7}. Therefore, to derive Theorem~\ref{Thm linear-stable}, we provide an alternative construction of a {\it regularized control} in the form of \eqref{eq zeta def}.

     More precisely, the desired relations  \eqref{eq z sqz1},\eqref{eq z sqz2} are obtained by solving the following convex minimization problem:
    \begin{equation}\label{eq outf}
        \inf_{\zeta\in \mathcal{H}_{s,t}}J(\zeta)=\inf_{\zeta\in \mathcal{H}_{s,t}}\left(\frac{1}{2}\|\zeta\|_{\mathcal{H}}^2+\frac{1}{2\beta}\|\mathcal{A}_{s,t,N,g}\zeta+\mathcal{J}_{s,t,g}z_s\|^2\right).
    \end{equation}
   By the Fenchel--Rockafellar theorem, the problem \eqref{eq outf} is equivalent to the following dual maximization
    \begin{equation}\label{eq outg}
        \sup_{w\in H}J^*(w)= \sup_{w\in H}\left(-\frac{1}{2}\langle(\mathcal{G}_{s,t,N,g}+\beta)w,w\rangle+\langle w,\mathcal{J}_{s,t,g}z_s\rangle\right).
    \end{equation}
    Using the weak observability inequality \eqref{eq outa}, we show that the dual problem \eqref{eq outg} leads to the stabilization estimates \eqref{eq z sqz1},\eqref{eq z sqz2}. The structure of the proof is illustrated in Figure \ref{figure 3}.

    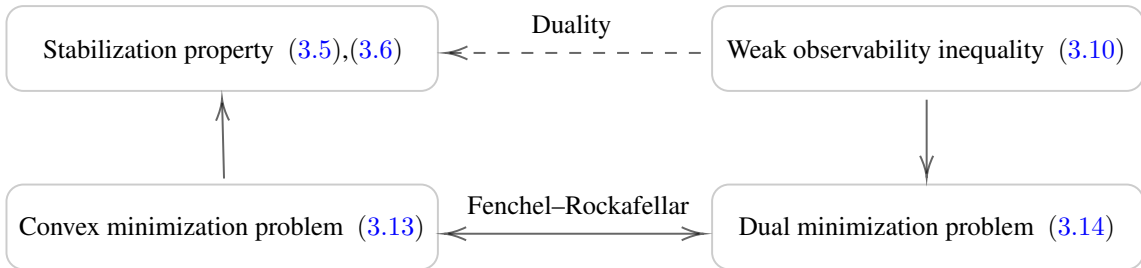
\begin{figure}[th]
    \centering
    \tikzset{every picture/.style={line width=0.75pt}} 

\begin{tikzpicture}[x=0.75pt,y=0.75pt,yscale=-1,xscale=1]

\draw  [color={rgb, 255:red, 0; green, 0; blue, 0 }  ,draw opacity=0.2 ] (7,14.01) .. controls (7,9.31) and (10.81,5.5) .. (15.51,5.5) -- (214.92,5.5) .. controls (219.62,5.5) and (223.43,9.31) .. (223.43,14.01) -- (223.43,39.54) .. controls (223.43,44.24) and (219.62,48.05) .. (214.92,48.05) -- (15.51,48.05) .. controls (10.81,48.05) and (7,44.24) .. (7,39.54) -- cycle ;
\draw [color={rgb, 255:red, 0; green, 0; blue, 0 }  ,draw opacity=0.6 ] [dash pattern={on 4.5pt off 4.5pt}]  (229.33,29) -- (357.33,29) ;
\draw [shift={(227.33,29)}, rotate = 0] [color={rgb, 255:red, 0; green, 0; blue, 0 }  ,draw opacity=0.6 ][line width=0.75]    (10.93,-3.29) .. controls (6.95,-1.4) and (3.31,-0.3) .. (0,0) .. controls (3.31,0.3) and (6.95,1.4) .. (10.93,3.29)   ;
\draw [color={rgb, 255:red, 0; green, 0; blue, 0 }  ,draw opacity=0.6 ][line width=0.75]    (115.19,54) -- (116.14,92) ;
\draw [shift={(115.14,52)}, rotate = 88.57] [color={rgb, 255:red, 0; green, 0; blue, 0 }  ,draw opacity=0.6 ][line width=0.75]    (10.93,-3.29) .. controls (6.95,-1.4) and (3.31,-0.3) .. (0,0) .. controls (3.31,0.3) and (6.95,1.4) .. (10.93,3.29)   ;
\draw [color={rgb, 255:red, 0; green, 0; blue, 0 }  ,draw opacity=0.6 ]   (228.67,119.7) -- (354.67,119.7) ;
\draw [shift={(356.67,119.7)}, rotate = 180] [color={rgb, 255:red, 0; green, 0; blue, 0 }  ,draw opacity=0.6 ][line width=0.75]    (10.93,-3.29) .. controls (6.95,-1.4) and (3.31,-0.3) .. (0,0) .. controls (3.31,0.3) and (6.95,1.4) .. (10.93,3.29)   ;
\draw [shift={(226.67,119.7)}, rotate = 0] [color={rgb, 255:red, 0; green, 0; blue, 0 }  ,draw opacity=0.6 ][line width=0.75]    (10.93,-3.29) .. controls (6.95,-1.4) and (3.31,-0.3) .. (0,0) .. controls (3.31,0.3) and (6.95,1.4) .. (10.93,3.29)   ;

\draw  [color={rgb, 255:red, 0; green, 0; blue, 0 }  ,draw opacity=0.2 ] (7,103.54) .. controls (7,98.84) and (10.81,95.03) .. (15.51,95.03) -- (214.92,95.03) .. controls (219.62,95.03) and (223.43,98.84) .. (223.43,103.54) -- (223.43,129.07) .. controls (223.43,133.77) and (219.62,137.58) .. (214.92,137.58) -- (15.51,137.58) .. controls (10.81,137.58) and (7,133.77) .. (7,129.07) -- cycle ;

\draw [color={rgb, 255:red, 0; green, 0; blue, 0 }  ,draw opacity=0.6 ]   (468.64,52.46) -- (468.64,89.54) ;
\draw [shift={(468.64,91.54)}, rotate = 270] [color={rgb, 255:red, 0; green, 0; blue, 0 }  ,draw opacity=0.6 ][line width=0.75]    (10.93,-3.29) .. controls (6.95,-1.4) and (3.31,-0.3) .. (0,0) .. controls (3.31,0.3) and (6.95,1.4) .. (10.93,3.29)   ;
\draw  [color={rgb, 255:red, 0; green, 0; blue, 0 }  ,draw opacity=0.2 ] (361,103.54) .. controls (361,98.84) and (364.81,95.03) .. (369.51,95.03) -- (568.92,95.03) .. controls (573.62,95.03) and (577.43,98.84) .. (577.43,103.54) -- (577.43,129.07) .. controls (577.43,133.77) and (573.62,137.58) .. (568.92,137.58) -- (369.51,137.58) .. controls (364.81,137.58) and (361,133.77) .. (361,129.07) -- cycle ;
\draw  [color={rgb, 255:red, 0; green, 0; blue, 0 }  ,draw opacity=0.2 ] (361,14.01) .. controls (361,9.31) and (364.81,5.5) .. (369.51,5.5) -- (568.92,5.5) .. controls (573.62,5.5) and (577.43,9.31) .. (577.43,14.01) -- (577.43,39.54) .. controls (577.43,44.24) and (573.62,48.05) .. (568.92,48.05) -- (369.51,48.05) .. controls (364.81,48.05) and (361,44.24) .. (361,39.54) -- cycle ;

\draw (3,111.11) node [anchor=north west][inner sep=0.75pt]   [align=left] { \ \ };
\draw (236.86,99) node [anchor=north west][inner sep=0.75pt]  [font=\small] [align=left] {{\fontfamily{ptm}\selectfont Fenchel--Rockafellar}};
\draw (3,109) node [anchor=north west][inner sep=0.75pt]  [font=\small] [align=left] { \ \ {\fontfamily{ptm}\selectfont Convex minimization problem \ \eqref{eq outf}}};
\draw (358,20) node [anchor=north west][inner sep=0.75pt]  [font=\small] [align=left] { \ \ {\fontfamily{ptm}\selectfont Weak observability inequality \ \eqref{eq outc}}};
\draw (15,20) node [anchor=north west][inner sep=0.75pt]  [font=\small] [align=left] { \ \ {\fontfamily{ptm}\selectfont Stabilization property \ \eqref{eq z sqz1},\eqref{eq z sqz2}}};
\draw (269.07,8.63) node [anchor=north west][inner sep=0.75pt]  [font=\small] [align=left] {{\fontfamily{ptm}\selectfont Duality}};
\draw (364,109) node [anchor=north west][inner sep=0.75pt]  [font=\small] [align=left] { \ \ {\fontfamily{ptm}\selectfont Dual minimization problem \ \eqref{eq outg}}};

\end{tikzpicture}

    \vspace{-0.6em}
    \caption{Proof of Theorem \ref{Thm linear-stable}.}\label{figure 3}
    \end{figure}

    \subsubsection{Step 1: Quantitative  weak observability inequality}\label{Sec 3.1.1} In this section, we establish Proposition \ref{prop OI}. We first recall the following observability inequality.     

    \begin{lemma}\label{lemma OI}{\rm(\hspace{-0.1mm}}{\rm\cite[Theorem 1.2]{FCZ-00})} For any $0\leq a<b\leq \pi$, there exists a constant $C>0$ such that for any $y_t\in H$  and $g\in L^\infty(Q)$, the solution $y$ of \eqref{y equation} satisfies that
        \begin{align}\label{eq OI}
            \|y(s)\|^2\leq 
            C\exp\left(C((t-s)\|g\|_{L^\infty(Q)}+(t-s)^{-2})\right)\int_s^t\int_a^b|y(r,x)|^2dxdr.
        \end{align}
    \end{lemma}

    Note that the proof of Lemma \ref{lemma OI} relies on global Carleman inequalities; see e.g. \cite{FCZ-00,FLZ-19,DZZ-08}, in particular \cite{DZZ-08} provided the sharp estimates. For clarity, we  give a brief explanation of the role of observability inequalities and Carleman estimates. Observability inequalities quantify how the total energy of a solution can be controlled by measurements on a subdomain, and their duality with controllability makes them central in control theory. Carleman estimates provide one crucial analytic tool for establishing such inequalities: weighted a priori estimates with carefully designed exponential weights, which align with the structure of the underlying operator.  

    Specifically, the Carleman estimate in \cite[Proposition 2.1]{FCZ-00} for the proof of \eqref{eq OI} is formulated as follows. We introduce a function $\xi\in C^2([0,\pi])\cap H_0^1(0,\pi)$ such that 
    \begin{equation*}
        \xi(x)>0\quad \text{for }x\in(0,\pi),\quad \xi'(x)\neq0   \quad \text{for }x\in[0,a]\cup[b,\pi].
    \end{equation*}
    Let $K>0$ be such that $K\geq 5\max_{x\in [0,\pi]} \xi-6\min_{x\in [0,\pi]} \xi$ and set 
    \begin{equation*}
        \beta(x)=\xi(x)+K,\quad \bar{\beta}=\frac{5}{4}\max_{x\in [0,\pi]} \beta,\quad \rho(x)=\exp\left(L\bar\beta(x)\right)-\exp\left(L\beta(x)\right),
    \end{equation*}
    where $L>0$ is a sufficiently large positive constant that only depends on $a,b$. For any $T>0$, we also introduce 
    \begin{equation*}
        \varphi(x,t)=\exp\left(t^{-1}(T-t)^{-1}\rho(x)\right).
    \end{equation*}

    \begin{lemma}\label{lemma Carleman}{\rm(\hspace{-0.1mm}}{\rm\cite[Proposition 2.1]{FCZ-00})} For any $0\leq a<b\leq \pi$, there exist constants $C,c>0$ such that for any $T>0$, $q\in C^2([0,T]\times[0,\pi])\cap L^2(0,T;H^1_0(0,\pi))$ and $s\geq c(T+T^2)$,
        \begin{align*}
            &\int_0^T\int_0^\pi s^{-1}\varphi^{-2s}t(T-t)(|\partial_tq|^2+|\partial_{x}^2q|^2)dxdt\\
            &\quad+\int_0^T\int_0^\pi s\varphi^{-2s}t^{-1}(T-t)^{-1}|\partial_xq|^2dxdt+\int_0^T\int_0^\pi s^3\varphi^{-2s}t^{-3}(T-t)^{-3}|q|^2dxdt\\
            &\leq C\left(\int_0^T\int_0^\pi \varphi^{-2s}|\partial_tq+\partial_{x}^2q|^2dxdt+\int_0^T\int_a^b s^3\varphi^{-2s}t^{-3}(T-t)^{-3}|q|^2dxdt\right).
        \end{align*}
    \end{lemma}

    \vspace{0.3em}

    Using Lemma \ref{lemma OI}, we derive the following {\it truncated }observability inequality for \eqref{y equation}.

     \begin{proposition}\label{prop OI} For any $0\leq a<b\leq \pi$, there exists a constant $C>0$ such that for any $0\leq s<t\leq s+1$, $N\in\N^+$, $y_t\in H$  and $g\in L^\infty(Q)$, the solution $y$ of \eqref{y equation} satisfies that
        \begin{align}\label{eq TOI}
            \|y(s)\|\leq 
            C\exp\left(C((t-s)\|g\|_{L^\infty(Q)}+(t-s)^{-2})\right)\left(\|\mathsf{P}_N(\mathbf{1}_{(a,b)}y)\|_{\mathcal{H}}+N^{-1/2}\|y_t\|\right).
        \end{align}
    \end{proposition}

    For ease of notation, we in this subsection use $C\geq 1$  to denote generic constants that may vary from line to line, which are independent of $s$, $t$, $N$, $y$.

    \begin{proof} By Lemma \ref{lemma OI}, we have
    \begin{align*}
        \|y(s)\|^2&\leq  C\exp\left(C((t-s)\|g\|_{L^\infty(Q)}+(t-s)^{-2})\right)\|\mathbf{1}_{(a,b)}y\|^2_{\mathcal{H}}\\
        &= C\exp\left(C((t-s)\|g\|_{L^\infty(Q)}+(t-s)^{-2})\right)(\left\|\mathsf{P}_N(\mathbf{1}_{(a,b)}y)\|^2_{\mathcal{H}}+\|(I-\mathsf{P}_N)(\mathbf{1}_{(a,b)}y)\|^2_{\mathcal{H}}\right)\\
        &\leq C\exp\left(C((t-s)\|g\|_{L^\infty(Q)}+(t-s)^{-2})\right)\left(\|\mathsf{P}_N(\mathbf{1}_{(a,b)}y)\|^2_{\mathcal{H}}+\sum_{j>N}\int_s^t\langle y(r),\psi_j\rangle^2dr\right)\\
        &\leq C\exp\left(C((t-s)\|g\|_{L^\infty(Q)}+(t-s)^{-2})\right)\left(\|\mathsf{P}_N(\mathbf{1}_{(a,b)}y)\|^2_{\mathcal{H}}+N^{-1}\| y\|^2_{L^2(s,t;H^1(a,b))}\right)\\
        &\leq  C\exp\left(C((t-s)\|g\|_{L^\infty(Q)}+(t-s)^{-2})\right)\left(\|\mathsf{P}_N(\mathbf{1}_{(a,b)}y)\|^2_{\mathcal{H}}+N^{-1}\|y\|^2_{L^2(s,t;H^1(0,\pi))}\right).
    \end{align*}
    Here in the third to fourth line, we used
    \begin{equation*}
        |\langle f,\psi_j\rangle|\leq Cj^{-1}\|f\|_{H^1(a,b)}\quad \forall\, f\in H_0^1(0,\pi),\; j\in\N^+.
    \end{equation*}
   We also note that the $N^{-1}$ rate is sharp in general,  due to the possible nonzero traces at $a$, $b$ after restricting an $H_0^1(0,1)$ function to $(a,b)$.

   Consequently,  the desired inequality \eqref{eq TOI} follows from 
    \begin{equation*}
       \|y\|^2_{L^2(s,t;H^1)}\leq  C\exp\left(C(t-s)\|g\|_{L^\infty(Q)}\right)\|y_t\|^2,
    \end{equation*}
    which can be calculated directly.
        
    \end{proof}

     \subsubsection{Step 2: Duality for stabilization and weak observability}\label{Sec 3.1.2}

    \begin{proposition}\label{prop dual}
    For any $0\leq a<b\leq \pi$, $0\leq s<t$, $\varepsilon,C\geq0$, $N\in\N^+$ and $g\in L^\infty(Q)$, the following two statements are equivalent.

    \begin{itemize}       
        \item [(1)] For any $z_s\in H$, there exists a control $\zeta\in \mathcal{H}_{s,t}$ such that the solution $z$ of \eqref{z equation} satisfies 
        \begin{equation*}
            \|z(t)\|\leq  \varepsilon\|z_s\|,\quad \|\zeta\|_{\mathcal{H}}\leq C\|z_s\|.
        \end{equation*}
        \item [(2)] For any $y_t\in H$, the solution $y$ of \eqref{y equation} satisfies that
       \begin{equation*}
           \|y(s)\|\leq C\|\mathsf{P}_NB^* y\|_{\mathcal{H}}+\varepsilon\|y_t\|.
       \end{equation*}
    \end{itemize}  
    Moreover, when $(2)$ holds, we provide an explicit construction for the control $\zeta$.
    \end{proposition}

    The proof relies on the classical Fenchel--Rockafellar duality, which follows from the analysis in \cite{Lions-92,TWX-20}. We mention that when $\varepsilon=0$, the controllability considered here coincides with the usual null controllability.

    \vspace{0.3em}
    
    \begin{proof}
    Following \cite{Lions-92}, we define two  convex and lower semicontinuous functions by 
    \begin{equation*}
        F\colon \mathcal{H}_{s,t}\rightarrow\R^+,\quad F(\zeta)=\frac{1}{2}\|\zeta\|^2_{\mathcal{H}},
    \end{equation*}
    and 
    \begin{equation*}
       G\colon H\rightarrow\R^+\cup\{+\infty\},\quad G(w)=\begin{cases}
           0, &\quad \text{for }\|w+\mathcal{J}_{s,t,g}z_s\|\leq \varepsilon\|z_s\|,\\
           +\infty, &\quad \text{otherwise}.
       \end{cases}
    \end{equation*}
    
    We then consider the following minimization problem
    \begin{equation}\label{eq problem1}
        \inf_{\zeta\in \mathcal{H}_{s,t}}J(\zeta)=\inf_{\zeta\in \mathcal{H}_{s,t}}(F(\zeta)+G(\mathcal{A}_{s,t,N,g}\zeta)).
    \end{equation}
    Noting that $J$ is strictly convex, thus problem \eqref{eq problem1} admits at most one solution. In addition, for any $\zeta\in \mathcal{H}_{s,t}$,
    \begin{equation*}
        J(\zeta)<+\infty\quad \text{if and only if}\quad \|z(t)\|\leq \varepsilon\|z_s\|,
    \end{equation*}
    in this case,
    \begin{equation*}
        J(\zeta)=F(\zeta)=\frac{1}{2}\|\zeta\|^2_{\mathcal{H}}.
    \end{equation*}

    Therefore, by our construction, statement (1) is equivalent to
    \begin{equation}\label{eq J equivalence}
        \inf_{\zeta\in \mathcal{H}_{s,t}}J(\zeta)\leq \frac{1}{2}C^2\|z_s\|^2,\quad \forall\, z_s\in H.
    \end{equation}

    \vspace{0.6em}

    Meanwhile, the Fenchel conjugates $F^*\colon \mathcal{H}_{s,t}\rightarrow\R^+$ and $G^*\colon H\rightarrow\R^+\cup\{+\infty\}$ are given by
    \begin{equation*}
        F^*(\zeta)=\sup_{\zeta'\in \mathcal{H}_{s,t}}\left(\langle \zeta,\zeta'\rangle_{\mathcal{H}}-F(\zeta')\right)=\frac{1}{2}\|\zeta\|^2_{\mathcal{H}}=F(\zeta),
    \end{equation*}
    and 
    \begin{align*}
        G^*(w)&=\sup_{w'\in H}\left(\langle w,w'\rangle-G(w')\right)=\sup\left\{\langle w,w'\rangle:\|w'+\mathcal{J}_{s,t,g}z_s\|\leq \varepsilon\|z_s\|\right\}\\
        &=-\langle w,\mathcal{J}_{s,t,g}z_s\rangle+\varepsilon\|z_s\|\|w\|.
    \end{align*}

    The corresponding dual problem is then given by
    \begin{equation}\label{eq problem1*}
    \begin{aligned}
        \sup_{w\in H}J^*(w)&=\sup_{w\in H}\left(-G^*(-w)-F^*(\mathcal{A}_{s,t,N,g}^*w)\right)=\sup_{w'\in H}\left(-G^*(w')-F^*(-\mathcal{A}_{s,t,N,g}^*w')\right)\\
        &=-\inf_{w\in H}\hat{J}^*(w),
    \end{aligned}
    \end{equation}
    where 
    \begin{equation*}
        \hat{J}^*\colon H\rightarrow\R^+,\quad \hat{J}^*(w):=\tfrac{1}{2}\langle\mathcal{G}_{s,t,N,g}w,w\rangle-\langle w,\mathcal{J}_{s,t,g}z_s\rangle+\varepsilon\|z_s\|\|w\|.
    \end{equation*}

    \vspace{0.6em}
    
    In the case of $\varepsilon>0$, when problem \eqref{eq problem1} admits a solution, the range of $\mathcal{A}_{s,t,N,g}$ intersects the set of points at which $G$ is finite and continuous. Then, invoking Proposition \ref{prop F-R duality}, both problems \eqref{eq problem1} and \eqref{eq problem1*} have a solution. In particular, 
    \begin{equation}\label{eq dual relation}
        \inf_{\zeta\in \mathcal{H}_{s,t}}J(\zeta)=\sup_{w\in H}J^*(w)=-\inf_{w\in H}\hat{J}^*(w).
    \end{equation}

    Moreover, this above relation result is still valid for $\varepsilon=0$, which is, however,  not a consequence of the Fenchel--Rockafellar duality theorem. For $\varepsilon=0$, this is due to  the usual Hilbert uniqueness method; see e.g. \cite{Lions-88}.

    Consequently, for any $\varepsilon\geq 0$, statement (1) holds if and only if 
    \begin{equation}\label{eq J equivalence2}
        \inf_{w\in H}\hat{J}^*(w)\geq -\frac{1}{2}C^2\|z_s\|^2,\quad \forall\, z_s\in H.
    \end{equation}

    \vspace{0.3em}
    
    In what follows, we show that condition \eqref{eq J equivalence2} is equivalent to statement (2). 
    For any fixed $w, z_s\in H$, let us consider
    \begin{equation*}
        \hat{J}^*(lw)=\tfrac{1}{2}\langle\mathcal{G}_{s,t,N,g}w,w\rangle l^2-\left(\langle w,\mathcal{J}_{s,t,g}z_s\rangle-\varepsilon\|z_s\|\|w\|\right)l,\quad l\in\R^+.
    \end{equation*}
    By standard calculations, one has
    \begin{equation*}
        \inf_{l\in \R^+}\hat{J}^*(lw)=\begin{cases}              0,&\quad \text{ if }\langle w,\mathcal{J}_{s,t,g}z_s\rangle-\varepsilon\|z_s\|\|w\|\leq 0,\\
            -\frac{1}{2}\frac{(\langle w,\mathcal{J}_{s,t,g}z_s\rangle-\varepsilon\|z_s\|\|w\|)^2}{\langle\mathcal{G}_{s,t,N,g}w,w\rangle},&\quad \text{ if } \langle\mathcal{G}_{s,t,N,g}w,w\rangle\neq 0,\langle w,\mathcal{J}_{s,t,g}z_s\rangle-\varepsilon\|z_s\|\|w\|>0,\\
            -\infty,&\quad \text{ if } \langle\mathcal{G}_{s,t,N,g}w,w\rangle=0, \langle w,\mathcal{J}_{s,t,g}z_s\rangle-\varepsilon\|z_s\|\|w\|>0.
        \end{cases}
    \end{equation*}

    Note that condition \eqref{eq J equivalence2} is equivalent to
    \begin{equation*}
        \inf_{w\in H}\inf_{l\in \R^+}\hat{J}^*(lw)\geq -\frac{1}{2}C^2\|z_s\|^2,\quad \forall\, z_s\in H.
    \end{equation*}
    Therefore, we derive that \eqref{eq J equivalence2} holds if and only if, for any $z_s\in H$ and $w\in H$,  when $C>0$,
    \begin{equation*}
    \begin{cases}
        -\frac{1}{2}\frac{(\langle w,\mathcal{J}_{s,t,g}v_0\rangle-\varepsilon\|z_s\|\|w\|)^2}{\langle\mathcal{G}_{s,t,N,g}w,w\rangle}\geq -\frac{1}{2}C^2\|z_s\|^2,&\quad \text{ if } \langle\mathcal{G}_{s,t,N,g}w,w\rangle\neq 0,\langle w,\mathcal{J}_{s,t,g}z_s\rangle-\varepsilon\|z_s\|\|w\|>0,\\
        \langle w,\mathcal{J}_{s,t,g}z_s\rangle-\varepsilon\|z_s\|\|w\|\leq 0,&\quad \text{ if } \langle\mathcal{G}_{s,t,N,g}w,w\rangle=0,
    \end{cases}        
    \end{equation*}
    and when $C=0$,
    \begin{equation*}
        \langle w,\mathcal{J}_{s,t,g}z_s\rangle-\varepsilon\|z_s\|\|w\|\leq0.
    \end{equation*}
    This is further equivalent to 
    \begin{equation*}
        \langle w,\mathcal{J}_{s,t,g}z_s\rangle\leq C\|z_s\|\langle \mathcal{G}_{s,t,N,g}w,w\rangle^{1/2}+\varepsilon \|z_s\|\|w\|,\quad \forall\, z_s,w\in H.
    \end{equation*}

    Finally, since the above relation holds for any $z_s\in H$, it implies the desired equivalence between statement (2) and \eqref{eq J equivalence2}. The proof is thus completed.
        
    \end{proof}

    \begin{remark}\label{Rem 3.7}
        Let us mention that when statement (2) is satisfied, the desired control $\zeta$ in statement (1) can be constructed explicitly.  Indeed, let $\zeta$ be the unique minimizer $\zeta^*$ of problem \eqref{eq problem1}. Note that 
            \begin{equation*}
            \inf_{\zeta\in \mathcal{H}_{s,t}}J(\zeta)=0\quad\text{if and only if }\quad\|\mathcal{}_{s,t,g}z_s\|\leq \varepsilon\|z_s\|,
        \end{equation*}
        and $\zeta=\zeta^*=0$ in this case. Taking the dual relation \eqref{eq dual relation} into account, for any minimizer $w^*\in H$ of $\hat{J}^*$, $w^*=0$ if and only if $\|\mathcal{J}_{s,t,g}z_s\|\leq \varepsilon\|z_s\|$. Meanwhile, as $\hat{J^*}$ is differentiable except at $w=0$, when $\|\mathcal{J}_{s,t,g}z_s\|>\varepsilon\|z_s\|$, i.e. $w^*\neq 0$, one has
        \begin{equation}\label{eq w*}
            \mathcal{G}_{s,t,N,g}w^*-\mathcal{J}_{s,t,g}z_s+\varepsilon\|z_s\|\frac{w^*}{\|w^*\|}=0.
        \end{equation}
        Thus in both cases, problem \eqref{eq problem1*} admits a unique solution $w^*$. Consequently, in view of the Fenchel duality argument, the optimal control $\zeta^*$ is given as a function of $w^*$ by
        \begin{equation*}
            \zeta^*(r)=\mathsf{P}_N\left(B^*\mathcal{J}_{s,t,g}^*w^*\right),\quad r\in(s,t).
        \end{equation*}
    In particular, $w^*$ given by \eqref{eq w*} does not admit an explicit expression. Moreover, the mapping $g\mapsto\zeta^*$ may lack sufficient regularity and exhibit singular behavior, which makes it difficult to apply in the stochastic setting.
    \end{remark}

    \subsubsection{Step 3: Regularized stabilization via Fenchel--Rockafellar duality}\label{Sec 3.1.3} 
     We combine the Fenchel--Rockafellar duality with the weak observability inequality to establish the regularized  stabilization result, and construct the desired control given in the form of \eqref{eq zeta def}.
 
    \begin{proof}[Proof of Theorem \ref{Thm linear-stable}]
    Analogous to the proof of Proposition \ref{prop dual}, let us define two  convex and continuous functions by 
    \begin{equation*}
        F\colon \mathcal{H}_{s,t}\rightarrow\R^+,\quad F(\zeta)=\frac{1}{2}\|\zeta\|^2_{\mathcal{H}},
    \end{equation*}
    and 
    \begin{equation*}
       G\colon H\rightarrow\R^+,\quad G(w)=\frac{1}{2\beta}\|w+\mathcal{J}_{s,t,g}z_s\|^2.
    \end{equation*}

    We then consider the following minimization problem
    \begin{equation}\label{eq problem2}
        \inf_{\zeta\in \mathcal{H}_{s,t}}J(\zeta)=\inf_{\zeta\in \mathcal{H}_{s,t}}\left(F(\zeta)+G(\mathcal{A}_{s,t,N,g}\zeta)\right).
    \end{equation}

    Note that $J$ is strictly convex, thus problem \eqref{eq problem2} admits at most one solution. The Fenchel conjugates $F^*\colon \mathcal{H}_{s,t}\rightarrow\R^+$ and $G^*\colon H\rightarrow\R^+$ are given by
    \begin{equation*}
        F^*(\zeta)=\sup_{\zeta'\in \mathcal{H}_{s,t}}\left(\langle \zeta,\zeta'\rangle_{\mathcal{H}}-F(\zeta')\right)=\frac{1}{2}\|\zeta\|^2_{\mathcal{H}}=F(\zeta),
    \end{equation*}
    and 
    \begin{align*}
        G^*(w)=\sup_{w'\in H}\left(\langle w,w'\rangle-G(w')\right)=
            \frac{\beta}{2}\|w\|^2-\langle w,\mathcal{J}_{s,t,g}z_s\rangle.
    \end{align*}

    The corresponding dual problem is then given by
    \begin{equation}\label{eq problem2*}
        \sup_{w\in H}J^*(w)=\sup_{w\in H}\left(-G^*(-w)-F^*(\mathcal{A}_{s,t,N,g}^*w)\right)=-\inf_{w\in H}\hat{J}^*(w),
    \end{equation}
    where 
    \begin{equation*}
        \hat{J}^*\colon H\rightarrow\R^+,\quad \hat{J}^*(w):=\tfrac{1}{2}\langle(\mathcal{G}_{s,t,N,g}+\beta)w,w\rangle+\langle w,\mathcal{J}_{s,t,g}z_s\rangle.
    \end{equation*}

    \vspace{0.6em}
    
   Clearly, $F$ and $G$ are finite and continuous. Thus, by the Fenchel--Rockafellar duality, both problems \eqref{eq problem2} and \eqref{eq problem2*} admit a solution. In particular, 
    \begin{equation*}
        \inf_{\zeta\in \mathcal{H}_{s,t}}J(\zeta)=\sup_{w\in H}J^*(w)=-\inf_{w\in H}\hat{J}^*(w).
    \end{equation*}

    Since $\hat{J}^*$ is differentiable, and its unique minimizer $w^*\in H$ is given by
    \begin{equation*}
        (\mathcal{G}_{s,t,N,g}+\beta)w^*+\mathcal{J}_{s,t,g}z_s=0,
    \end{equation*}
    that is,
    \begin{equation}\label{eq w* def}
        w^*=-(\mathcal{G}_{s,t,N,g}+\beta)^{-1}\mathcal{J}_{s,t,g}z_s.
    \end{equation}

    \vspace{0.6em}
    Meanwhile, using the Fenchel dual relation, the unique minimizer $\zeta^*$ of $J$ is given by
    \begin{equation}\label{eq zeta* def}
            \zeta^*(r)=\mathsf{P}_N\left(B^*\mathcal{J}_{r,t,g}^*w^*\right),\quad r\in(s,t).
    \end{equation}
    This algins with the control $\zeta$ defined in Theorem \ref{Thm linear-stable}. In what follows, we verify the desired stabilization property \eqref{eq z sqz1}, \eqref{eq z sqz2}.
     \vspace{0.6em}
    
    By definition, one has
    \begin{equation*}
        J(\zeta^*)=-\hat{J}^*(w^*)=\frac{1}{2}\langle w^*,\mathcal{J}_{s,t,g}z_s\rangle=\frac{1}{2}\langle \mathcal{J}_{s,t,g}^*w^*,z_s\rangle.
    \end{equation*}
   Also, from our construction, it follows that
    \begin{equation}\label{eq zeta-beta1}
        \frac{1}{2}\|\zeta^*\|^2_{\mathcal{H}}+\frac{1}{2\beta}\|z(t)\|^2=J(\zeta^*)=\frac{1}{2}\langle \mathcal{J}_{s,t,g}^*w^*,z_s\rangle.
    \end{equation}

    Applying the weak observability inequality, i.e. \eqref{eq TOI}, the Cauchy--Schwarz inequality and \eqref{eq zeta* def}, it yields that
    \begin{equation}\label{eq zeta-beta2}
    \begin{aligned}
        &|\langle \mathcal{J}_{s,t,g}^*w^*,z_s\rangle|\\
        &\leq \|\mathcal{J}_{s,t,g}^*w^*\| \|z_s\|\\
        &\leq C\exp\left(C((t-s)\|g\|_{L^\infty(Q)}+(t-s)^{-2})\right)\left(\|\mathsf{P}_N(\mathbf{1}_{(a,b)}\mathcal{J}^*_{\cdot,t,g}w_*)\|_{\mathcal{H}}+N^{-1/2}\|w_*\|\right)\|z_s\|\\
        &\leq C\exp\left(C((t-s)\|g\|_{L^\infty(Q)}+(t-s)^{-2})\right)\left(\hat{b}_N\|\zeta^*\|_{\mathcal{H}}+N^{-1/2}\|w^*\|\right)\|z_s\|.
    \end{aligned}          
    \end{equation}
    
    In addition, in view of the relation between $\zeta^*$ and $w^*$, one has
    \begin{equation}\label{eq zeta-beta3}
    \begin{aligned}        z(t)&=\mathcal{J}_{s,t,g}z_s+\int_s^t\mathcal{J}_{r,t,g}\mathsf{P}_N\zeta^*(r)dr=\mathcal{J}_{s,t,g}z_s+\int_s^t\mathcal{J}_{r,t,g}\mathsf{P}_NB^*\mathcal{J}^*_{r,t,g}w^*dr\\
        &=\mathcal{J}_{s,t,g}z_s+\mathcal{G}_{s,t,N,g}w^*=-\beta w^*,
    \end{aligned}
    \end{equation}
    where the last equality is due to \eqref{eq w* def}.

    Summarizing \eqref{eq zeta-beta1}-\eqref{eq zeta-beta3}, we obtain that
    \begin{equation*}
        \|\zeta^*\|^2_{\mathcal{H}}+\frac{1}{\beta}\|z(t)\|^2\leq C\exp\left(C((t-s)\|g\|_{L^\infty(Q)}+(t-s)^{-2})\right)\left(\hat{b}_N\|\zeta^*\|_{\mathcal{H}}+\beta^{-1}N^{-1/2}\|z(t)\|\right)\|z_s\|.
    \end{equation*}
    Using Young's inequality, one has
     \begin{equation*}
        \|\zeta^*\|^2_{\mathcal{H}}+\frac{1}{\beta}\|z(t)\|^2\leq C\exp\left(C((t-s)\|g\|_{L^\infty(Q)}+(t-s)^{-2})\right)\left(\hat{b}_N^2+{\beta}^{-1}N^{-1}\right)\|z_s\|^2,
    \end{equation*}
    which implies that
    \begin{equation*}
        \|z(t)\|\leq C\exp\left(C((t-s)\|g\|_{L^\infty(Q)}+(t-s)^{-2})\right)\left(\hat{b}_N\beta^{1/2}+N^{-1/2}\right)\|z_s\|,
    \end{equation*}
    and 
    \begin{equation*}
        \|\zeta^*\|_{\mathcal{H}}\leq C\exp\left(C((t-s)\|g\|_{L^\infty(Q)}+(t-s)^{-2})\right)\left(\hat{b}_N+\beta^{-1/2}N^{-1/2}\right)\|z_s\|,
    \end{equation*}
    completing the proof of \eqref{eq z sqz1} and \eqref{eq z sqz2}.

    \end{proof}

    \subsection{Asymptotic strong Feller property}\label{Sec 3.2}
    With the stabilization result for the linearized equation, i.e. Theorem \ref{Thm linear-stable} at hand, this subsection derives the asymptotic strong Feller property, i.e. Proposition \ref{prop ASF}. The proof also relies on the Malliavin calculus techniques, as well as a priori estimates collected in Appendix \ref{Sec B}. 

    \vspace{0.6em}

    The overall strategy is inspired by the approach in \cite{HM-06,HM-11b}. Specifically, it suffices to construct suitable controls $v\in L^2(\Omega;\mathcal{H})$, which are not necessarily adapted, to stabilize the random process 
    \begin{equation}\label{eq rho}
        \rho_t=\mathcal{J}_{0,t}\xi-\mathcal{A}_{0,t}v.
    \end{equation}    
    The construction is explicitly given by formulas \eqref{eq vn def},\eqref{eq v def}, which combines two observations:
    \begin{itemize}
        \item[(\runum{1})] The formulation of random process $\rho$ aligns with the deterministic stabilization problem addressed in Theorem~\ref{Thm linear-stable};
        \item[(\runum{2})] The process $\rho$ can be stabilized within a short time interval by using controls acting only on {\it finitely many} spatial modes.
    \end{itemize}

    \begin{proof}[Proof of Proposition \ref{prop ASF}]

    Let us fix any $B_0,\gamma,\alpha>0$ and  $\xi\in H$ with $\|\xi\|=1$. For simplicity, we suppress the dependence on $B_0$ in what follows.  Using equation \eqref{eq rho}, it follows that
    \begin{align*}
        \langle\nabla P_t\varphi(u_0),\xi\rangle&=\E((\nabla\varphi)(u_t)(\mathcal{A}_{0,t}v+\mathcal{J}_{0,t}\xi-\mathcal{A}_{0,t}v))\\
        &=\E((\nabla\varphi)(u_t)\langle \mathcal{D}u_t,v\rangle_{\mathcal{H}})+\E((\nabla\varphi)(u_t)\rho_t)\\
        &=\E(\langle \mathcal{D}\varphi(u_t),v\rangle_{\mathcal{H}})+\E((\nabla\varphi)(u_t)\rho_t)\\
        &=\E\left(\varphi(u_t)\int_0^tv(s)dW(s)\right)+\E((\nabla\varphi)(u_t)\rho_t),
    \end{align*}
    where the last inequality is due to the integration by parts formula, see e.g. \cite{Nualart-06}. This leads to
    \begin{align}\label{eq ASF0}       
        |\langle\nabla P_t\varphi(u_0),\xi\rangle|&\leq \left(\E\left|\int_0^tv(s)dW(s)\right|^2\right)^{1/2}\sqrt{P_t|\varphi|^2(u_0)}+ \left(\E\|\rho_t\|^2\right)^{1/2}\sqrt{P_t\|\nabla\varphi\|^2(u_0)}.
    \end{align}

    \vspace{0.6em}

    Consequently, it suffices to construct a suitable control $v$ to stabilize $\rho$ with finite cost,    \begin{equation}\label{rho equation}
    \begin{cases}
         \partial_t\rho_t-\nu \partial_x^2\rho_t+(3u_t^2-\lambda)\rho_t=-Bv(t),\quad t>0,\\         
        \rho_0=\xi,
    \end{cases}
    \end{equation}
     which has the same structure as equation \eqref{z equation}, except that the deterministic potential $g$ is now replaced by the random potential $3u_t^2-\lambda$, and the original control acts only on finitely many spatial modes. Therefore, we  apply the stabilization result, i.e. Theorem~\ref{Thm linear-stable}, which is well-aligned with the present stochastic setting.

    \vspace{0.6em}
    
    Let $\delta=\delta(\gamma,\alpha)\in(0,1/2)$, $\beta=\beta(\gamma,\alpha)>0$ and $N=N(\gamma,\alpha)\in\N^+$ be three parameters to be determined later. For each $n\in\N^+$, define $v_n\in L^2(\Omega;\mathcal{H}_{n-\delta,n})$ as
   \begin{equation}\label{eq vn def}
        v_n(t)=\mathsf{P}_N\left(B^*\mathcal{J}_{t,n}^*(\mathcal{M}_{n-\delta,n,N}+\beta)^{-1}\mathcal{J}_{n-\delta,n}\,\rho_{n-\delta}\right),\quad t\in(n-\delta,n).
    \end{equation}
    We further set
    \begin{equation}\label{eq v def}
        v(t)=\begin{cases}
            v_n(t),\quad &\text{for } n\in\N^+,\;t\in(n-\delta,n),\\
            0,\quad &\text{otherwise}.
        \end{cases}
    \end{equation}   
    \vspace{0.3em}
    
    Invoking the well-posedness of equation \eqref{AC equation} given by Proposition \ref{prop well-posed}, we know that
    \begin{equation*}
        u\in C([s,+\infty);H_0^1(0,\pi))\quad a.s.\quad \forall\,s>0,
    \end{equation*}
    thus
    \begin{equation*}
        u^2\in L^\infty((s,t)\times (0,\pi))\quad a.s.\quad \forall\, t>s>0.
    \end{equation*}

    Therefore, for each $n\in\N^+$, one can see that $\{\rho_t\}_{t\in(n-\delta,n)}$ solves \eqref{z equation} with
    \begin{equation*}
        g(t)=3u_t^2-\lambda,\quad \zeta(t)=v(t),\quad t\in(n-\delta,n),\quad z_{n-\delta}=\rho_{n-\delta}.
    \end{equation*}
    In particular, 
    \begin{equation*}
        \mathcal{A}_{n-\delta,n,N}=\mathcal{A}_{n-\delta,n,N,g},\quad \mathcal{M}_{n-\delta,n,N}=\mathcal{G}_{n-\delta,n,N,g}\quad a.s.
    \end{equation*}

    With these preparations,  according to   stabilization result Theorem \ref{Thm linear-stable}, there exists a universal constant $C>0$, which is independent of $n,\delta,\beta,N,u$, such that     \begin{equation}\label{eq rho bdd1}
        \|\rho_n\|\leq  Ce^{C\delta^{-2}}\exp\left(C\delta\|u^2\|_{L^\infty(Q_{n,\delta})}\right)\left(\hat{b}_N\beta^{1/2}+N^{-1/2}\right)\|\rho_{n-\delta}\|\quad \forall\, n\in\N^+,
    \end{equation}
    where $Q_{n,\delta}:=(n-\delta,n)\times(0,\pi)$ denotes the space-time domain.

    Since $v(t)=0$ for $t\in(n-1,n-\delta)$, using Lemma \ref{lemma J bounds}, one has 
    \begin{equation}\label{eq rho bdd3}
        \|\rho_{n-\delta}\|\leq e^{\lambda(1-\delta)}\|\rho_{n-1}\|\leq e^{\lambda}\|\rho_{n-1}\|\quad \forall\, n\in\N^+,\;\delta\in(0,1/2).
    \end{equation}

    It follows that there exists a universal positive constant, still denoted by $C$, such that
    \begin{equation}\label{eq rho bdd4}
         \|\rho_n\|\leq  Ce^{C\delta^{-2}}\exp\left(C\delta\|u^2\|_{L^\infty(Q_{n,\delta})}\right)\left(\hat{b}_N\beta^{1/2}+N^{-1/2}\right)\|\rho_{n-1}\|\quad \forall\, n\in\N^+.
    \end{equation}

    Invoking this relation \eqref{eq rho bdd4}, we then first choose $\delta$ sufficiently small such that the term $\E\exp(C\delta\|u^2\|_{L^\infty(Q_{n,\delta})})$ is bounded, then determine the parameters $N$ and $\beta$ to stabilize $\rho$. In summary, we derive the following lemma. 
    
    \begin{lemma}\label{prop rho_n 2} 

    For the equation given by \eqref{AC equation}-\eqref{eq psi-def} and any $B_0,\gamma,\alpha>0$, there exists a constant $N\in \N^+$ such that if the sequence $\{b_j\}_{j\in\N^+}$ satisfies 
    \begin{equation*}
       \sum_{j\in\N^+} j^2b_j^2\leq B_0\quad \text{and}\quad b_j\neq 0\quad \text{for }1\leq j\leq N,
    \end{equation*}
     then there are constants $\delta\in(0,1/2)$ and $C,\beta>0$ such that the solution $\rho$ of equation \eqref{rho equation} with $v$ given by \eqref{eq vn def},\eqref{eq v def} satisfies that
        \begin{equation}\label{eq rho_n 2}
            \E\|\rho_t\|^{2}\leq Ce^{-\alpha t}\exp\left(\gamma\|u_0\|^2\right)\quad \forall\,u_0\in H,\;t\geq 0,
        \end{equation}
        and 
         \begin{equation}\label{eq v}
            \E\left(\int_{\R^+}v(t)dW(t)\right)^2\leq C\exp\left(\gamma\|u_0\|^2\right)\quad\forall\,u_0\in H.
        \end{equation}
    \end{lemma}

    \vspace{0.6em}
   
    To establish Lemma \ref{prop rho_n 2}, we first derive  the following lemma by using energy estimates and choosing parameters properly.    

     \begin{lemma}\label{prop rho_n 1}
       For the equation given by \eqref{AC equation}-\eqref{eq psi-def} and any $B_0,\gamma,\varepsilon>0$, there exists a constant $N\in \N^+$ such that if the sequence $\{b_j\}_{j\in\N^+}$ satisfies 
     \begin{equation*}
       \sum_{j\in\N^+} j^2b_j^2\leq B_0\quad \text{and}\quad b_j\neq 0\quad \text{for }1\leq j\leq N,
     \end{equation*} then there are constants $\delta\in(0,1/2)$ and $\beta>0$ such that the solution $\rho$ of equation \eqref{rho equation} with $v$ given by \eqref{eq vn def},\eqref{eq v def} satisfies that
     \begin{equation}\label{eq rho_n 1}
            \E\left(\|\rho_n\|^4|\mathcal{F}_{n-1}\right)\leq \varepsilon e^{\gamma\|u_{n-1}\|^2}\|\rho_{n-1}\|^4\quad \forall\,u_0\in H,\;n\in\N^+.
        \end{equation}
    \end{lemma}

    \begin{proof}
        Let $B_0>0$ and $\gamma,\varepsilon\in(0,1)$ be arbitrarily given. We denote the generic constant $C$ presented in \eqref{eq rho bdd4} by $\hat{C}$.   Taking a priori estimate \eqref{eq L-inf bdd} into consideration, there exist constants $C_0,c_0,\gamma_0>0$ such that
        for any $\delta\in(0,1/2)$, $\eta\in(0,\gamma_0]$, $u_0\in H$ and $n\in\N^+$,       
        \begin{equation*}       \E\left(\exp\left(\eta\|u^2\|_{L^\infty((n-\delta,n)\times(0,\pi))}\right)|\mathcal{F}_{n-1}\right)\leq C_0\exp(c_0\eta\|u_{n-1}\|^2).
        \end{equation*}
        
       Accordingly, let us first fix any $\delta\in (0,1/2)$ such that
        \begin{equation*}
            \delta\leq\frac{\gamma}{4c_0\hat{C}}\wedge\frac{\gamma_0}{4\hat{C}},            
        \end{equation*}
        which ensures that
        \begin{align*}
            \E\left(\exp\left(4\hat{C}\delta \|u^2\|_{L^\infty(Q_{n,\delta})}\right)|\mathcal{F}_{n-1}\right)\leq C_0\exp(\gamma\|u_{n-1}\|^2).
        \end{align*}

        Then by \eqref{eq rho bdd4}, we derive that
        \begin{align*}
             \E\left(\|\rho_n\|^4|\mathcal{F}_{n-1}\right)&\leq \hat{C}^4e^{4\hat{C}\delta^{-2}}\left(\hat{b}_N\beta^{1/2}+N^{-1/2}\right)^4\|\rho_{n-1}\|^4\E\left(\exp\left(4\hat{C}\delta\|u^2\|_{L^\infty(Q_{n,\delta})}\right)|\mathcal{F}_{n-1}\right)\\
             &\leq C_0\hat{C}^4e^{4\hat{C}\delta^{-2}}\left(\hat{b}_N\beta^{1/2}+N^{-1/2}\right)^4e^{\gamma\|u_{n-1}\|^2}\|\rho_{n-1}\|^4\\
             &:=\tilde{C}\left(\hat{b}_N\beta^{1/2}+N^{-1/2}\right)^4e^{\gamma\|u_{n-1}\|^2}\|\rho_{n-1}\|^4.
        \end{align*}

        Therefore, by choosing $N=N(\gamma,\varepsilon)\in\N^+$ such that
        \begin{equation*}
            N^{-1/2}\leq 2^{-1}\tilde{C}^{-1/4}\varepsilon^{1/4},
        \end{equation*}
        and taking $\beta=\beta(\gamma,\varepsilon)>0$ as
        \begin{equation*}
            \hat{b}_N\beta^{1/2}:=2^{-1}\tilde{C}^{-1/4}\varepsilon^{1/4},
        \end{equation*}
        we establish the desired relation \eqref{eq rho_n 1}.  
        
    \end{proof}

     Finally, uisng Lemma \ref{prop rho_n 1}, the proof of Lemma \ref{prop rho_n 2} is fairly standard and is provided in Appendix \ref{Sec C}. Combining relation \eqref{eq ASF0} with Lemma \ref{prop rho_n 2}, we thus complete the proof of Proposition~\ref{prop ASF}.

    \end{proof}

    \section{Global steady-state controllability and irreducibility}\label{Sec 4}

    This section is devoted to establishing the global steady-state controllability for the deterministic Allen--Cahn equation, i.e. Theorem \ref{thm null-control} and Theorem \ref{thm irre-control}. The latter result will then be applied to deduce the irreducibility property, i.e. Proposition \ref{prop irreducibility}, which is further used in Section~\ref{Sec 5} for the proof of the Main Theorem. 
    
    The controllability analysis forms the main part of this section, which is presented in Section~\ref{Sec 4.1}, along with its outline. Meanwhile, the proof of irreducibility is provided in Section~\ref{Sec 4.2}.
    
    \vspace{0.6em}

    Let us consider the deterministic Allen--Cahn equation given by 
    \begin{equation}\label{AC determine equation}
    \begin{cases}
         \partial_tw-\nu \partial_x^2w+w^3-\lambda w=h(t,x),\quad x\in (0,\pi),\;t\in(0,T),\\
        w(0,\cdot)=u_0(\cdot),
    \end{cases}
    \end{equation}
    where $h\in L^2(0,T;L^2(a,b))$ denotes a control. The solution $w$ is denoted by $w(t,u_0,h)$.

    \vspace{0.3em}
    We also introduce the unforced Allen--Cahn equation $\phi$, which reads
    \begin{equation}\label{AC unforced equation}
    \begin{cases}
         \partial_t\phi-\nu \partial_x^2\phi+\phi^3-\lambda \phi=0,\quad x\in (0,\pi),\;t>0,\\
        \phi(0,\cdot)=u_0(\cdot).
    \end{cases}
    \end{equation}
    The associated equilibrium set for the equation \eqref{AC unforced equation}  is denoted by
    \begin{equation}\label{eq S def}
        \mathcal{S}=\{\phi\in H_0^1(0,\pi):-\nu\phi''+\phi^3-\lambda \phi=0\}.
    \end{equation}

    When $0<\lambda\leq \nu$, it is straightforward to show that equation \eqref{AC unforced equation} is globally stable and $\mathcal{S}=\{0\}$.  In the case of $\lambda>\nu$, by applying \cite[Theorem 4.3.12]{Hale-88} and performing a transformation of  $\tilde{\phi}(x)=\lambda^{-1/2}\phi(\pi x)$, one sees that equation \eqref{AC unforced equation} admits exactly $2\lfloor\lambda/\nu\rfloor+1$ steady-states, which are written as  
    \begin{equation*}
        \mathcal{S}=\left\{\phi_k:0\leq |k|\leq \lfloor\lambda/\nu\rfloor\right\}\subset C^\infty([0,\pi]).
    \end{equation*}
    In particular,  $\phi_0=0$ and  $\phi_1\geq 0\geq \phi_{-1}$. Also denote the solution of equation \eqref{AC unforced equation} by $\phi(t,u_0)$.

    \vspace{0.3em}
    The goal is to establish the global steady-state controllability  in the following sense:

    \begin{definition}\label{def global steady-state}
        The system \eqref{AC determine equation} is called globally steady-state controllable if for any $\phi\in\mathcal{S}$ and $u_0\in H$, there exists $T>0$ and $h\in L^2(0,T;L^2(a,b))$ such that the solution $w$ of equation \eqref{AC determine equation} satisfies that
       \begin{equation*}
           w(T,u_0,h)=\phi.
       \end{equation*}
    \end{definition}

    \begin{theorem}\label{thm null-control}{\rm(}Global steady-state exact controllability{\rm)} 
       For any $\nu,\lambda>0$ and $0\leq a<b\leq \pi$, the system \eqref{AC determine equation} is globally steady-state controllable.  In particular, the case of  $\phi=\phi_0=0$ implies the global null controllability.
    \end{theorem}

    Moreover, to derive the irreducibility, we further have the following global approximate controllability by using a {\it well-prepared finite family of finite dimensional controls}. Recall that $\mathsf{P}_NL^2(a,b)=\text{span}\{\psi_j:1\leq j\leq N\}$.

    \begin{theorem}\label{thm irre-control} For any $\lambda>\nu>0$ and $0\leq a<b\leq \pi$, let $u_*\in \{\phi_1,\phi_{-1}\}$. Then there exist constants $N,m\in\N^+$, $T_*>0$ and a finite family of controls 
    \begin{equation}\label{eq zeta_m}
        \{\zeta_n\in L^2(0,T_*;\mathsf{P}_NL^2(a,b)): 1\leq n\leq m\}
    \end{equation}
    with the following property.  For any $\varepsilon>0$, there exists $T=T(\varepsilon)\geq T_*$ such that for any $u_0\in H$, there exists $n\in\{1,\cdots,m\}$ such that the solution $w$ of equation \eqref{AC determine equation} satisfies 
        \begin{equation}\label{eq control}
            \|w(t,u_0,\zeta_n)-u_*\|\leq \varepsilon\quad \forall\,t\geq T,
        \end{equation}
    where  each $\zeta_n$ is extended by zero outside $[0,T_*]$.
    \end{theorem}

    \begin{remark}\label{rmk lambda}
        For the case $\lambda\leq\nu$, direct estimates yield that the solution $\phi$ of unforced equation \eqref{AC unforced equation} is globally stable: for any $\varepsilon>0$, there exists a constant $T=T(\varepsilon)>0$  such that
        \begin{equation*}
            \|\phi(t,u_0)-0\|\leq \varepsilon\quad \forall\,u_0\in H,\;t\geq T.
        \end{equation*}
    \end{remark}

    \vspace{0.3em}

    With the help of Theorem \ref{thm irre-control} and Remark \ref{rmk lambda}, we show the  irreducibility.
    
    \begin{proposition}\label{prop irreducibility}{\rm(}Irreducibility{\rm)} 
    For the equation given by \eqref{AC equation}-\eqref{eq psi-def}, there exists $u_*\in H$ and $N\in\N^+$ such that if the sequence $\{b_j\}_{j\in\N^+}$  satisfies 
    \begin{equation*}
        \sum_{j\in \N^+}b_j^2<\infty\quad \text{and}\quad b_j\neq 0\quad \text{for }1\leq j\leq N,
    \end{equation*}
    then for any $\varepsilon,R>0$, there exists  $T=T(\varepsilon)>0$ such that for any $t\geq T$, the solution $u$ of equation \eqref{AC equation} satisfies that
        \begin{equation}\label{eq irreducibility}
            \inf_{u_0\in B_H(0,R)}\Pb\left(\|u(t,u_0)-u_*\|\leq \varepsilon\right)>0.
        \end{equation}
    \end{proposition}

    Let us mention that the notion of irreducibility plays a fundamental role in the study of ergodicity and mixing. When the unforced system is globally stable, irreducibility  can often be established directly; see e.g. \cite{EM-01,FGRT-15,LWXZZ-24}. In contrast, when the global stability fails, irreducibility is a more dedicate issue and can be obtained through global controllability arguments. In our setting, the case $\lambda\leq \nu$  corresponds to the globally stable regime and hence falls into the former scenario,  while the case $\lambda>\nu$ requires a controllability-based approach and thus falls into the latter. Specifically, for $\lambda>\nu$, the dynamics of the unforced system can be described as shown in Figure~\ref{figure 4} below.

    \vspace{-1.5em}
    
    \begin{figure}[th]
    \centering

\tikzset{every picture/.style={line width=0.75pt}} 

\begin{tikzpicture}[x=0.75pt,y=0.75pt,yscale=-1,xscale=1]

\draw  [line width=2.25] [line join = round][line cap = round] (230.45,104.96) .. controls (230.45,104.96) and (230.45,104.96) .. (230.45,104.96) ;
\draw [color={rgb, 255:red, 208; green, 2; blue, 27 }  ,draw opacity=0.7 ]   (229.14,29.9) .. controls (201.84,31.67) and (170.81,54.02) .. (160.81,60.02) .. controls (150.81,66.02) and (118.5,85) .. (105.74,103.74) ;
\draw [shift={(189.61,42.78)}, rotate = 333.46] [fill={rgb, 255:red, 208; green, 2; blue, 27 }  ,fill opacity=0.7 ][line width=0.08]  [draw opacity=0] (7.14,-3.43) -- (0,0) -- (7.14,3.43) -- cycle    ;
\draw [shift={(128.08,81.73)}, rotate = 323.12] [fill={rgb, 255:red, 208; green, 2; blue, 27 }  ,fill opacity=0.7 ][line width=0.08]  [draw opacity=0] (7.14,-3.43) -- (0,0) -- (7.14,3.43) -- cycle    ;
\draw [color={rgb, 255:red, 245; green, 166; blue, 35 }  ,draw opacity=1 ]   (107.11,104.95) -- (166.61,104.95) -- (228.84,105.14) ;
\draw [shift={(131.56,104.95)}, rotate = 0] [fill={rgb, 255:red, 245; green, 166; blue, 35 }  ,fill opacity=1 ][line width=0.08]  [draw opacity=0] (7.14,-3.43) -- (0,0) -- (7.14,3.43) -- cycle    ;
\draw [shift={(192.42,105.03)}, rotate = 0.18] [fill={rgb, 255:red, 245; green, 166; blue, 35 }  ,fill opacity=1 ][line width=0.08]  [draw opacity=0] (7.14,-3.43) -- (0,0) -- (7.14,3.43) -- cycle    ;
\draw [color={rgb, 255:red, 245; green, 166; blue, 35 }  ,draw opacity=1 ]   (353.9,105.04) -- (292.73,105.95) -- (231.73,104.95) ;
\draw [shift={(328.61,105.42)}, rotate = 179.15] [fill={rgb, 255:red, 245; green, 166; blue, 35 }  ,fill opacity=1 ][line width=0.08]  [draw opacity=0] (7.14,-3.43) -- (0,0) -- (7.14,3.43) -- cycle    ;
\draw [shift={(267.53,105.53)}, rotate = 180.94] [fill={rgb, 255:red, 245; green, 166; blue, 35 }  ,fill opacity=1 ][line width=0.08]  [draw opacity=0] (7.14,-3.43) -- (0,0) -- (7.14,3.43) -- cycle    ;
\draw  [line width=2.25] [line join = round][line cap = round] (355.45,104.96) .. controls (355.45,104.96) and (355.45,104.96) .. (355.45,104.96) ;
\draw  [line width=2.25] [line join = round][line cap = round] (105.45,104.96) .. controls (105.45,104.96) and (105.45,104.96) .. (105.45,104.96) ;
\draw  [line width=2.25] [line join = round][line cap = round] (230.45,179.96) .. controls (230.45,179.96) and (230.45,179.96) .. (230.45,179.96) ;
\draw  [line width=2.25] [line join = round][line cap = round] (230.45,29.96) .. controls (230.45,29.96) and (230.45,29.96) .. (230.45,29.96) ;
\draw [color={rgb, 255:red, 245; green, 166; blue, 35 }  ,draw opacity=1 ]   (230.5,178.43) -- (230.48,140.95) -- (230.5,106.47) ;
\draw [shift={(230.5,164.99)}, rotate = 269.97] [fill={rgb, 255:red, 245; green, 166; blue, 35 }  ,fill opacity=1 ][line width=0.08]  [draw opacity=0] (7.14,-3.43) -- (0,0) -- (7.14,3.43) -- cycle    ;
\draw [shift={(230.49,129.01)}, rotate = 270.04] [fill={rgb, 255:red, 245; green, 166; blue, 35 }  ,fill opacity=1 ][line width=0.08]  [draw opacity=0] (7.14,-3.43) -- (0,0) -- (7.14,3.43) -- cycle    ;
\draw [color={rgb, 255:red, 245; green, 166; blue, 35 }  ,draw opacity=1 ]   (230.5,31.54) -- (230.5,63.54) -- (230.5,103.51) ;
\draw [shift={(230.5,42.24)}, rotate = 90] [fill={rgb, 255:red, 245; green, 166; blue, 35 }  ,fill opacity=1 ][line width=0.08]  [draw opacity=0] (7.14,-3.43) -- (0,0) -- (7.14,3.43) -- cycle    ;
\draw [shift={(230.5,78.22)}, rotate = 90] [fill={rgb, 255:red, 245; green, 166; blue, 35 }  ,fill opacity=1 ][line width=0.08]  [draw opacity=0] (7.14,-3.43) -- (0,0) -- (7.14,3.43) -- cycle    ;
\draw [color={rgb, 255:red, 208; green, 2; blue, 27 }  ,draw opacity=0.7 ]   (229.05,180.09) .. controls (202.25,179.29) and (169.29,155.08) .. (162.79,151.08) .. controls (156.29,147.08) and (118.5,120.33) .. (106.13,106.24) ;
\draw [shift={(190.57,167.98)}, rotate = 26.41] [fill={rgb, 255:red, 208; green, 2; blue, 27 }  ,fill opacity=0.7 ][line width=0.08]  [draw opacity=0] (7.14,-3.43) -- (0,0) -- (7.14,3.43) -- cycle    ;
\draw [shift={(130.26,127.51)}, rotate = 37.69] [fill={rgb, 255:red, 208; green, 2; blue, 27 }  ,fill opacity=0.7 ][line width=0.08]  [draw opacity=0] (7.14,-3.43) -- (0,0) -- (7.14,3.43) -- cycle    ;
\draw [color={rgb, 255:red, 208; green, 2; blue, 27 }  ,draw opacity=0.7 ]   (232.12,30.06) .. controls (258.92,30.86) and (290.7,54.91) .. (297.2,58.91) .. controls (303.7,62.91) and (339.17,87) .. (354.55,103.37) ;
\draw [shift={(270.07,42.1)}, rotate = 206.77] [fill={rgb, 255:red, 208; green, 2; blue, 27 }  ,fill opacity=0.7 ][line width=0.08]  [draw opacity=0] (7.14,-3.43) -- (0,0) -- (7.14,3.43) -- cycle    ;
\draw [shift={(330.3,82.34)}, rotate = 217.21] [fill={rgb, 255:red, 208; green, 2; blue, 27 }  ,fill opacity=0.7 ][line width=0.08]  [draw opacity=0] (7.14,-3.43) -- (0,0) -- (7.14,3.43) -- cycle    ;
\draw [color={rgb, 255:red, 208; green, 2; blue, 27 }  ,draw opacity=0.7 ]   (231.83,179.78) .. controls (258.5,179) and (290.16,155.66) .. (300.16,149.66) .. controls (310.16,143.66) and (344.5,119.67) .. (355.33,106.47) ;
\draw [shift={(271.49,167.08)}, rotate = 153.09] [fill={rgb, 255:red, 208; green, 2; blue, 27 }  ,fill opacity=0.7 ][line width=0.08]  [draw opacity=0] (7.14,-3.43) -- (0,0) -- (7.14,3.43) -- cycle    ;
\draw [shift={(332.12,127.2)}, rotate = 142.72] [fill={rgb, 255:red, 208; green, 2; blue, 27 }  ,fill opacity=0.7 ][line width=0.08]  [draw opacity=0] (7.14,-3.43) -- (0,0) -- (7.14,3.43) -- cycle    ;

\draw (86.94,98.36) node [anchor=north west][inner sep=0.75pt]  [font=\scriptsize]  {$\phi _{1}$};
\draw (241.2,111.36) node [anchor=north west][inner sep=0.75pt]  [font=\scriptsize]  {$\phi _{0} =0$};
\draw (363.44,98.36) node [anchor=north west][inner sep=0.75pt]  [font=\scriptsize]  {$\phi _{-1}$};
\draw (223.94,14.36) node [anchor=north west][inner sep=0.75pt]  [font=\scriptsize]  {$\phi _{2}$};
\draw (222.44,183.86) node [anchor=north west][inner sep=0.75pt]  [font=\scriptsize]  {$\phi _{-2}$};

\end{tikzpicture}

    \vspace{-1em}
    \caption{Dynamics of the unforced system for $\lambda\in(2\nu,3\nu]$.}
        \label{figure 4}
    \end{figure}
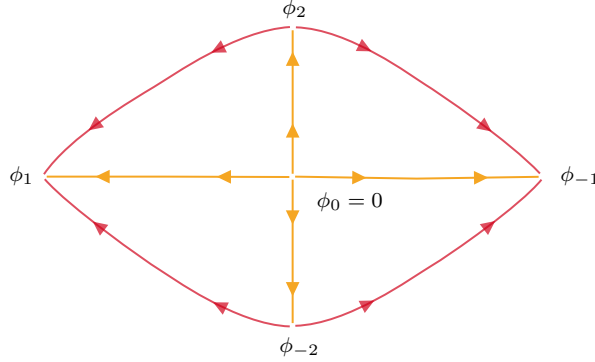

\vspace{-1em}

    \subsection{Global steady-state controllability}\label{Sec 4.1}

    We in this subsection, establish the proof of Theorem \ref{thm null-control} and Theorem \ref{thm irre-control}. The strategy builds upon the quasi-static deformation method introduced in \cite{Coron-02,CT-04}. The proofs rely on two  ingredients:

    \begin{itemize}
        \item[(\runum{1})] An explicit construction of the control in a feedback-type form, obtained by the controlled Lyapunov theory combined with the quasi-static deformation;
        \item[(\runum{2})] The global dynamics of unforced equation, which generates a gradient system.
    \end{itemize}

    \vspace{0.3em}
    
    The proof is divided into four steps, as outlined below. Details are in Sections \ref{Sec 4.1.1}-\ref{Sec 4.1.4}.

    \noindent {\it Step 1: Approximate controllability between steady-states.} We begin by establishing the steady-state controllability, which constitutes the core of the argument. Specifically, we show that for any $\phi,\hat{\phi}\in\mathcal{S}$ and $\varepsilon>0$, there exists a control $h\in H^1(0,T;C^2([a,b]))$ such that 
    \begin{equation}\label{eq out4}
        \|w(T,\phi,h)-\hat{\phi}\|_{H^1}\leq \varepsilon.
    \end{equation}
    Our strategy is inspired by the approach in \cite{CT-04}, which studies exact boundary controllability of 1D semilinear heat equations. The argument combines controlled Lyapunov  theory with quasi-static deformation theory, leading to an explicit feedback-type construction of the control.  The construction essentially reduces to solving a finite-dimensional stabilization problem.

    \vspace{0.6em}
    
    \noindent {\it Step 2: Convergence to the steady-states.} Next, we summarize dynamical properties for the unforced equation \eqref{AC unforced equation}. More precisely, for any $u_0\in H$, there exists $\phi_{u_0}\in\mathcal{S}$ such that 
        \begin{equation}\label{eq out5}
            \lim\limits_{t\rightarrow\infty}\|\phi(t,u_0)-\phi_{u_0}\|_{H^1}=0.
        \end{equation}
    Invoking this convergence with compactness arguments, we further derive that for any $\varepsilon>0$, there exists $T>0$ such that for any $u_0\in H$, there exists $\phi_{u_0}\in\mathcal{S}$ and $t_{u_0}\in [1,T+1]$ satisfying 
        \begin{equation}\label{eq out6}
            \|\phi(t_{u_0},u_0)-\phi_{u_0}\|_{H^1}\leq \varepsilon.
        \end{equation}

    \noindent {\it Step 3: Local exact controllability and the proof of Theorem \ref{thm null-control}.}   Next we prove the following local controllability: for any $\phi\in\mathcal{S}$ and $T>0$, there exists $\delta>0$ such that for any $u_0\in B_{H^1}(\phi,\delta)$, there exists a control $\zeta\in L^2(0,T;L^2(a,b))$ satisfying
    \begin{equation}\label{eq out7}
        w(T,u_0,\zeta)=\phi.
    \end{equation}
    This local control result is standard; see e.g.  \cite{Emanuilov-95}. 
    Combining Steps 1-3, we conclude Theorem~\ref{thm null-control}.

    \vspace{0.6em}

    \noindent {\it Step 4: Construction of finitely many controls and proof of Theorem~\ref{thm irre-control}.} 
To establish irreducibility, the main difficulty is to construct a finite family of control inputs that works uniformly for all initial data. More precisely, the controls $\zeta_n$ are required to satisfy:
\begin{itemize}
    \item[(\runum{1})] they are finite in number, say $m$,
    \item[(\runum{2})] they belong to a finite-dimensional space $\mathsf{P}_N L^2(a,b)$,
    \item[(\runum{3})] they act over a finite time interval $T_*$,
\end{itemize}
as in \eqref{eq zeta_m}. In particular, sll parameters $m, N, T_*$ are independent of $u_0\in H$  $\varepsilon>0$.

The construction combines the existence of {\it stable steady states} $\phi_{\pm 1}$ for the unforced system~\eqref{AC unforced equation}: there exist constants $C,c,\delta>0$ such that
\begin{equation}\label{eq out8}
\|\phi(t,u_0)-\phi_{\pm 1}\|_{H^1}\leq Ce^{-ct}\|u_0-\phi_{\pm 1}\|
\quad \forall\, u_0\in B_{H}(\phi_{\pm 1},\delta),\; t\geq1.
\end{equation}
Roughly speaking, the global dynamics~\eqref{eq out5} steers the solution into a neighborhood of a steady state, which is then transferred near $\phi_{\pm1}$  via approximate controllability~\eqref{eq out4}, and finally attracted by local stability. The key point is that this procedure can be implemented using only finitely many control inputs, uniformly with respect to $u_0$ and $\varepsilon$. See Figure~\ref{figure 5} for an illustration.

    \vspace{-0.6em}
   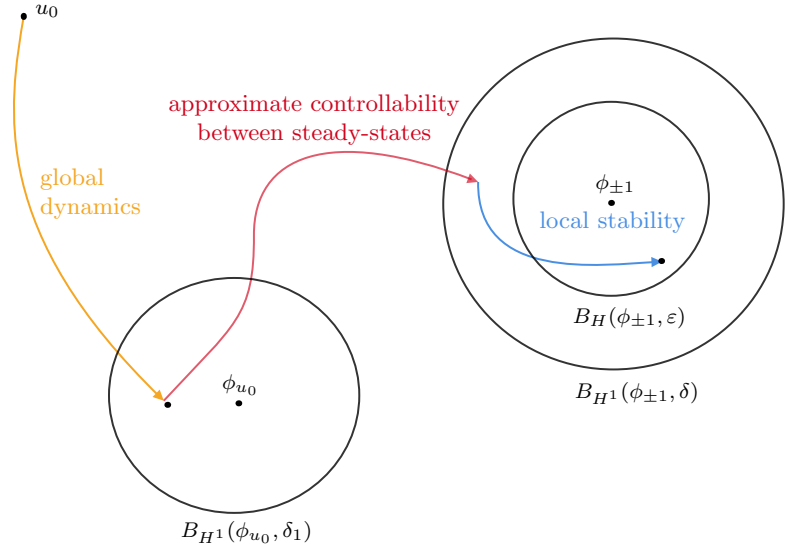
\begin{figure}[th]
    \centering

\tikzset{every picture/.style={line width=0.75pt}} 

\begin{tikzpicture}[x=0.75pt,y=0.75pt,yscale=-1,xscale=1]

\draw  [color={rgb, 255:red, 0; green, 0; blue, 0 }  ,draw opacity=0.8 ] (317.33,109.37) .. controls (317.33,63.53) and (355.56,26.36) .. (402.72,26.36) .. controls (449.89,26.36) and (488.12,63.53) .. (488.12,109.37) .. controls (488.12,155.21) and (449.89,192.37) .. (402.72,192.37) .. controls (355.56,192.37) and (317.33,155.21) .. (317.33,109.37) -- cycle ;
\draw [color={rgb, 255:red, 74; green, 144; blue, 226 }  ,draw opacity=1 ]   (334.83,98.32) .. controls (334.16,148.48) and (395.6,139.99) .. (424.16,138.44) ;
\draw [shift={(427.14,138.3)}, rotate = 177.88] [fill={rgb, 255:red, 74; green, 144; blue, 226 }  ,fill opacity=1 ][line width=0.08]  [draw opacity=0] (5.36,-2.57) -- (0,0) -- (5.36,2.57) -- cycle    ;
\draw  [line width=2.25] [line join = round][line cap = round] (401.98,108.83) .. controls (401.83,108.83) and (401.69,108.83) .. (401.54,108.83) ;
\draw  [line width=2.25] [line join = round][line cap = round] (178.91,210.12) .. controls (179.06,210.12) and (179.2,210.12) .. (179.35,210.12) ;
\draw  [line width=2.25] [line join = round][line cap = round] (334.84,97.83) .. controls (334.84,97.83) and (334.84,97.83) .. (334.84,97.83) ;
\draw [color={rgb, 255:red, 245; green, 166; blue, 35 }  ,draw opacity=1 ]   (107.16,13.65) .. controls (98.47,62.99) and (90.11,123.46) .. (175.84,206.79) ;
\draw [shift={(177.14,208.05)}, rotate = 223.95] [fill={rgb, 255:red, 245; green, 166; blue, 35 }  ,fill opacity=1 ][line width=0.08]  [draw opacity=0] (5.36,-2.57) -- (0,0) -- (5.36,2.57) -- cycle    ;
\draw  [line width=2.25] [line join = round][line cap = round] (106.89,14.96) .. controls (106.89,15.16) and (106.89,15.36) .. (106.89,15.56) ;
\draw  [color={rgb, 255:red, 0; green, 0; blue, 0 }  ,draw opacity=0.8 ] (352.53,106.77) .. controls (352.53,79.89) and (374.49,58.11) .. (401.58,58.11) .. controls (428.66,58.11) and (450.62,79.89) .. (450.62,106.77) .. controls (450.62,133.64) and (428.66,155.42) .. (401.58,155.42) .. controls (374.49,155.42) and (352.53,133.64) .. (352.53,106.77) -- cycle ;
\draw  [line width=2.25] [line join = round][line cap = round] (215.02,209.36) .. controls (214.87,209.36) and (214.73,209.36) .. (214.58,209.36) ;
\draw [color={rgb, 255:red, 208; green, 2; blue, 27 }  ,draw opacity=0.65 ]   (331.54,97.04) .. controls (259.97,69.75) and (222.14,84.91) .. (222.35,125.19) .. controls (222.57,166.08) and (212.43,169.98) .. (177.14,208.05) ;
\draw [shift={(334.83,98.32)}, rotate = 200.02] [fill={rgb, 255:red, 208; green, 2; blue, 27 }  ,fill opacity=0.65 ][line width=0.08]  [draw opacity=0] (5.36,-2.57) -- (0,0) -- (5.36,2.57) -- cycle    ;
\draw  [color={rgb, 255:red, 0; green, 0; blue, 0 }  ,draw opacity=0.8 ] (149.64,205.4) .. controls (149.64,172.85) and (177.75,146.47) .. (212.43,146.47) .. controls (247.1,146.47) and (275.21,172.85) .. (275.21,205.4) .. controls (275.21,237.94) and (247.1,264.33) .. (212.43,264.33) .. controls (177.75,264.33) and (149.64,237.94) .. (149.64,205.4) -- cycle ;
\draw  [line width=2.25] [line join = round][line cap = round] (427.13,137.98) .. controls (426.99,137.98) and (426.84,137.98) .. (426.7,137.98) ;

\draw (111.22,7.87) node [anchor=north west][inner sep=0.75pt]  [font=\scriptsize]  {$u_{0}$};
\draw (164.93,104.16) node  [font=\footnotesize,color={rgb, 255:red, 245; green, 166; blue, 35 }  ,opacity=1 ] [align=left] {\begin{minipage}[lt]{74.66pt}\setlength\topsep{0pt}
global \\dynamics
\end{minipage}};
\draw (254.68,67.12) node  [font=\footnotesize] [align=left] {\begin{minipage}[lt]{123.76pt}\setlength\topsep{0pt}
\begin{center}
\textcolor[rgb]{0.82,0.01,0.11}{approximate controllability }\\\textcolor[rgb]{0.82,0.01,0.11}{between steady-states }
\end{center}

\end{minipage}};
\draw (397.91,124.71) node  [font=\footnotesize] [align=left] {\begin{minipage}[lt]{61.51pt}\setlength\topsep{0pt}
\textcolor[rgb]{0.29,0.56,0.89}{ \ \ local stability \ \ \ }
\end{minipage}};
\draw (392.21,92.95) node [anchor=north west][inner sep=0.75pt]  [font=\scriptsize]  {$\phi _{\pm 1}$};
\draw (204.94,193.36) node [anchor=north west][inner sep=0.75pt]  [font=\scriptsize]  {$\phi _{u_{0}}$};
\draw (184.2,266.86) node [anchor=north west][inner sep=0.75pt]  [font=\scriptsize]  {$B_{H^{1}}( \phi _{u_{0}} ,\delta _{1})$};
\draw (381.77,197.32) node [anchor=north west][inner sep=0.75pt]  [font=\scriptsize]  {$B_{H^{1}}( \phi _{\pm 1} ,\delta )$};
\draw (380.77,159.32) node [anchor=north west][inner sep=0.75pt]  [font=\scriptsize]  {$B_{H}( \phi _{\pm 1} ,\varepsilon )$};

\end{tikzpicture}

    \vspace{-0.6em}
    \caption{Proof of Theorem \ref{thm irre-control}.}
        \label{figure 5}
    \end{figure}

    \subsubsection{Step 1: Approximate controllability between steady-states}\label{Sec 4.1.1}  

    \begin{proposition}\label{prop steady control} 
        For any $\nu,\lambda>0$ and $0\leq a<b\leq \pi$, the approximate controllability between steady-states holds: for any $\phi,\hat{\phi}\in\mathcal{S}$ and $\varepsilon>0$, there exists $T>0$ and $h\in H^1(0,T;C^2([a,b]))$ such that the solution $w$ of equation \eqref{AC determine equation} satisfies that
        \begin{equation}\label{eq steady control}
            \|w(T,\phi,h)-\hat{\phi}\|_{H^1}\leq \varepsilon.
        \end{equation}
    \end{proposition}

    \begin{proof}
    Let $\varepsilon>0$ and $\phi,\hat{\phi}\in\mathcal{S}$ be fixed.
        
    \vspace{0.3em}
    
     \noindent {\it Part 1: Extended steady-states $\mathcal{S}_{\rm ex}$ and connectedness. }  Let us first introduce the set of extended steady-states $\mathcal{S}_{\rm ex}$, which is given by
    \begin{equation}\label{eq Sh def}
        \mathcal{S}_{\rm ex}:=\left\{y\in C^2([0,\pi]):-\nu y''+y^3-\lambda y=h,\quad y(0)=y(\pi)=0, \quad h\in C^2([a,b])\right\},
    \end{equation}
    endowed with $C^1$-topology. Clearly, $\mathcal{S}\subset \mathcal{S}_{\rm ex}$. We shall show that $\mathcal{S}$ is $C^1$-connected in $\mathcal{S}_{\rm ex}$, i.e., there exists a $C^1$-path $\bar{y}\in C^1([0,1];H_0^1(0,\pi)\cap H^2(0,\pi))$ and $\bar{h}\in C^1([0,1];C^2([a,b]))$ such that    
    \begin{equation*}
        -\nu\partial_x^2\bar{y}(\tau,x)+\bar{y}(\tau,x)^3-\lambda \bar{y}(\tau,x)=\bar{h}(\tau,x),\quad x\in (0,\pi),\;\tau\in[0,1],
    \end{equation*}
    and
    \begin{gather*}
        \bar{y}(0,\cdot)=\phi(\cdot),\quad \bar{y}(1,\cdot)=\hat{\phi}(\cdot), \quad  \bar{y}(\tau, \cdot)\in \mathcal{S}_{\rm ex}  \quad\forall\,\tau\in [0, 1], \\
        \bar{h}(0,\cdot)=\bar{h}(1,\cdot)=0.
    \end{gather*}
    
    \vspace{0.3em}
    To prove the connectedness, let us note some facts about the maximal solutions of equations
    \begin{align}
        -\nu y''+y^3-\lambda y&=0,\quad y(0)=0,\quad y'(0)=\theta\in\R,\label{y equation2}\\
        -\nu \hat{y}''+\hat{y}^3-\lambda \hat{y}&=0,\quad \hat{y}(\pi)=0,\quad \hat{y}'(\pi)=\vartheta\in\R.\label{y equation3}
    \end{align}     
    By \eqref{y equation2}, it follows that
    \begin{equation}\label{eq y2}
        -\frac{\nu}{2} (y')^2+\frac{1}{4}y^4-\frac{\lambda}{2} y^2=-\frac{\nu}{2}\theta^2.
    \end{equation}    
    Noting that for any $\tilde\phi\in\mathcal{S}$, there exists $x_0\in(0,\pi)$ such that $\tilde\phi'(x_0)=0$. This gives that
    \begin{equation*}
       \frac{1}{4}\tilde\phi(x_0)^4-\frac{\lambda}{2} \tilde\phi(x_0)^2=-\frac{\nu}{2}(\tilde\phi'(0))^2,
    \end{equation*}
    which implies 
    \begin{equation*}
        |\tilde\phi'(0)|\leq \frac{\lambda}{\sqrt{2\nu}}.
    \end{equation*}

    Meanwhile, we claim that for any $\theta\in\left[-\frac{\lambda}{\sqrt{2\nu}},\frac{\lambda}{\sqrt{2\nu}}\right]$, the solution of \eqref{y equation2} is well-defined on $[0,\pi]$. Specifically, for $\theta=\pm \frac{\lambda}{\sqrt{2\nu}}$, it admits explicit solutions
    \begin{equation*}        y(x)=\pm\sqrt{\lambda}\tanh\left(\sqrt{\frac{\lambda}{2\nu}}x\right).
    \end{equation*}
    For $|\theta|<\frac{\lambda}{\sqrt{2\nu}}$, by \eqref{eq y2},
    \begin{equation*}
       (y')^2=\theta^2+\frac{1}{2\nu}y^4-\frac{\lambda}{\nu}y^2:=Q(y^2).
    \end{equation*}    
    Noting that $Q(\cdot)$ admits two positive roots $\lambda\pm\sqrt{\lambda^2-2\nu\theta^2}$ and $Q(y^2)\geq 0$,  one has
    \begin{equation*}
        y^2\in\left[0,\lambda-\sqrt{\lambda^2-2\nu\theta^2}\right]\quad \text{or}\quad y^2\in\left[\lambda+\sqrt{\lambda^2-2\nu\theta^2},+\infty\right).
    \end{equation*}
    Combining $y(0)=0$, we conclude that both $y(x)$ and $y'(x)$ are bounded uniformly on $[0,\infty)$.

    \vspace{0.3em}
    
    Let us set, for $\tau\in[0,1]$,
    \begin{align*}
        \theta_\tau&=(1-\tau)\phi'(0)+\tau\hat{\phi}'(0)\in\left[-\frac{\lambda}{\sqrt{2\nu}},\frac{\lambda}{\sqrt{2\nu}}\right],\\
        \vartheta_\tau&=(1-\tau)\phi'(\pi)+\tau\hat{\phi}'(\pi)\in\left[-\frac{\lambda}{\sqrt{2\nu}},\frac{\lambda}{\sqrt{2\nu}}\right].
    \end{align*}
    Then for any $\tau\in[0,1]$, the solution of \eqref{y equation2} with $\theta=\theta_\tau$ and  is well-defined on $[0,\pi]$. By similar arguments, the solution of \eqref{y equation3} with $\vartheta=\vartheta_\tau$ also exists on $[0,\pi]$. We denote these solutions by $y^\tau(\cdot)$ and $\hat{y}^\tau(\cdot)$, respectively.

    \vspace{0.3em}
    
    By our construction, one derives that for any $\tau\in[0,1]$,
    \begin{equation*}
        \bar{y}(\tau,x)=\begin{cases}
            y^\tau(x)\quad\text{for }x\in[0,a],\\
            \hat{y}^\tau(x)\quad\text{for }x\in[b,1].
        \end{cases}
    \end{equation*}
    Therefore, it remains to construct suitable controls $\bar{h}(\tau,\cdot)\in C^2([a,b])$ such that $\bar{y}(\tau,\cdot)$ connects $(y^\tau(a),(y^\tau)'(a))$ and $(\hat{y}^\tau(b),(\hat{y}^\tau)'(b))$ on $[a,b]$. Indeed, we can take $\bar{y}(\tau,x)$ to be in the form of 
    \begin{equation*}
        \bar{y}(\tau,x)=\sum_{0\leq k\leq 3}c_k(\tau)(x-a)^k,\quad x\in[a,b],
    \end{equation*}
     where $c_k(\tau)$ are constants satisfying the the initial and terminal conditions. By construction, 
     \begin{equation*}
         \bar{h}(\tau,x)=-\nu\partial_x^2\bar{y}(\tau,x)+\bar{y}(\tau,x)^3-\lambda \bar{y}(\tau,x)\in C^\infty([a,b]).
     \end{equation*}     
     Moreover, as the mappings $\tau\mapsto y^\tau$ and $\tau\mapsto \hat{y}^\tau$ are $C^1$-smooth, we can accordingly make the mapping $\tau\mapsto\bar{h}(\tau,\cdot)$ to be $C^1$-smooth. Thus, we establish a $C^1$-path $\bar{y}$ in $\mathcal{S}_{\rm ex}$ connecting $\phi$ to $\hat{\phi}$.

    \vspace{0.6em}
    
    \noindent {\it Part 2: Reduction of the problem.}  Recall that $w(t,\phi,h)$ denotes the solution of \eqref{AC determine equation} with initial condition $w(0,\cdot)=\phi(\cdot)$ and a control $h\in H^1(0,T;C^2([a,b]))$. We now define
    \begin{equation*}
        z(t,\cdot)=w(t,\phi,h)(\cdot)-\bar{y}(\varepsilon t,\cdot),\quad \zeta(t,\cdot)=h(t,\cdot)-\bar{h}(\varepsilon t,\cdot),\quad T=\varepsilon^{-1},\quad t\in [0,\varepsilon^{-1}].
    \end{equation*}
    Then $z$ satisfies the equation 
    \begin{equation}\label{z equation2}
    \begin{cases}
         \partial_tz-\nu \partial_x^2z-(\lambda-3\bar{y}^2) z+z^2(z+3\bar{y})+\varepsilon\partial_\tau\bar{y}=\zeta,\quad x\in (0,\pi),\;t\in[0,\varepsilon^{-1}],\\
         z|_{x=0,\pi}=0,\\
        z(0,\cdot)=0.
    \end{cases}
    \end{equation}

    In particular, the desired inequality \eqref{eq steady control} is thus equivalent to
    \begin{equation}\label{eq z stabilize}
        \|z(\varepsilon^{-1})\|_{H^1}\leq \varepsilon.
    \end{equation}

    \vspace{0.6em}

    Following \cite{CT-04}, we then introduce the one-parameter family of linear operators
    \begin{equation}\label{eq A_tau def}
        A(\tau)=\nu\partial_x^2+\lambda-3\bar{y}(\tau,\cdot)^2,\quad \tau\in[0,1],
    \end{equation}
    defined on $H_0^1(0,\pi)\cap H^2(0,\pi)$. Let  $\{e_j(\tau,\cdot)\}_{j\in\N^+}$ be an orthonormal basis of $L^2(0,\pi)$ of eigenfunctions of $A(\tau)$, such that for each $j\in\N^+$ and $\tau\in[0,1]$,
    \begin{equation*}
        e_j(\tau,\cdot)\in H_0^1(0,\pi)\cap C^2([0,\pi]),
    \end{equation*}
    given that $\bar{y}(\tau,\cdot)\in H_0^1(0,\pi)\cap H^2(0,\pi)$. Let $\lambda_j(\tau)$ be the corresponding eigenvalues. It follows that these eigenfunctions and eigenvalues are $C^1$ functions of $\tau$. Additionally, for each $\tau\in[0,1]$, 
    \begin{equation*}
        -\infty<\cdots<\lambda_m(\tau)<\cdots<\lambda_1(\tau)\quad \text{and }\lim\limits_{j\rightarrow\infty}\lambda_j(\tau)=-\infty.
    \end{equation*}
    In view of the continuity of the eigenvalues on $[0,1]$, there exists $\eta>0$ and $m\in\N^+$ such that
    \begin{equation}\label{eq m def}
        \lambda_k(\tau)<-\eta<0\quad \forall\,k>m,\;\tau\in[0,1].
    \end{equation}

    In this notation, the solution $z(t,\cdot)\in H_0^1(0,\pi)\cap H^2(0,\pi)$ of \eqref{z equation2} leads to
    \begin{equation}\label{z equation3}
        \partial_t z(t,\cdot)=A(\varepsilon t)z(t,\cdot)+\zeta(t,\cdot)+r(\varepsilon,t,\cdot),
    \end{equation}
    where 
    \begin{equation}\label{eq r def}
        r(\varepsilon,t,\cdot)=-z(t,\cdot)^2(z(t,\cdot)+3\bar{y}(\varepsilon t,\cdot))-\varepsilon\partial_\tau\bar{y}(\varepsilon t,\cdot).
    \end{equation}
    Moreover, $z(t,\cdot)\in H_0^1(0,\pi)\cap H^2(0,\pi)$ can be expanded as a series of the eigenfunctions $e_j(\varepsilon t,\cdot)$,  
    \begin{equation*}
        z(t,\cdot)=\sum_{j\in\N^+}z_j(t)e_j(\varepsilon t,\cdot),
    \end{equation*}
    where the convergence is in $H_0^1(0,\pi)$.

    \vspace{0.6em}
    
    For any $\tau\in[0,1]$, let $\mathscr{P}(\tau)$ be the orthogonal projection in $L^2(0,\pi)$ onto the finite-dimensional subspace spanned by $\{e_j(\tau,\cdot)\}_{1\leq j\leq m}$, and set 
    \begin{equation}\label{eq hat_z def}
        \hat{z}(t,\cdot)=\mathscr{P}(\varepsilon t)z(t,\cdot)=\sum_{1\leq j\leq m}z_j(t)e_j(\varepsilon t,\cdot).
    \end{equation}    
    Note that the operators $\mathscr{P}(\tau)$ and $A(\tau)$ commute, and for any $v\in L^2(0,\pi)$,
    \begin{equation*}
        \mathscr{P}'(\tau)v=\sum_{1\leq j\leq m}\langle v,e_j(\tau,\cdot)\rangle\partial_\tau e_j(\tau,\cdot)+\sum_{1\leq j\leq m}\langle v,\partial_\tau e_j(\tau,\cdot)\rangle e_j(\tau,\cdot).
    \end{equation*}
    Hence differentiating \eqref{eq hat_z def} with respect to $t$ and using 
    \begin{equation*}
        A(\varepsilon t)\hat{z}(t,\cdot)=\sum_{1\leq j\leq m}\lambda_j(\varepsilon t)z_j(t)e_j(\varepsilon t,\cdot),
    \end{equation*}
    we obtain 
    \begin{equation}\label{z equation4}
        \sum_{1\leq j\leq m}z_j'(t)e_j(\varepsilon t,\cdot)=\sum_{1\leq j\leq m}\lambda_j(\varepsilon t)z_j(t)e_j(\varepsilon t,\cdot)+\mathscr{P}(\varepsilon t)\zeta(t,\cdot)+\hat{r}(\varepsilon,t,\cdot),
    \end{equation}
    where 
    \begin{equation*}
        \hat{r}(\varepsilon,t,\cdot)=\mathscr{P}(\varepsilon t)r(\varepsilon,t,\cdot)+\varepsilon\sum_{1\leq j\leq m}\langle z(t,\cdot),\partial_\tau e_j(\varepsilon t,\cdot)\rangle e_j(\varepsilon t,\cdot).
    \end{equation*}

    We now take the control $\zeta\in H^1(0,T;C^2([a,b]))$ to have the form of
    \begin{equation}\label{eq zeta def2}
        \zeta(t,\cdot)=\sum_{1\leq j\leq m}\zeta_j(t)\mathbf1_{[a,b]}(\cdot)e_j(\varepsilon t,\cdot),
    \end{equation}
    where 
    \begin{equation*}
        \zeta_j\in C^1([0,\varepsilon^{-1}]),\quad\zeta_j(0)=0\quad\text{for }1\leq j\leq m.
    \end{equation*}

    \vspace{0.6em}
    
    Projecting \eqref{z equation4} on each $e_j$ for $1\leq j\leq m$, it leads to
    \begin{equation*}
        z_j'(t)=\lambda_j(\varepsilon t)z_j(t)+\sum_{1\leq k\leq m} b_{jk}(\varepsilon t) \zeta_k(t)+\hat{r}_j(\varepsilon,t),
    \end{equation*}
    where 
    \begin{equation*}
        b_{jk}(\tau)=\langle \mathbf1_{[a,b]}(\cdot)e_k(\tau,\cdot),e_j(\tau,\cdot)\rangle,\quad \hat{r}_j(\varepsilon,t)=\langle \hat{r}(\varepsilon,t,\cdot),e_j(\varepsilon t,\cdot) \rangle.
    \end{equation*}
    Set 
    \begin{equation*}
        \alpha_j(t)=\zeta_j'(t),
    \end{equation*}
    and consider now $\zeta_j(t)$ as a state and $\alpha_j(t)$ as a control. Then the former finite dimensional system can be rewritten as    
    \begin{equation}\label{z equation5}
    \begin{cases}
         z_1'=\lambda_1z_1+\sum_{1\leq k\leq m} b_{k1}\zeta_k+\hat{r}_1,\\
         \quad\;\;\vdots\\z_m'=\lambda_mz_m+\sum_{1\leq k\leq m} b_{km}\zeta_k+\hat{r}_m,\\
         \zeta_1'=\alpha_1,\\
         \quad\;\;\vdots\\         
        \zeta_m'=\alpha_m,\\
    \end{cases}
    \end{equation}
    Let us introduce the matrix notation
    \begin{equation*}
        X(t)=(z_1(t),\cdots, z_m(t), \zeta_1(t),\cdots,\zeta_m(t))^{\rm T},\quad Y(t)= (\alpha_1(t),\cdots, \alpha_m(t))^{\rm T}, 
    \end{equation*}
    \begin{equation}\label{eq R def}
        R(\varepsilon,t)=(\hat{r}_1(\varepsilon,t),\cdots, \hat{r}_m(\varepsilon,t), 0,\cdots,0)^{\rm T},
    \end{equation}
    \begin{equation*}
        \hat{A}(\tau)=\text{diag}\left\{\lambda_1(\tau),\cdots,\lambda_m(\tau)\right\},\quad B(\tau)=\left(b_{jk}(\tau)\right)_{1\leq j,k\leq m},\quad 
    \end{equation*}
    and 
    \begin{equation*}
        D(\tau)=\left(\begin{matrix}
            \hat{A}(\tau) &  B(\tau)\\
            \mathbf{0}_m & \mathbf{0}_m
        \end{matrix}\right)\in \R^{2m\times 2m},\quad E=\left(\begin{matrix}
           \mathbf{0}_m \\
             I_m
        \end{matrix}\right)\in \R^{2m\times m},
    \end{equation*}
    where $\mathbf{0}_m$ denotes the $m\times m$ zero matrix, and $I_m$ denotes the $m\times m$ identity matrix.

    Invoking this notation, equations \eqref{z equation5}  yield the finite dimensional linear control system
    \begin{equation}\label{z equation6}
        X'(t)=D(\varepsilon t)X(t)+EY(t)+R(\varepsilon,t).
    \end{equation}

    \vspace{0.6em}

    \noindent {\it Part 3: Kalman condition and pole shifting property.} In what follows, we show the pair $(D(\tau), E)$ satisfies the Kalman condition, and hence derives a pole shifting result.  
    \begin{lemma}\label{lemma Kalman}
        For each $\tau\in[0,1]$, the pair $(D(\tau), E)$ satisfies the Kalman condition, i.e.,        \begin{equation}\label{eq Kalman}
            {\rm rank }\left (E, D(\tau)E, \cdots, D(\tau)^{2m-1}E \right)=2m.
        \end{equation}
        Therefore, there exists $K(\tau)\in \R^{m\times 2m}$ such that the matrix $D(\tau)+EK(\tau)$ admits $-1$ as an eigenvalue with order $2m$. Moreover, there exists a $C^1$-mapping 
        \begin{equation*}
            Q\colon [0,1]\rightarrow \R^{2m\times 2m},\quad Q(\tau)\text{ is symmetric and positive definite}, 
        \end{equation*}
        such that
        \begin{equation}\label{eq Q def}
            Q(\tau)\left(D(\tau)+EK(\tau))+(D(\tau)+EK(\tau)\right)^{\rm T}Q(\tau)=-I_{2m}.
        \end{equation}        
    \end{lemma}
    The proof is rather standard, which is collected in Appendix~\ref{Sec E}.

    \vspace{0.6em}
    
    \noindent {\it Part 4: Stabilization via Lyapunov functional.}    We now construct a Lyapunov functional to stabilize system \eqref{z equation6}. Let $c>0$  be a constant specified later.  Define
    \begin{equation*}
        V\colon [0,\varepsilon^{-1}]\times \R^m\times H_0^1(0,\pi)\cap H^2(0,\pi),\quad V(t,\xi,v):=cX_1(t)^{\rm T}Q(\varepsilon t)X_1(t)-\tfrac{1}{2}\langle v,A(\varepsilon t)v\rangle,
    \end{equation*}
    where $X_1(t)$ is a vector, defined as
    \begin{equation*}
        X_1(t)=(v_1(t),\cdots,v_m(t),\xi_1,\cdots,\xi_m)^{\rm T},
    \end{equation*}
    and 
    \begin{equation*}
        v_j(t)=\langle v(\cdot), e_j(\varepsilon t,\cdot)\rangle,\;j\in\N^+, \quad \xi=(\xi_1,\cdots,\xi_m)^{\rm T}.
    \end{equation*}
    In particular, one has
    \begin{equation*}
        V(t,\xi,v)=cX_1(t)^{\rm T}Q(\varepsilon t)X_1(t)-\frac{1}{2}\sum_{j\in\N^+}\lambda_j(\varepsilon t)v_j(t)^2.
    \end{equation*}

    Let $\|\cdot\|_{\R^{2m}}$ stand for the Euclidean norm in $\R^{2m}$. Recall that $Q(\tau)$ is symmetric  positive definite, and the eigenvalues of $A(\tau)$, except the $m$ first ones, are less than $-\eta<0$. Therefore, there exists $c_0>0$ sufficiently large, such that for any $c\geq c_0$, there exists $C>1$ satisfying
    \begin{equation}\label{eq V estimate}
        C^{-1}\left(\|X_1(t)\|_{\R^{2m}}^2-\sum_{j>m}\lambda_j(\varepsilon t)v_j(t)^2\right)\leq V(t,\xi,v)\leq C\left(\|X_1(t)\|_{\R^{2m}}^2-\sum_{j>m}\lambda_j(\varepsilon t)v_j(t)^2\right)
    \end{equation}
    for any $\varepsilon\in(0,1]$, $t\in(0,\varepsilon^{-1}]$, $\xi\in\R^m$ and $v\in H_0^1(0,\pi)\cap H^2(0,\pi)$, where 
    \begin{equation*}
        \|X_1(t)\|_{\R^{2m}}^2=\sum_{1\leq j\leq m}\xi_j^2+\sum_{1\leq j\leq m}v_j(t)^2.
    \end{equation*}
    We further have the following lemma, whose proof is placed in Appendix~\ref{Sec E}.
    \begin{lemma}\label{lemma L-functional}
        There exists $c_0>0$ such that for any $c\geq c_0$, there exists $C>1$ satisfying
        \begin{equation}\label{eq V estimate1}
             C^{-1}\left(\|\xi\|_{\R^{2m}}^2+\|v\|_{H^1}^2\right)\leq V(t,\xi,v)\leq C\left(\|\xi\|_{\R^{2m}}^2+\|v\|_{H^1}^2\right)
        \end{equation}
        and 
        \begin{equation}\label{eq V estimate2} V(t,\xi,v)\leq C\left(\|X_1(t)\|_{\R^{2m}}^2+\|A(\varepsilon t)v\|^2\right)
        \end{equation}
        for any $\varepsilon\in(0,1]$, $t\in(0,\varepsilon^{-1}]$, $\xi\in\R^{2m}$ and $v\in H_0^1(0,\pi)\cap H^2(0,\pi)$.
    \end{lemma}

    \vspace{0.6em}

    It remains to construct suitable control $\zeta$ to derive \eqref{eq z stabilize}. Recall that $z$ denotes the solution of \eqref{z equation2}, and the control $\zeta$ is of form of \eqref{eq zeta def2}. Taking Lemma \ref{lemma Kalman} into account, let us choose the control in the feedback form, i.e.,
    \begin{equation*}
        (\alpha_1(t),\cdots,\alpha_m(t))^{\rm T}=K(\varepsilon t)X(t),\quad 
        \hat{\zeta}'=(\zeta_1',\cdots,\zeta_m')^{\rm T}=(\alpha_1,\cdots,\alpha_m)^{\rm T},\quad 
        \hat{\zeta}(0)=0.
    \end{equation*}
    We set
    \begin{equation}\label{eq hat V def}
        \hat{V}(t)=V(t,\hat{\zeta}(t),z(t,\cdot))=cX(t)^{\rm T}Q(\varepsilon t)X(t)-\tfrac{1}{2}\langle z(t,\cdot), A(\varepsilon t)z(t,\cdot)\rangle.
    \end{equation}

    Differentiating $\hat{V}$ with respect to $t$ and using the definitions,  we compute that
    \begin{align*}
       \hat{V}'&=-c\|X\|^2-\|Az\|^2-\langle Az,\zeta\rangle-\langle Az,r\rangle+c(R^{\rm T}QX+X^{\rm T}QR)\\
       &\quad+\varepsilon cX^{\rm T}Q'X+3\varepsilon\langle z,\bar{y}\partial_\tau\bar{y}z\rangle.
    \end{align*}

    Next we specify the choice of the parameter $c$. From the definition of $\zeta$ and $\hat{V}$, i.e. \eqref{eq zeta def2}, \eqref{eq hat V def}, there exists a generic constant $c_1>0$ such that
    \begin{equation*}
        |\langle Az,\zeta\rangle|\leq\tfrac{1}{2}\|Az\|^2+\|\zeta\|^2\leq \tfrac{1}{2}\|Az\|^2+c_1\|\hat{\zeta}\|^2_{\R^{2m}}\leq\tfrac{1}{2}\|Az\|^2+c_1\|X\|^2_{\R^{2m}}.
    \end{equation*}
    We now fix any $c>c_0+c_1$, and denote $\hat{c}=c-c_1>0$.

    In what follows, we use the letter $C$ to denote generic constants, which are independent of $\varepsilon\in(0,1]$, $t\in[0,\varepsilon^{-1}]$, and may vary from line to line. We have the following estimates:
    \begin{itemize}[leftmargin=2em]
        \item[\tiny$\bullet$] By Lemma \ref{lemma Kalman}, the mapping $\tau\mapsto Q(\tau)$ is $C^1$-smooth, thus
        \begin{equation*}
            |\varepsilon cX^{\rm T}Q'X|\leq C\varepsilon\|X\|^2_{\R^{2m}}\leq C\varepsilon \hat{V}.
        \end{equation*}
        \item[\tiny$\bullet$]  Recall that $\bar{y}\in H^1(0,1;H_0^1(0,\pi))\cap H^2(0,\pi)$, and hence
        \begin{equation*}
           |3\varepsilon\langle z,\bar{y}\partial_\tau\bar{y}z\rangle|\leq C\varepsilon\|z\|^2_{L^\infty(0,\pi)}\leq C\varepsilon\|z\|_{H^1}^2\leq C\varepsilon \hat{V}.
        \end{equation*}
        
        \item[\tiny$\bullet$] The remainder $r(\varepsilon,t,\cdot)$ is given by \eqref{eq r def}. We have
        \begin{align*}
            |\langle Az,r\rangle|=|\langle Az,z^2(z+3\bar{y})+\varepsilon\partial_\tau\bar{y}\rangle|\leq C\|Az\|\left(\|z\|_{L^\infty(0,\pi)}^3+\|z\|_{L^\infty(0,\pi)}^2\right)+C\varepsilon\|Az\|.
        \end{align*}
         Since $H^1(0,\pi)$ is continuously embedded in $C([0,\pi])$, for $\|z(t,\cdot)\|_{L^\infty(0,\pi)}\leq 1$, we derive 
         \begin{align*}
            |\langle Az,r\rangle|\leq C\left(\|Az\|\|z\|_{L^\infty(0,\pi)}^2+\varepsilon\|Az\|\right)\leq C\|Az\|(\hat{V}+\varepsilon).
        \end{align*}

        \item[\tiny$\bullet$] Notice that $R(\varepsilon,t)$ is given by \eqref{eq R def} and
        \begin{equation*}
            \hat{r}_j(\varepsilon,t)=\langle r(t,\cdot),e_j(\varepsilon t,\cdot)\rangle+\varepsilon\langle z(t,\cdot),\partial_\tau e_j(\varepsilon t,\cdot)\rangle,\quad 1\leq j \leq m.
        \end{equation*}
        We compute that 
        \begin{equation*}
            |\hat{r}_j(\varepsilon,t)|\leq C\left(\varepsilon+\|z(t,\cdot)\|_{L^\infty(0,\pi)}^2+\|z(t,\cdot)\|_{L^\infty(0,\pi)}^3\right),
        \end{equation*}
        which implies that, for $\|z(t,\cdot)\|_{L^\infty(0,\pi)}\leq 1$, 
        \begin{equation*}
            \|R(\varepsilon,t)\|_{L^\infty(\R^{2m})}\leq C\left(\varepsilon+\|z(t,\cdot)\|_{H^1}^2\right).
        \end{equation*}
        Therefore, for $\|z(t,\cdot)\|_{L^\infty(0,\pi)}\leq 1$, it follows that
        \begin{equation*}
            |c(R^{\rm T}QX+X^{\rm T}QR)|\leq C\|X\|_{\R^{2m}}\left(\varepsilon+\|z\|_{H^1}^2\right)\leq C\left(\varepsilon \hat{V}^{1/2}+\hat{V}^{3/2}\right).
        \end{equation*}
    \end{itemize}

    In summary, collecting these estimates, we arrive at, for $\|z(t,\cdot)\|_{L^\infty(0,\pi)}\leq 1$,  
     \begin{align*}
       \hat{V}'+\hat{c}\|X\|^2+\tfrac{1}{2}\|Az\|^2&\leq C\left(\varepsilon \hat{V}^{1/2}+\hat{V}^{3/2}+\varepsilon \|Az\|+\hat{V}\|Az\|\right).
    \end{align*}
    Note that, for any $\theta>0$,
    \begin{align*}
        \varepsilon \hat{V}^{1/2}\leq \frac{\theta}{2}\hat{V}+\frac{1}{2\theta}\varepsilon^2,\quad 
        \varepsilon \|Az\|\leq \frac{\theta}{2}\|Az\|^2+\frac{1}{2\theta}\varepsilon^2,\quad
        \hat{V} \|Az\|\leq \frac{\theta}{2}\|Az\|^2+\frac{1}{2\theta}\hat{V}^2.
    \end{align*}

    Hence, taking Lemma \ref{lemma L-functional} into account and letting $\theta>0$ be sufficiently small, there exists a constant $\hat{C}>1$ such that, for $\|z(t,\cdot)\|_{L^\infty(0,\pi)}\leq 1$,  
    \begin{equation*}
        \hat{V}'+\hat{C}^{-1}\hat{V}\leq \hat{C}\left(\hat{V}^2+\hat{V}^{3/2}+\varepsilon^2\right).
    \end{equation*}
    In particular, there exists $\sigma>0$ such that for any $\varepsilon\in(0,1]$, $t\in[0,\varepsilon^{-1}]$,  $\hat{V}(t)\leq \sigma$ and $\|z(t,\cdot)\|_{L^\infty(0,\pi)}\leq 1$, 
    \begin{equation*}
        \hat{V}'\leq \hat{C}\varepsilon^2.
    \end{equation*}

    Meanwhile, recall that $\hat{V}(0)=0$ and \eqref{eq V estimate1}. Thus for any
    \begin{equation*}
        0<\varepsilon\leq \min\left\{\sigma/\hat{C},\sigma/\tilde{C}\right\}:=\hat{\varepsilon},
    \end{equation*}
    where $\tilde{C}>0$ is a constant such that $\|z(t,\cdot)\|_{L^\infty(0,\pi)}\leq \tilde{C}\hat{V}(t)$, we obtain
    \begin{equation*}
        \hat{V}(t)\leq \hat{C}\varepsilon\quad\forall\,t\in[0,\varepsilon^{-1}].
    \end{equation*}
    
    Consequently, we conclude that
    \begin{equation*}
        \|z(\varepsilon^{-1})\|_{H^1}^2\leq C\hat{V}(\varepsilon^{-1})\leq C\varepsilon,
    \end{equation*}
    where the constant $C>0$ does not depend on $\varepsilon\in(0,\hat{\varepsilon}]$.

    \vspace{0.3em}

    This completes the proof of \eqref{eq z stabilize}, and thereby, the proof of Proposition \ref{prop steady control}.

    \end{proof}

    \subsubsection{Step 2: Convergence to the steady-states}\label{Sec 4.1.2}

   This subsection collects some dynamical properties for the unforced equation \eqref{AC unforced equation}.  Lemma \ref{lemma pointwise} is used to derive global steady-state controllability, i.e. Theorem \ref{thm null-control}; while Lemma \ref{lemma uniform} is applied to establish Theorem \ref{thm irre-control}.

    \begin{lemma}\label{lemma pointwise}
       Let $\nu,\lambda>0$ be arbitrary. Then for any $u_0\in H$, there exists $\phi_{u_0}\in\mathcal{S}$ such that the solution $\phi$ of equation \eqref{AC unforced equation} satisfies that
        \begin{equation}\label{eq uniform3}
            \lim\limits_{t\rightarrow\infty}\|\phi(t,u_0)-\phi_{u_0}\|_{H^1}=0.
        \end{equation}
    \end{lemma}

    \begin{proof}[Proof]
        Let us first check that $\phi(t,\cdot)\colon H_0^1(0,\pi)\rightarrow H_0^1(0,\pi)$, $t\geq 0$, is a gradient system (see \cite[Definition 3.8.1]{Hale-88}). To this end, in view of the parabolic regularization, it suffices to construct a suitable Lyapunov function. We take
    \begin{equation*}
        V\colon H_0^1(0,\pi)\rightarrow\R,\quad V(\phi)=\int_0^\pi\left(\frac{\nu}{2}|\phi'|^2-\frac{\lambda}{2}\phi^2+\frac{1}{4}\phi^4\right)dx.
    \end{equation*}
    Note that $V$ is continuous in $H_0^1(0,\pi)$ and bounded below. Moreover,
    \begin{equation*}
        \partial_t V(\phi(t,u_0))=-\|\partial_t\phi(t,u_0)\|^2\leq 0,
    \end{equation*}
    which ensures that $V(\phi(t,u_0))$ is strictly decreasing unless $u_0\in\mathcal{S}$.  Thus $\phi(t,\cdot)$ is a gradient system. 

     In particular, by \cite[Section 5.3]{Henry-81}, for any $u_0\in H_0^1(0,\pi)$, the $\omega$-limit set $\omega(u_0)\subset \mathcal{S}$. Since $\omega(u_0)$ is connected and $\mathcal{S}$ is finite, there exists $\phi_{u_0}\in\mathcal{S}$ such that $\omega(u_0)=\{\phi_{u_0}\}$. We claim that this implies \eqref{eq uniform3}. Otherwise there exists $t_n\rightarrow\infty$ and $\eta>0$ such that
     \begin{equation}\label{eq uniform4}
         \|\phi(t_n,u_0)-\phi_{u_0}\|_{H^1}\geq \eta\quad \forall\,n\in\N^+.
     \end{equation}
    On the other hand, the precompactness of the orbit ensures the existence of a convergent subsequence $\phi(t_{n_k},u_0)\rightarrow\psi\in\omega(u_0)=\{\phi_{u_0}\}$. This gives $\psi=\phi_{u_0}$, contradicting to \eqref{eq uniform4}. Consequently, in view of the parabolic smoothing effect, we complete the proof of Lemma \ref{lemma pointwise}.
    \end{proof}

    \begin{lemma}\label{lemma uniform}
       Let $\nu,\lambda>0$ be arbitrary. Then for any $\varepsilon>0$, there exists a constant $T=T(\varepsilon)>0$ with the following property: for any $u_0\in H$, there exists $\phi_{u_0}\in\mathcal{S}$ and $t_{u_0}\in[1,T+1]$ such that the solution $\phi$ of equation \eqref{AC unforced equation} satisfies that
       \begin{equation}\label{eq uniform}
           \|\phi(t_{u_0},u_0)-\phi_{u_0}\|_{H^1}\leq \varepsilon.
       \end{equation}
    \end{lemma}

    \begin{proof}[Proof]

     \noindent {\it Part 1: Uniform bounds.}
         We first show that $\phi(t,u_0)$ for $t\geq 1$ is uniformly bounded in $H^2(0,\pi)$, i.e., there exists a constant $R=R(\nu,\lambda)$ such that
         \begin{equation}\label{eq uniform1}
             \|\phi(t,u_0)\|_{H^2}\leq R\quad\forall\,u_0\in H,\;t\geq 1.
         \end{equation}
         In view of the parabolic smoothing effect, it suffices to show that $\phi(\tfrac{1}{2},u_0)$ is uniformly bounded in $H$. To this end, using Young's inequality, one has
     \begin{align*}
            \tfrac{1}{2}\partial_t\|\phi\|^2=-\nu\|\phi\|_{H^1}^2+\lambda\|\phi\|^2-\|\phi^2\|^2\leq -\nu\|\phi\|_{H^1}^2-\tfrac{1}{2}\|\phi^2\|^2+2\pi\lambda^2\leq -\tfrac{1}{2\pi^2}\|\phi\|^4+2\pi\lambda^2.
        \end{align*}
    Thus we see that the function $g(t)=\|\phi(t)\|^2$ satisfies the differential inequality
    \begin{equation}\label{eq uniform2}
        \partial_t g+2c_1g^2\leq c_2
    \end{equation}
    with $c_1=\pi^{-2}/2$ and $c_2=4\pi\lambda^2$. It follows that, as long as $g(t)\geq K_1:=(c_2/c_1)^{1/2}$, then 
    \begin{equation*}
        \partial_t g+c_1g^2\leq 0.
    \end{equation*}
    Solving the  differential inequality \eqref{eq uniform2}, we obtain 
    \begin{equation*}
    g(t)\leq \left(g(s)^{-1}+c_3(t-s)\right)^{-1}\quad \text{for }0\leq s\leq t,
    \end{equation*}
    where $c_3>0$ is a constant. Thus if $g(0)\leq K_1$, then $g(t)\leq K_1$ for any $t\geq 0$. Conversely, if $g(0)>K_1$, then
    \begin{equation*}
        g(t)\leq \left(g(0)^{-1}+c_3t\right)^{-1}\quad \text{for }0\leq t\leq t_*,
    \end{equation*}
    where $t_*>0$ is the first instant $t>0$ such that $g(t)=K_1$. Now taking $K_2=2c_3^{-1}$, we thus derive that
    \begin{equation*}
        \|\phi(t,u_0)\|\leq K:=\max\left\{K_1^{1/2},K_2^{1/2}\right\}\quad\forall\, u_0\in H,\; t\geq 1/2.
    \end{equation*}

    \vspace{0.6em}

     \noindent {\it Part 2: Proof of \eqref{eq uniform}.} In view of \eqref{eq uniform1}, it suffices to
     show that for any $\varepsilon>0$, there exists $T=T(\varepsilon)>0$ such that for any $u_0\in B_{H^2}(R)$, there exists $\phi_{u_0}\in\mathcal{S}$ and $t_{u_0}\in[0,T]$ satisfying
       \begin{equation}\label{eq uniform5}
           \|\phi(t_{u_0},u_0)-\phi_{u_0}\|_{H^1}\leq \varepsilon.
       \end{equation}
     Using the pointwise convergence \eqref{eq uniform3}, for any $u_0\in B_{H^2}(R)$, there exists $s_{u_0}>0$ such that
     \begin{equation*}
         \|\phi(s_{u_0},u_0)-\phi_{u_0}\|_{H^1}< \varepsilon/2.
     \end{equation*}
    Then by the continuity of $v\mapsto \phi(s_{u_0},v)$ in the $H^1$-topology, there exists $\delta_{u_0}>0$ such that
    \begin{equation*}
         \|\phi(s_{u_0},v)-\phi_{u_0}\|_{H^1}<\varepsilon\quad\forall\, v\in B_{H^1}^\circ(u_0,\delta_{u_0}),
     \end{equation*}
    where $B_{X}^\circ(x,r)$ denotes the open ball in $X$ centered at $x$ with radius $r$. Therefore, the family $\{B_{H^1}^\circ(u_0,\delta_{u_0})\}_{u_0\in B_{H^2}(R)}$ forms an open cover of the compact set $B_{H^2}(R)$. This implies that there exists $N\in\N^+$, $\{u_k\}_{1\leq k\leq N}\subset B_{H^2}(R)$ such that
    \begin{equation*}
        B_{H^2}(R)\subset \bigcup_{1\leq k\leq N} B_{H^1}(u_j,\delta_{u_j}),
    \end{equation*}
    which ensures \eqref{eq uniform5} with $T:=\max\{s_{u_k}:1\leq k\leq N\}$. This completes the proof of Lemma~\ref{lemma uniform}.

    \end{proof}

    \subsubsection{Step 3: Local exact controllability and the proof of Theorem \ref{thm null-control}}\label{Sec 4.1.3} In view of the dynamics of the unforced system, i.e. Lemma~\ref{lemma pointwise}, and the approximate controllability between steady-states, i.e. Proposition~\ref{prop steady control}, the global steady-state controllability would then follow from the following local result.

    \begin{lemma}\label{lemma local control}
    For any $\nu,\lambda>0$ and $0\leq a<b\leq \pi$, the local steady-state controllability holds: for any $T>0$, there exists $\delta>0$ such that for any $\phi\in\mathcal{S}$ and $u_0\in B_{H^1}(\phi,\delta)$, there exists $\zeta\in L^2(0,T;L^2(a,b))$ such that the solution $w$ of equation \eqref{AC determine equation} satisfies that
    \begin{equation*}
        w(T,u_0,\zeta)=\phi.
    \end{equation*}
    
    \end{lemma}

    \begin{proof}[Proof]
        The proof follows the idea of \cite[Section 2.4]{CT-04}, which invokes \cite[Theorem 3.1]{Emanuilov-95} to construct the required controls. Let $T>0$, $\phi\in\mathcal{S}$ and $u_0\in B_{H^1}(\phi,\delta)$ be given, where $\delta>0$ will be specified later. Recall that $\mathcal{S}$ is a finite set. Define $f\in C^1(\R)$ by a globally Lipschitz function such that
        \begin{equation}\label{eq local1}
            f(s)=s^3-\lambda s\quad\forall\,s\in\left[-\|\phi\|_{L^\infty(0,\pi)}-1,\|\phi\|_{L^\infty(0,\pi)}+1\right].
        \end{equation}
        Then by \cite[Theorem 3.1]{Emanuilov-95}, there exists $\zeta\in L^2(0,T;L^2(a,b))$ such that the solution of    
         \begin{equation*}
        \begin{cases}
         \partial_tv-\nu \partial_x^2v+ f(v+\phi)-f(\phi)=\zeta(t,x),\quad x\in (0,\pi),\;t\in(0,T),\\
        v(0,\cdot)=u_0(\cdot)-\phi(\cdot),
        \end{cases}
        \end{equation*}
        satisfies that
        \begin{equation*}
            v(T)=0.
        \end{equation*}
        Additionally, from the proof of \cite[Theorem 3.1]{Emanuilov-95}, there exists a constant $C>0$ such that 
        \begin{equation*}\label{eq local}
            \|v\|_{Y_T}\leq C\|u_0-\phi\|_{H^1},
        \end{equation*}
        where $ \|v\|_{Y_T}^2:=\|v\|_{L^2(0,T;H^2(0,\pi))}^2+\|\partial_tv\|_{L^2(0,T;H)}^2$. This implies that for $q=v+\bar{\phi}$ with $\bar{\phi}(t)\equiv\phi$, one has
        \begin{equation}\label{eq local3}
             \begin{cases}
         \partial_tq-\nu \partial_x^2q+ f(q)=\zeta(t,x),\quad x\in (0,\pi),\;t\in(0,T),\\
        q(0,\cdot)=u_0(\cdot),\\
        q(T,\cdot)=\phi(\cdot),
        \end{cases}
        \end{equation}
        and 
        \begin{equation}\label{eq local2}
            \|q-\bar{\phi}\|_{Y_T}\leq C\|u_0-\phi\|_{H^1}.
        \end{equation}

        In particular, \eqref{eq local2} implies that there exists a constant $\hat{C}>0$ such that
        \begin{equation*}
            \|q-\bar{\phi}\|_{L^\infty((0,T)\times (0,\pi))}\leq \hat{C}\|u_0-\phi\|_{H^1}\leq \hat{C}\delta.
        \end{equation*}
        Therefore, by taking $\delta>0$ be such that $\hat{C}\delta\leq 1$ and using the definitions \eqref{eq local1},\eqref{eq local3}, we conclude that
        \begin{equation*}
            f(q)=q^3-\lambda q,\quad w(T,u_0,\zeta)=q(T)=\phi,
        \end{equation*}
        completing the proof of Lemma \ref{lemma local control}.

    \end{proof}

    Consequently, this completes the proof of Theorem \ref{thm null-control}.

    \subsubsection{Step 4: Construction of finitely many controls and proof of Theorem~\ref{thm irre-control}}\label{Sec 4.1.4}

    The construction proceeds as follows. We first use Proposition~\ref{prop steady control} to select a finite family of controls $h_j$, which are then truncated to a common finite-dimensional space. The main issue is to make the construction uniform with respect to the initial condition $u_0$. Although each trajectory approaches a steady state, the hitting time is not explicit. By Lemma~\ref{lemma uniform}, there exists a deterministic time $T_{\delta_2}$, independent of $u_0$, such that every trajectory enters a small neighborhood of some steady state before time $T_{\delta_2}+1$.

    We then adopt a “waiting-time” construction: partition $[1,T_{\delta_2}+1]$into finitely many subintervals and activate one of the controls at each discrete time, keeping the control zero elsewhere. This ensures that, for any $u_0$, there exists a suitable activation time at which the trajectory is sufficiently close to a steady state. The local stability then drives the solution into the desired $\varepsilon$-neighborhood after time $\hat T$. See Figure~\ref{figure 6} for an illustration.

       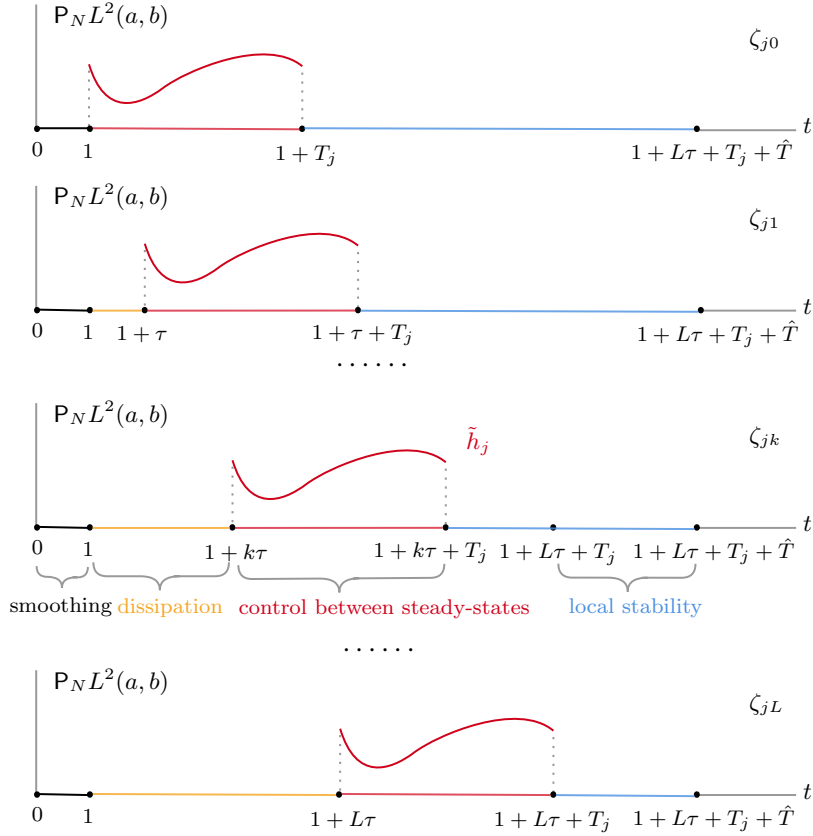
\begin{figure}[th]
    \centering

\tikzset{every picture/.style={line width=0.75pt}} 

\begin{tikzpicture}[x=0.75pt,y=0.75pt,yscale=-1,xscale=1]

\draw [color={rgb, 255:red, 208; green, 2; blue, 27 }  ,draw opacity=0.7 ]   (42.75,69) -- (148,69.54) ;
\draw  [line width=2.25] [line join = round][line cap = round] (16.17,69.14) .. controls (16.17,69.25) and (16.17,69.35) .. (16.17,69.45) ;
\draw  [color={rgb, 255:red, 0; green, 0; blue, 0 }  ,draw opacity=1 ][line width=2.25] [line join = round][line cap = round] (42.67,69.12) .. controls (42.67,69.22) and (42.67,69.32) .. (42.67,69.42) ;
\draw [color={rgb, 255:red, 208; green, 2; blue, 27 }  ,draw opacity=1 ]   (42.23,36.92) .. controls (50.23,63.45) and (67.23,57.92) .. (79.23,48.92) .. controls (91.23,39.92) and (131.23,22.92) .. (149.23,37.92) ;
\draw  [line width=2.25] [line join = round][line cap = round] (347.17,69.14) .. controls (347.17,69.25) and (347.17,69.35) .. (347.17,69.45) ;
\draw  [color={rgb, 255:red, 0; green, 0; blue, 0 }  ,draw opacity=1 ][line width=2.25] [line join = round][line cap = round] (149.17,69.12) .. controls (149.17,69.22) and (149.17,69.32) .. (149.17,69.42) ;
\draw [color={rgb, 255:red, 0; green, 0; blue, 0 }  ,draw opacity=0.4 ]   (15.75,69) -- (15.75,6.92) ;
\draw [line width=3] [line join = round][line cap = round]    ;
\draw [color={rgb, 255:red, 0; green, 0; blue, 0 }  ,draw opacity=0.4 ] [dash pattern={on 0.84pt off 2.51pt}]  (149.23,69) -- (149.23,37.92) ;
\draw [color={rgb, 255:red, 0; green, 0; blue, 0 }  ,draw opacity=0.4 ] [dash pattern={on 0.84pt off 2.51pt}]  (42.23,68) -- (42.23,36.92) ;
\draw [color={rgb, 255:red, 74; green, 144; blue, 226 }  ,draw opacity=0.8 ]   (150.25,69) -- (345.75,69.54) ;
\draw  [line width=2.25] [line join = round][line cap = round] (16.17,160.12) .. controls (16.17,160.23) and (16.17,160.33) .. (16.17,160.43) ;
\draw  [color={rgb, 255:red, 0; green, 0; blue, 0 }  ,draw opacity=1 ][line width=2.25] [line join = round][line cap = round] (70.17,160.1) .. controls (70.17,160.2) and (70.17,160.3) .. (70.17,160.4) ;
\draw  [color={rgb, 255:red, 0; green, 0; blue, 0 }  ,draw opacity=1 ][line width=2.25] [line join = round][line cap = round] (42.67,160.1) .. controls (42.67,160.2) and (42.67,160.3) .. (42.67,160.4) ;
\draw  [line width=2.25] [line join = round][line cap = round] (349.17,160.12) .. controls (349.17,160.23) and (349.17,160.33) .. (349.17,160.43) ;
\draw  [color={rgb, 255:red, 0; green, 0; blue, 0 }  ,draw opacity=1 ][line width=2.25] [line join = round][line cap = round] (177.17,160.1) .. controls (177.17,160.2) and (177.17,160.3) .. (177.17,160.4) ;
\draw [color={rgb, 255:red, 0; green, 0; blue, 0 }  ,draw opacity=0.4 ]   (15.75,159.98) -- (15.75,97.9) ;
\draw [color={rgb, 255:red, 245; green, 166; blue, 35 }  ,draw opacity=0.8 ]   (43,402.98) -- (168,403.52) ;
\draw  [line width=2.25] [line join = round][line cap = round] (16.17,403.12) .. controls (16.17,403.23) and (16.17,403.33) .. (16.17,403.43) ;
\draw  [color={rgb, 255:red, 0; green, 0; blue, 0 }  ,draw opacity=1 ][line width=2.25] [line join = round][line cap = round] (42.67,403.1) .. controls (42.67,403.2) and (42.67,403.3) .. (42.67,403.4) ;
\draw  [line width=2.25] [line join = round][line cap = round] (167.67,403.1) .. controls (167.67,403.2) and (167.67,403.3) .. (167.67,403.4) ;
\draw  [line width=2.25] [line join = round][line cap = round] (275.17,403.12) .. controls (275.17,403.23) and (275.17,403.33) .. (275.17,403.43) ;
\draw  [line width=2.25] [line join = round][line cap = round] (347.17,403.12) .. controls (347.17,403.23) and (347.17,403.33) .. (347.17,403.43) ;
\draw [color={rgb, 255:red, 0; green, 0; blue, 0 }  ,draw opacity=0.4 ]   (15.75,402.98) -- (15.75,340.9) ;
\draw  [color={rgb, 255:red, 0; green, 0; blue, 0 }  ,draw opacity=0.4 ] (16,288.67) .. controls (15.94,292.19) and (17.67,293.98) .. (21.19,294.04) -- (21.19,294.04) .. controls (26.23,294.13) and (28.72,295.93) .. (28.66,299.46) .. controls (28.72,295.93) and (31.27,294.22) .. (36.3,294.31)(34.03,294.27) -- (36.3,294.31) .. controls (39.82,294.37) and (41.61,292.64) .. (41.68,289.11) ;
\draw  [color={rgb, 255:red, 0; green, 0; blue, 0 }  ,draw opacity=0.4 ] (277.72,287.43) .. controls (277.71,292.1) and (280.04,294.43) .. (284.71,294.44) -- (302.5,294.46) .. controls (309.17,294.47) and (312.5,296.8) .. (312.49,301.47) .. controls (312.5,296.8) and (315.83,294.48) .. (322.5,294.49)(319.5,294.48) -- (339.48,294.51) .. controls (344.15,294.52) and (346.48,292.19) .. (346.49,287.52) ;
\draw  [color={rgb, 255:red, 0; green, 0; blue, 0 }  ,draw opacity=0.4 ] (117.16,288.19) .. controls (117.19,292.86) and (119.53,295.17) .. (124.2,295.14) -- (157.78,294.92) .. controls (164.45,294.88) and (167.8,297.19) .. (167.83,301.86) .. controls (167.8,297.19) and (171.11,294.84) .. (177.78,294.8)(174.78,294.82) -- (213.53,294.56) .. controls (218.2,294.53) and (220.52,292.19) .. (220.49,287.52) ;
\draw [color={rgb, 255:red, 0; green, 0; blue, 0 }  ,draw opacity=1 ]   (15.75,159.98) -- (43,160.52) ;
\draw [color={rgb, 255:red, 0; green, 0; blue, 0 }  ,draw opacity=1 ]   (15.75,69) -- (43,69) ;
\draw [color={rgb, 255:red, 208; green, 2; blue, 27 }  ,draw opacity=0.7 ]   (71.5,160.52) -- (176.1,160.46) ;
\draw [color={rgb, 255:red, 208; green, 2; blue, 27 }  ,draw opacity=0.7 ]   (169,402.98) -- (273.77,403.52) ;
\draw [color={rgb, 255:red, 74; green, 144; blue, 226 }  ,draw opacity=0.8 ]   (178.5,160.54) -- (348,161.08) ;
\draw [color={rgb, 255:red, 74; green, 144; blue, 226 }  ,draw opacity=0.8 ]   (276.5,403.54) -- (346,404.08) ;
\draw [color={rgb, 255:red, 0; green, 0; blue, 0 }  ,draw opacity=0.4 ]   (347.75,70) -- (397,70) ;
\draw [color={rgb, 255:red, 0; green, 0; blue, 0 }  ,draw opacity=0.4 ]   (348,160.5) -- (397,160.5) ;
\draw [color={rgb, 255:red, 0; green, 0; blue, 0 }  ,draw opacity=0.4 ]   (348.25,403.5) -- (397,403.5) ;
\draw [color={rgb, 255:red, 245; green, 166; blue, 35 }  ,draw opacity=0.8 ]   (44,160.5) -- (68.8,160.5) ;
\draw [color={rgb, 255:red, 0; green, 0; blue, 0 }  ,draw opacity=1 ]   (15.75,403) -- (43,403.54) ;
\draw  [line width=2.25] [line join = round][line cap = round] (16.17,269.12) .. controls (16.17,269.23) and (16.17,269.33) .. (16.17,269.43) ;
\draw  [color={rgb, 255:red, 0; green, 0; blue, 0 }  ,draw opacity=1 ][line width=2.25] [line join = round][line cap = round] (114.17,269.1) .. controls (114.17,269.2) and (114.17,269.3) .. (114.17,269.4) ;
\draw  [color={rgb, 255:red, 0; green, 0; blue, 0 }  ,draw opacity=1 ][line width=2.25] [line join = round][line cap = round] (42.67,269.1) .. controls (42.67,269.2) and (42.67,269.3) .. (42.67,269.4) ;
\draw  [line width=2.25] [line join = round][line cap = round] (275.14,269.12) .. controls (275.14,269.23) and (275.14,269.33) .. (275.14,269.43) ;
\draw  [line width=2.25] [line join = round][line cap = round] (347.17,269.12) .. controls (347.17,269.23) and (347.17,269.33) .. (347.17,269.43) ;
\draw  [color={rgb, 255:red, 0; green, 0; blue, 0 }  ,draw opacity=1 ][line width=2.25] [line join = round][line cap = round] (221.17,269.1) .. controls (221.17,269.2) and (221.17,269.3) .. (221.17,269.4) ;
\draw [color={rgb, 255:red, 0; green, 0; blue, 0 }  ,draw opacity=0.4 ]   (15.75,268.98) -- (15.75,206.9) ;
\draw [color={rgb, 255:red, 0; green, 0; blue, 0 }  ,draw opacity=1 ]   (15.75,268.98) -- (43,269.52) ;
\draw [color={rgb, 255:red, 208; green, 2; blue, 27 }  ,draw opacity=0.7 ]   (115.5,269.52) -- (219.62,269.46) ;
\draw [color={rgb, 255:red, 74; green, 144; blue, 226 }  ,draw opacity=0.8 ]   (222.5,269.54) -- (346,270.08) ;
\draw [color={rgb, 255:red, 0; green, 0; blue, 0 }  ,draw opacity=0.4 ]   (347.75,269.5) -- (397,269.5) ;
\draw [color={rgb, 255:red, 245; green, 166; blue, 35 }  ,draw opacity=0.8 ]   (44,269.52) -- (112.8,269.52) ;
\draw  [color={rgb, 255:red, 0; green, 0; blue, 0 }  ,draw opacity=0.4 ] (44.43,286.76) .. controls (44.43,291.43) and (46.76,293.76) .. (51.43,293.76) -- (67.78,293.76) .. controls (74.45,293.76) and (77.78,296.09) .. (77.78,300.76) .. controls (77.78,296.09) and (81.11,293.76) .. (87.78,293.76)(84.78,293.76) -- (105.43,293.76) .. controls (110.1,293.76) and (112.43,291.43) .. (112.43,286.76) ;
\draw [color={rgb, 255:red, 208; green, 2; blue, 27 }  ,draw opacity=1 ]   (70.23,126.92) .. controls (78.23,153.45) and (95.23,147.92) .. (107.23,138.92) .. controls (119.23,129.92) and (159.23,112.92) .. (177.23,127.92) ;
\draw [color={rgb, 255:red, 0; green, 0; blue, 0 }  ,draw opacity=0.4 ] [dash pattern={on 0.84pt off 2.51pt}]  (177.23,159) -- (177.23,127.92) ;
\draw [color={rgb, 255:red, 0; green, 0; blue, 0 }  ,draw opacity=0.4 ] [dash pattern={on 0.84pt off 2.51pt}]  (70.23,158) -- (70.23,126.92) ;
\draw [color={rgb, 255:red, 208; green, 2; blue, 27 }  ,draw opacity=1 ]   (114.23,235.58) .. controls (122.23,262.11) and (139.23,256.58) .. (151.23,247.58) .. controls (163.23,238.58) and (203.23,221.58) .. (221.23,236.58) ;
\draw [color={rgb, 255:red, 0; green, 0; blue, 0 }  ,draw opacity=0.4 ] [dash pattern={on 0.84pt off 2.51pt}]  (221.23,267.67) -- (221.23,236.58) ;
\draw [color={rgb, 255:red, 0; green, 0; blue, 0 }  ,draw opacity=0.4 ] [dash pattern={on 0.84pt off 2.51pt}]  (114.23,266.67) -- (114.23,235.58) ;
\draw [color={rgb, 255:red, 208; green, 2; blue, 27 }  ,draw opacity=1 ]   (168.23,370.25) .. controls (176.23,396.78) and (193.23,391.25) .. (205.23,382.25) .. controls (217.23,373.25) and (257.23,356.25) .. (275.23,371.25) ;
\draw [color={rgb, 255:red, 0; green, 0; blue, 0 }  ,draw opacity=0.4 ] [dash pattern={on 0.84pt off 2.51pt}]  (275.23,402.33) -- (275.23,371.25) ;
\draw [color={rgb, 255:red, 0; green, 0; blue, 0 }  ,draw opacity=0.4 ] [dash pattern={on 0.84pt off 2.51pt}]  (168.23,401.33) -- (168.23,370.25) ;

\draw (12,74.42) node [anchor=north west][inner sep=0.75pt]  [font=\scriptsize]  {$0$};
\draw (399,62.42) node [anchor=north west][inner sep=0.75pt]  [font=\footnotesize]  {$t$};
\draw (37.5,75.42) node [anchor=north west][inner sep=0.75pt]  [font=\scriptsize,color={rgb, 255:red, 0; green, 0; blue, 0 }  ,opacity=1 ]  {$1$};
\draw (312.5,73.04) node [anchor=north west][inner sep=0.75pt]  [font=\scriptsize,color={rgb, 255:red, 0; green, 0; blue, 0 }  ,opacity=1 ]  {$1+L\tau +T_{j} +\hat{T}$};
\draw (23,4.42) node [anchor=north west][inner sep=0.75pt]  [font=\footnotesize]  {$\mathsf P_{N} L^{2}( a,b)$};
\draw (372,17.42) node [anchor=north west][inner sep=0.75pt]  [font=\footnotesize]  {$\zeta _{j0}$};
\draw (12,165.4) node [anchor=north west][inner sep=0.75pt]  [font=\scriptsize]  {$0$};
\draw (399,153.4) node [anchor=north west][inner sep=0.75pt]  [font=\footnotesize]  {$t$};
\draw (37.5,166.4) node [anchor=north west][inner sep=0.75pt]  [font=\scriptsize,color={rgb, 255:red, 0; green, 0; blue, 0 }  ,opacity=1 ]  {$1$};
\draw (54.5,167.02) node [anchor=north west][inner sep=0.75pt]  [font=\scriptsize,color={rgb, 255:red, 0; green, 0; blue, 0 }  ,opacity=1 ]  {$1+\tau $};
\draw (315.5,164.02) node [anchor=north west][inner sep=0.75pt]  [font=\scriptsize,color={rgb, 255:red, 0; green, 0; blue, 0 }  ,opacity=1 ]  {$1+L\tau +T_{j} +\hat{T}$};
\draw (23,95.4) node [anchor=north west][inner sep=0.75pt]  [font=\footnotesize]  {$\mathsf P_{N} L^{2}( a,b)$};
\draw (372,107.9) node [anchor=north west][inner sep=0.75pt]  [font=\footnotesize]  {$\zeta _{j1}$};
\draw (167.83,326.35) node [anchor=north west][inner sep=0.75pt]  [font=\normalsize,color={rgb, 255:red, 0; green, 0; blue, 0 }  ,opacity=1 ]  {$\cdots \cdots $};
\draw (12,408.4) node [anchor=north west][inner sep=0.75pt]  [font=\scriptsize]  {$0$};
\draw (399,396.4) node [anchor=north west][inner sep=0.75pt]  [font=\footnotesize]  {$t$};
\draw (37.5,409.4) node [anchor=north west][inner sep=0.75pt]  [font=\scriptsize,color={rgb, 255:red, 0; green, 0; blue, 0 }  ,opacity=1 ]  {$1$};
\draw (151.5,410.02) node [anchor=north west][inner sep=0.75pt]  [font=\scriptsize,color={rgb, 255:red, 0; green, 0; blue, 0 }  ,opacity=1 ]  {$1+L\tau $};
\draw (243.83,410.02) node [anchor=north west][inner sep=0.75pt]  [font=\scriptsize,color={rgb, 255:red, 0; green, 0; blue, 0 }  ,opacity=1 ]  {$1+L\tau +T_{j}$};
\draw (313.5,407.02) node [anchor=north west][inner sep=0.75pt]  [font=\scriptsize,color={rgb, 255:red, 0; green, 0; blue, 0 }  ,opacity=1 ]  {$1+L\tau +T_{j} +\hat{T}$};
\draw (23,338.4) node [anchor=north west][inner sep=0.75pt]  [font=\footnotesize]  {$\mathsf P_{N} L^{2}( a,b)$};
\draw (372,350.9) node [anchor=north west][inner sep=0.75pt]  [font=\footnotesize]  {$\zeta _{jL}$};
\draw (26.24,309.23) node   [align=left] {\begin{minipage}[lt]{35.69pt}\setlength\topsep{0pt}
{\scriptsize smoothing}
\end{minipage}};
\draw (317.6,309.49) node  [color={rgb, 255:red, 74; green, 144; blue, 226 }  ,opacity=1 ] [align=left] {\begin{minipage}[lt]{52.04pt}\setlength\topsep{0pt}
{\scriptsize local stability}
\end{minipage}};
\draw (205.97,310.08) node  [color={rgb, 255:red, 208; green, 2; blue, 27 }  ,opacity=1 ] [align=left] {\begin{minipage}[lt]{132.85pt}\setlength\topsep{0pt}
{\scriptsize control between steady-states}
\end{minipage}};
\draw (230,217.39) node [anchor=north west][inner sep=0.75pt]  [font=\footnotesize,color={rgb, 255:red, 208; green, 2; blue, 27 }  ,opacity=1 ]  {$\tilde{h}_{j}$};
\draw (134.17,76.04) node [anchor=north west][inner sep=0.75pt]  [font=\scriptsize,color={rgb, 255:red, 0; green, 0; blue, 0 }  ,opacity=1 ]  {$1+T_{j}$};
\draw (152.5,166.04) node [anchor=north west][inner sep=0.75pt]  [font=\scriptsize,color={rgb, 255:red, 0; green, 0; blue, 0 }  ,opacity=1 ]  {$1+\tau +T_{j}$};
\draw (12,274.4) node [anchor=north west][inner sep=0.75pt]  [font=\scriptsize]  {$0$};
\draw (399,262.4) node [anchor=north west][inner sep=0.75pt]  [font=\footnotesize]  {$t$};
\draw (37.5,275.4) node [anchor=north west][inner sep=0.75pt]  [font=\scriptsize,color={rgb, 255:red, 0; green, 0; blue, 0 }  ,opacity=1 ]  {$1$};
\draw (98.5,276.02) node [anchor=north west][inner sep=0.75pt]  [font=\scriptsize,color={rgb, 255:red, 0; green, 0; blue, 0 }  ,opacity=1 ]  {$1+k\tau $};
\draw (246.5,276.02) node [anchor=north west][inner sep=0.75pt]  [font=\scriptsize,color={rgb, 255:red, 0; green, 0; blue, 0 }  ,opacity=1 ]  {$1+L\tau +T_{j}$};
\draw (313.5,273.02) node [anchor=north west][inner sep=0.75pt]  [font=\scriptsize,color={rgb, 255:red, 0; green, 0; blue, 0 }  ,opacity=1 ]  {$1+L\tau +T_{j} +\hat{T}$};
\draw (23,204.4) node [anchor=north west][inner sep=0.75pt]  [font=\footnotesize]  {$\mathsf P_{N} L^{2}( a,b)$};
\draw (372,216.9) node [anchor=north west][inner sep=0.75pt]  [font=\footnotesize]  {$\zeta _{jk}$};
\draw (183.5,275.04) node [anchor=north west][inner sep=0.75pt]  [font=\scriptsize,color={rgb, 255:red, 0; green, 0; blue, 0 }  ,opacity=1 ]  {$1+k\tau +T_{j}$};
\draw (164.5,183.52) node [anchor=north west][inner sep=0.75pt]  [font=\normalsize,color={rgb, 255:red, 0; green, 0; blue, 0 }  ,opacity=1 ]  {$\cdots \cdots $};
\draw (80.24,309.5) node  [color={rgb, 255:red, 245; green, 166; blue, 35 }  ,opacity=1 ] [align=left] {\begin{minipage}[lt]{35.69pt}\setlength\topsep{0pt}
{\scriptsize dissipation}
\end{minipage}};

\end{tikzpicture}
      
    \vspace{-0.6em}
    \caption{Construction of finitely many controls.}
        \label{figure 6}
    \end{figure}

    \vspace{0.6em}
    Let us recall some properties for the steady-states of equation \eqref{AC unforced equation}.  By \cite[Section 5]{Henry-81}, the steady-states $\phi_{\pm 1}$ are locally exponentially stable in the following sense: there exist constants $C_1,C_2>0$ and $\delta_0\in(0,1)$ such that
    \begin{equation*}
        \|\phi(t,u_0)-\phi_{\pm 1}\|_{H^1}\leq C_1e^{-C_2t}\|u_0-\phi_{\pm 1}\|_{H^1}\quad\forall\,u_0\in B_{H^1}(\phi_{\pm 1},\delta_0),\;t\geq 0.
    \end{equation*}
   Additionally, using the parabolic smoothing effect, there exist constants $C_3,C_4>0$ and $\delta\in(0,1)$
    \begin{equation}\label{eq stable}
        \|\phi(t,u_0)-\phi_{\pm 1}\|_{H^1}\leq C_3e^{-C_4t}\|u_0-\phi_{\pm 1}\|\quad\forall\,u_0\in B_{H}(\phi_{\pm 1},\delta),\;t\geq 1.
    \end{equation}
    
    \vspace{0.3em}

        Let $u_*\in\{\phi_1,\phi_{-1}\}$ be fixed. Next we shall apply Proposition \ref{prop steady control} to send each initial state $u_0$ to the $\delta$-neighborhood of $u_*$ within a finite family of controls. For any $\phi_j\in\mathcal{S}$, there exists $T_j>0$ and $h_j\in H^1(0,T_j;C^2([a,b]))$ such that  
    \begin{equation}\label{eq uniform13}
            \|w(T_j,\phi_j,h_j)-u_*\|\leq \delta/4.
    \end{equation}
    In view of the continuity $v\mapsto w(T_j,v,h_j)$, there exists $\delta_1=\delta_1(\delta)>0$ such that     \begin{equation*}
        \|w(T_{j},v,h_j)-u_*\|\leq\delta/2\quad\forall\, v\in B_{H}(\phi_j,\delta_1),\;0\leq |j|\leq \lfloor\lambda/\nu\rfloor.
    \end{equation*}
    
    Thus by performing a standard truncation procedure, there exists $N\in\N^+$ such that for each $\tilde{h}_j=\mathsf{P}_Nh_j\in H^1(0,T_{j};\mathsf{P}_NL^2(a,b))$, one has
    \begin{equation}\label{eq uniform10}
        \|w(T_{j},v,\tilde{h}_j)-u_*\|\leq\delta\quad\forall\, v\in B_{H}(\phi_j,\delta_1),\;0\leq |j|\leq \lfloor\lambda/\nu\rfloor,
    \end{equation}
    
    \vspace{0.3em}

    Additionally, using the continuity of the mapping $(s,v)\mapsto \phi(s,v)$, there exists $\tau=\tau(\delta)>0$ and $\delta_2=\delta_2(\delta)>0$ such that
    \begin{equation}\label{eq uniform11}
        \|\phi(t,v)-\phi_j\|\leq \delta_1/4\quad\forall\, v\in B_{H}(\phi_j,\delta_2),\; 0\leq t\leq \tau.
    \end{equation}

    For such $\delta_2>0$, Lemma \ref{lemma uniform} ensures a constant $T_{\delta_2}>0$ such that for any $u_0\in H$, there exists $\phi_{u_0}\in\mathcal{S}$ and $t_{u_0}\in[1,T_{\delta_2}+1]$ satisfying
       \begin{equation}\label{eq uniform14}
           \|\phi(t_{u_0},u_0)-\phi_{u_0}\|_{H}\leq \delta_2.
       \end{equation}

    \vspace{0.6em}

    With these preparations, we now set 
    \begin{equation*}
        T_*:=1+T_{\delta_2}+\max\left\{T_{j}:0\leq |j|\leq \lfloor\lambda/\nu\rfloor\right\},
    \end{equation*}
    which depends only on the parameters $\nu,\lambda,a,b$. We also set
    \begin{equation*}
        L:=\lceil T_{\delta_2}/\tau\rceil\in\N^+,
    \end{equation*}
    and construct controls $\zeta_{lj}\in L^2(0,T;\mathsf{P}_NL^2(a,b)) $ in the following way. For each $1\leq l\leq L$ and $|j|\leq \lfloor\lambda/\nu\rfloor$, define
    \begin{equation*}
        \zeta_{lj}(t)=\begin{cases}
            0\quad&\text{for }0\leq t\leq 1+l\tau,\\
            \tilde{h}_{j}(t-l\tau)\quad&\text{for }1+l\tau< t\leq 1+l\tau+T_{j},\\
            0\quad&\text{for }1+l\tau+T_{j}<t\leq T_*.
        \end{cases}
    \end{equation*}

    \vspace{0.3em}

    With such finite family of controls $\zeta_{lj}$ at hand, we are ready to verify the desired property \eqref{eq control}. Let $\varepsilon>0$ be arbitrarily given. Recall that the state state $u_*$ is locally exponentially stable in the sense of \eqref{eq stable}. Let $\hat T=\hat T(\varepsilon)\geq 1$ be such that
    \begin{equation}\label{eq T1}
        C_3 e^{-C_4t}\leq \varepsilon\quad\forall\, t\geq \hat T,
    \end{equation}
    and set
    \begin{equation}\label{eq T2}
        T=T(\varepsilon)=T_*+T_0.
    \end{equation}

    Then for any $u_0\in H$, in view of \eqref{eq uniform14}, for
    \begin{equation*}
        s_{u_0}:=\inf\{t\geq 0:\|\phi(t+1,u_0)-\phi_{u_0}\|\leq \delta_2\},
    \end{equation*}
    one has $s_{u_0}\leq T_{\delta_2}$ by definition. Thus there exists $j_{u_0}$ such that 
    \begin{equation*}
        \phi(s_{u_0},u_0)\in B_{H}(\phi_{j_{u_0}},\delta_2).
    \end{equation*}

   We hence claim that inequality \eqref{eq control} is satisfied by taking
    \begin{equation*}
        \zeta_n:=\zeta_{\lceil s_{u_0}/\tau\rceil,\,j_{u_0}}.
    \end{equation*}

    \vspace{0.3em}
    To verify this, by the definition of $\tilde{h}_{lj}$, it follows that
    \begin{equation*}
        w\left(s_{u_0},u_0,\zeta_n|_{\left[0,s_{u_0}+1\right]}\right)=\phi(s_{u_0},u_0)\in B_{H}(\phi_{j_{u_0}},\delta_2).
    \end{equation*}
    In addition, 
    \begin{equation*}
        |s_{u_0}-\lceil s_{u_0}/\tau\rceil\tau|\leq \tau.
    \end{equation*}
    Thus by \eqref{eq uniform11}, one has
    \begin{equation*}
        w\left(1+\lceil  s_{u_0}/\tau\rceil\tau,u_0,\zeta_n|_{\left[0,1+\lceil s_{u_0}/\tau\rceil\tau\right]}\right)\in B_{H}(\phi_{j_{u_0}},\delta_1/4).
    \end{equation*}
    Consequently, using \eqref{eq uniform10}, we derive that
    \begin{equation*}
        w\left(1+\lceil  s_{u_0}/\tau\rceil\tau+T_{j_{u_0}},u_0,\zeta_n|_{\left[0,1+\lceil s_{u_0}/\tau\rceil\tau+T_{j_{u_0}}\right]}\right)\in B_{H}(u_*,\delta).
    \end{equation*}

    Finally, by applying the stability \eqref{eq stable} and the definition of $T$ by \eqref{eq T1},\eqref{eq T2}, we conclude that the desired relation \eqref{eq control} holds, thereby completing the proof of Theorem \ref{thm irre-control}.

    \subsection{Irreducibility}\label{Sec 4.2}

    This section is devoted to the proof of irreducibility, i.e. Proposition \ref{prop irreducibility}. The proof combines a support lemma, i.e. Lemma \ref{lemma support}, with the global controllability within a finite family of controls, i.e. Theorem \ref{thm irre-control}.  The proof of Lemma \ref{lemma support} is placed at the end of this section.
   
    \begin{lemma}\label{lemma support} For the noise given by \eqref{eq noise W},\eqref{eq psi-def} and any $N\in\N^+$, if the sequence $\{b_j\}_{j\in\N^+}$  satisfies 
    \begin{equation*}
        \sum_{j\in \N^+}b_j^2<\infty\quad \text{and}\quad b_j\neq 0\quad \text{for }1\leq j\leq N,
    \end{equation*}
       then for any $T,\delta>0$, the Wiener process $W$ satisfies that
        \begin{equation*}
            \Pb(\|W-h\|_{C([0,T];L^2(a,b))}\leq \delta)>0\quad \forall\,h\in H^1(0,T; \mathsf{P}_NL^2(a,b))\quad \text{with}\quad h(0)=0.
        \end{equation*}
    \end{lemma}

     \begin{proof}[Proof of Proposition \ref{prop irreducibility}] Without loss of generality, we consider the case  $\lambda>\nu$; the case $\lambda\leq \nu$ can be treated analogously. Let $u_*\in\{\phi_1,\phi_{-1}\}$ and $\varepsilon,R>0$ be fixed.   By Theorem \ref{thm irre-control}, there exist constants $N,m\in\N^+$, $T_*>0$ and a finite family of controls $\zeta_n\in L^2(0,T_*;\mathsf{P}_NL^2(a,b))$, $1\leq n\leq m$, such that the follow holds. There exists a constant $T=T(\varepsilon)\geq T_*$ such for each $u_0\in H$, there is $n=n(u_0)\in\{1,\cdots,m\}$ such that
        \begin{equation*}
            \|w(t,{u}_0,\zeta_n)-u_*\|\leq \varepsilon/2\quad \forall\,t\geq T.
        \end{equation*}
        Fix any $t\ge T$ and let $\delta\in(0,1)$ be determined later, and set
        \begin{equation*}
            \hat{\zeta}_n(s)=\int_{0}^{s\wedge T_*}\zeta_n(r)dr,\quad 0\leq s\leq t.
        \end{equation*}
        We claim that there exists $\delta=\delta(\varepsilon,R,t)>0$ such that
        \begin{equation}\label{eq irre claim}
            \|\hat{w}(t,u_0,h)-u_*\|\leq \varepsilon\quad\forall\,u_0\in B_H(R),\;h\in B_{C([0,t];\mathsf{P}_NL^2(a,b))}(\hat\zeta_n,\delta)\text{ with }h(0)=0,
        \end{equation}
        where $\hat{w}(t,u_0,h)$ denotes the solution of 
        \begin{equation*}
            \hat{w}(t)-u_0-\nu\int_0^t\partial_x^2\hat{w}(s)ds+\int_0^t\hat{w}(s)^3ds-\lambda\int_0^t\hat{w}(s)ds=h(t),\quad t>0.
        \end{equation*}
        Then for $\hat{v}(t):=\hat{w}(t,u_0,h)-w(t,u_0,\zeta_n)-r(t)$ with $r(t):=h(t)-\hat{\zeta}_n(t)$, it satisfies 
        \begin{equation*}
            \partial_t\hat{v}-\nu\partial_x^2\hat{v}+(\hat{v}+w+r)^3-w^3-\lambda(\hat{v}+r)=\nu\partial_x^2r,\quad \hat{v}(0)=0.
        \end{equation*}
        We compute that there exists a constant $C>0$ such that
        \begin{equation*}
            \partial_t\|\hat{v}\|^2\leq C\left(\|w\|^2_{L^\infty(0,\pi)}+1\right)\left(\|\hat{v}\|^2+\|r^2\|^2+\|r\|^2+\|\partial_xr\|^2\right).
        \end{equation*} 
        Since $r\in \mathsf{P}_NL^2(a,b)$, this implies that there exists a constant $C_1=C_1(N)>0$ such that
        \begin{equation*}
            \|r(s)^2\|^2+\|r(s)\|^2+\|\partial_xr(s)\|^2\leq C_1\delta^2\quad\forall s\in(0,t).
        \end{equation*}
        Moreover, there exists a constant $C_2=C_2(R)>0$ such that
        \begin{equation*}
            \sup_{u_0\in B(0,R)}\int_0^t\|w(s,u_0,\zeta_n)\|^2_{L^\infty(0,\pi)}ds\leq C_2.
        \end{equation*}
        Thus for $\|\hat{v}(r)\|\leq 1$ with $r\in(0,s)$, we arrive at
        \begin{equation*}
            \|\hat{v}(s)\|^2\leq Ct\exp\left(\int_0^tC\|w(r,u_0,\zeta_n)\|_{L^\infty(0,\pi)}^2dr\right)\delta^2\leq C\delta^2,  
        \end{equation*}
        which implies that there exists a constant $C=C(R,t)$, independent of $u_0$ and $\delta$ such that
        \begin{equation*}
            \|\hat{v}(s)\|=\|\hat{w}(s,u_0,h)-w(s,u_0,\zeta_n)-r(s)\|\leq C\delta.
        \end{equation*}
         In particular, taking $\delta$ sufficiently small, we can ensure that $\|\hat{v}(s)\|\leq 1$ holds for $s\in(0,t)$. This gives that
        \begin{equation*}
            \|\hat{w}(t,u_0,h)-u_*\|\leq\varepsilon/2+\delta+C\delta,
        \end{equation*}
        which implies the claim \eqref{eq irre claim}.

            \vspace{0.3em}

        Consequently, taking Lemma \ref{lemma support} into account, we obtain that for such $u_0\in B_H(0,R)$,
        \begin{equation*}
            \Pb\left(\|u(t,u_0,W)-u_*\|\leq \varepsilon\right)\geq \Pb\left(\|W-\hat\zeta_n\|_{C([0,t];L^2(a,b))}\leq \delta\right):=p_n>0.
        \end{equation*}
        Here we use $\hat\zeta_n\in H^1(0,t; \mathsf{P}_NL^2(a,b))$ and $\hat{\zeta}(0)=0$ for any $t\geq T$ and $1\leq n\leq m$. Accordingly, we conclude that
         \begin{equation*}
            \inf_{u_0\in B_H(0,R)}\Pb(\|u(t,u_0,W)-u_*\|\leq \varepsilon)\geq \min_{1\leq n\leq m}\{p_n\}>0.
        \end{equation*}
        This completes the proof of \eqref{eq irreducibility}, i.e. Proposition \ref{prop irreducibility}.
    \end{proof}

    \vspace{0.6em}

    Finally, it remains to prove Lemma \ref{lemma support}  below.
    \begin{proof}[Proof of Lemma \ref{lemma support}]
    Fix any $N\in\N^+$ and $\varepsilon>0$.  
    Let us write the noise as the orthogonal decomposition
    \begin{equation*}
        W(t)=W_N(t)+W_{>N}(t),\quad 
        W_N(t)=\sum_{1\leq j\leq N} b_j\beta_j(t)\psi_j,\quad
        W_{>N}(t)=\sum_{j>N} b_j\beta_j(t)\psi_j.
    \end{equation*}
    The two processes $W_N$ and $W_{>N}$ are independent Gaussian processes.   Since $b_j\neq 0$ for $1\le j\le N$, the covariance of $W_N$ is nondegenerate on
    $\mathsf P_NL^2(a,b)$.  
    Hence its Cameron--Martin space is
    \begin{equation*}
        H_{W_N}=\left\{h\in H^1(0,T;\mathsf P_NL^2(a,b)):h(0)=0\right\},
    \end{equation*}
    and by the classical support theorem for Gaussian measures (e.g.
    \cite[Theorem~2.6.6]{Bogachev-98}),
    \begin{equation*}
        \supp\mathcal D(W_N)
        =\overline{H_{W_N}}^{\,C([0,T];L^2(a,b))}.
    \end{equation*}
    Thus for any $h\in H^1(0,T;\mathsf P_NL^2(a,b))$ with $h(0)=0$ and $\delta>0$,
    \begin{equation*}
        \Pb\left(\|W_N-h\|_{C([0,T];L^2(a,b))}\leq\delta/2\right)>0.
    \end{equation*}
    
    Meanwhile, as $W_{>N}$ is a continuous Gaussian process, one has 
    \begin{equation*}
        \Pb\left(\|W_{>N}\|_{C([0,T];L^2(a,b))}\leq\delta/2\right)>0.
    \end{equation*}
    Consequently, the independence of $W_N$ and $W_{>N}$ implies that
    \begin{align*}
        &\Pb\left(\|W-h\|_{C([0,T];L^2(a,b))}\leq\delta\right)
        \\
        &\geq 
        \Pb\left(\|W_N-h\|_{C([0,T];L^2(a,b))}\leq\delta/2\right)
        \Pb\left(\|W_{>N}\|_{C([0,T];L^2(a,b))}\leq\delta/2\right)
        >0,
    \end{align*}
    which completes the proof.
    \end{proof}

    \section{Exponential mixing for the Allen--Cahn equation}\label{Sec 5}

    With the preparations from the previous sections, we are now in a position to prove the Main Theorem. Specifically, the abstract criterion (Theorem~\ref{lemma HM}) reduces mixing to verifying a set of conditions, including the existence of a Lyapunov function, the asymptotic strong Feller property and a weak form of irreducibility. These conditions are established through Lemma~\ref{lemma L bdd}, Proposition~\ref{prop ASF} and Proposition~\ref{prop irreducibility}.

    \begin{theorem}\label{lemma HM}{\rm(}\hspace{-0.1mm}{\rm\cite[Theorems 3.4, 4.5]{HM-08})} Let $z=z(t,x)$ be a stochastic semiflow on a Hilbert space $X$ with $C^1$-dependence on $x\in X$. Define $\{P_t\}_{t\geq 0}$, $\{P_t^*\}_{t\geq }$ as the associated Markov semigroups. Assume that
    \begin{itemize}
        \item [(1)] there exist constants $C,\gamma>0$ and a decreasing function $\xi\colon [0,1]\rightarrow[0,1]$ with $\xi(1)<1$ such that for any $u_0\in X$, $r\in[1/4,3]$ and $t\in [0,1]$,
        \begin{equation}\label{eq HM1}
            \E\exp\left(r\gamma\|z(t,x)\|^2\right)\left(1+\|\nabla_{x} z(t,x)\|)\leq C\exp(r\gamma\xi(t)\|x\|^2\right);
        \end{equation}

        \item[(2)] there exist constants $C,\alpha>0$ such that for any Fréchet differentiable function $f\colon X\rightarrow \R$, $x\in X$ and $t\geq 0$,
        \begin{equation}\label{eq HM2}
            \|\nabla P_t f(x)\|\leq C\exp\left(\tfrac{\gamma}{2}\|x\|^2\right)\left(\sqrt{P_t(|f|^2)(x)}+e^{-\alpha t}\sqrt{P_t(\|\nabla f\|^2)(x)})\right);
        \end{equation}
        \item[(3)] for any $\varepsilon,R>0$ and $r\in(0,1)$, there exists a constant $T>0$ such that for any $t\geq T$, 
        \begin{equation}\label{eq HM3}
            \inf_{x,x'\in B_X(R)}\sup_{\Gamma\in\mathscr{C}(P_t^*\delta_{x},P_t^*\delta_{x'})}\Gamma\left(y,y'\in X\times X:\rho_r(y,y')\leq \varepsilon\right)>0,
        \end{equation}
    where $\rho_r$ is metric on $X$ defined as
    \begin{equation*}
        \rho_s(x,y):=\inf_{\substack{l\in C^1([0,1];X),\\l(0)=x, \,l(1)=y}}\int_0^1\exp\left(s\gamma\|l(t)\|^2\right)\|l'(t)\|dt\quad \text{for}\quad s\in(0,1].
    \end{equation*}    
    \end{itemize}
    Then there exist constants $C,\alpha>0$  such that
    \begin{equation*}
        \rho_1\left(P_t^*\mu,P_t^*\mu'\right)\leq C e^{-\alpha t}\rho_1(\mu,\mu')\quad\forall\,\mu,\mu' \in \mathcal{P}(X),\; t\geq 0.
    \end{equation*}
    In particular, $\{P_t^*\}_{t\geq 0}$ admits a unique invariant measure  $\mu_*\in\mathcal{P}(X)$ satisfying
    \begin{equation*}
        \left|P_tf(x)-\int_X fd\mu_*\right|\leq Ce^{-\alpha t}\exp\left(\gamma\|x\|^2\right)\left\|f-\int_X fd\mu_*\right\|_{\gamma}
    \end{equation*}
    for any $f\in\mathcal{O}_{\gamma}$,  $x\in X$ and $t\geq 0$, where
    \begin{equation*}
        \mathcal{O}_{\gamma}=\left\{f\in C^1(X):\|f\|_{\gamma}<\infty\right\},\quad \|f\|_{\gamma}=\sup_{x\in X} \exp\left(-\gamma\|x\|^2\right)\left(|f(x)|+\|\nabla f(x)\|\right).
    \end{equation*}
    \end{theorem}

    Invoking this criterion, we establish the Main Theorem in the following.

    \begin{proof}[Proof of the Main Theorem]
    From Section \ref{Sec 2.1}, equation \eqref{AC equation} defines a Feller family of Markov processes in $H$. Then according to Theorem \ref{lemma HM}, to complete the proof of the Main Theorem, it suffices to verify the conditions (1)-(3). 

    \vspace{0.3em}

    Here are the detailed verifications: let $B_0>0$ be arbitrarily given.
    \begin{itemize}[leftmargin=2em]
        \item[(1)] {\it Lyapunov function.} In view of a priori bound \eqref{eq L2 bdd}, there exist positive constants $\gamma_0=\gamma_0(B_0)$, $C=C(B_0)$ such that for any $\gamma\in(0,\gamma_0]$,
        \begin{equation*}
            \E\exp\left(\gamma\|u_t\|^2\right)\leq C\exp\left(\gamma e^{-\nu t}\|u_0\|^2\right)\quad \forall\,u_0\in H,\;t\geq 0.
     \end{equation*}
        Therefore, by taking $\gamma_*=\gamma_0/3$, $\xi(t)=e^{-\nu t}$ and noticing \eqref{eq J bound}, one has for any $r\in[1/4,3]$,
        \begin{equation*}
            \E\exp\left(r\gamma\|u(t,u_0\|^2\right)\left(1+\|\mathcal{J}_{0,t}\|\right)\leq C\exp\left(r\gamma\xi(t)\|u_0\|^2\right)\quad 
            \forall\, u_0\in H,\; t\in[0,1],
        \end{equation*}
        which verifies relation \eqref{eq HM1} for any $\gamma\in(0,\gamma_*]$.

        \vspace{0.3em}

        \item[(2)] {\it Asymptotic strong Feller property.} Invoking Proposition~\ref{prop ASF} and setting $\alpha=1$, for any $\gamma\in(0,\gamma_*]$, there exists a constant $N_1=N_1(\gamma)\in\N^*$ such that if the sequence $\{b_j\}_{j\in\N^+}$ in \eqref{eq noise W} satisfies
        \begin{equation*}
       \sum_{j\in\N^+} j^2b_j^2\leq B_0\quad \text{and}\quad b_j\neq 0\quad \text{for }1\leq j\leq N_1,
    \end{equation*}
    then the asymptotic strong Feller property \eqref{eq HM2} is satisfied.

        \vspace{0.3em}

        \item[(3)] {\it Irreducibility.} To verify inequality \eqref{eq HM3}, it suffices to first construct independent couplings and then apply Proposition \ref{prop irreducibility}. The proof is analogous to that of \cite[Theorem 5.5]{HM-08}, \cite[Theorem 2.3]{FGRT-15}.   Let $N_2\in\N^+$ be given by  Proposition \ref{prop irreducibility}, and assume that 
        \begin{equation*}
            b_j\neq 0\quad \text{for }1\leq j\leq N_2.
        \end{equation*}
        Note that $\sum_{j\in \N^+}b_j^2\leq B_0<\infty$. For any $u_0,\hat{u}_0\in H$ and $t\geq 0$, define $\hat{\Gamma}_t\in\mathscr{C}(P_t^*\delta_{u_0},P_t^*\delta_{\hat{u}_0})$ as
        \begin{equation*}
            \hat{\Gamma}_t(A_1\times A_2)=P_t(u_0,A_1)P_t(\hat{u}_0,A_2),\quad A_1,A_2\in\mathcal{B}(H),
        \end{equation*}
        where the notation $\mathscr{C}(\mu,\hat{\mu})$ denotes the set of couplings between two probability measures $\mu$ and $\hat{\mu}$.   Then we compute that for any $\varepsilon,R>0$ and $r\in(0,1)$,
        \begin{align*}
            &\inf_{u_0,\hat{u}_0\in B_H(0,R)}\sup_{\Gamma\in\mathscr{C}(P_t^*\delta_{u_0},P_t^*\delta_{\hat{u}_0})}\Gamma\left(v,\hat{v}\in H\times H:\rho_r(v,\hat{v})\leq \varepsilon\right)\\
            &\quad\geq \inf_{u_0,\hat{u}_0\in B_H(0,R)}\hat{\Gamma}_t\left(v,\hat{v}\in H\times H:\rho_r(v,\hat{v})\leq \varepsilon\right)\\
            &\quad\geq \inf_{u_0,\hat{u}_0\in B_H(0,R)}\hat{\Gamma}_t\left(v,\hat{v}\in B_H(u_*,1)\times  B_H(u_*,1):\|v-\hat{v}\|e^{2r\gamma(\|v\|^2+\|\hat{v}\|^2)}\leq \varepsilon\right)\\
            &\quad\geq \inf_{u_0,\hat{u}_0\in B_H(0,R)}\hat{\Gamma}_t\left(v,\hat{v}\in B_H(u_*,1)\times  B_H(u_*,1):\|v-u_*\|\lor\|\hat{v}-u_*\|\leq \tfrac{1}{2}\varepsilon e^{-4\gamma_*(1+\|u_*\|^2)}\right)\\
            &\quad\geq \left(\inf_{u_0\in B_H(0,R)}\Pb(\|u(t,u_0)-u_*\|\leq  \tfrac{1}{2}\varepsilon e^{-4\gamma_*(1+\|u_*\|^2)}\wedge1)\right)^2\\
            &\quad>0,
        \end{align*}
    provided that $t\geq T=T(\tfrac{1}{2}\varepsilon e^{-4\gamma_*(1+\|u_*\|^2)}\wedge1)$ established as in Proposition \ref{prop irreducibility}.        
    \end{itemize}

    \vspace{0.3em}

    Collecting these results and applying Theorem \ref{lemma HM}, we conclude that for $N_*=\max\{N_1(\gamma_*),N_2\}=N_*(B_0)$, if the sequence $\{b_j\}_{j\in\N^+}$ in \eqref{eq noise W} satisfies
    \begin{equation*}
       \sum_{j\in\N^+} j^2b_j^2\leq B_0\quad \text{and}\quad b_j\neq 0\quad \text{for }1\leq j\leq N_*,
    \end{equation*}
    then the Markov process associated to \eqref{AC equation} admits a unique invariant measure $\mu_*$. Additionally, the exponential mixing \eqref{eq exponential mixing} holds. This completes the proof of the Main Theorem.

    \end{proof}

    \section*{Appendix}\label{Appendix}

    \stepcounter{section}
    
    \numberwithin{equation}{subsection}
    \numberwithin{theorem}{subsection}
    
    \setcounter{equation}{0}
    \setcounter{theorem}{0}

    In this appendix, we summarize some useful supplementary probabilistic materials, as well as some technical proofs.
    
    \renewcommand\thesubsection{\Alph{subsection}}
 
    \subsection{Fenchel duality}\label{Sec A}

    To establish the stabilization result, we utilize a classical Fenchel duality result; see e.g. \cite{Rockafellar-67,ET-99}. For clarity and self-consistency, let us summarize the required result below. Let $X$ and $Y$ be two Hilbert spaces, and $F\colon X\rightarrow\R^+$, $G\colon Y\rightarrow\R^+\cup\{+\infty\}$ be two convex and lower semicontinuous functions. In addition, let $A\in\mathcal{L}(X;Y)$.

    We shall be concerned with the minimization problem
    \begin{equation}\label{eq problem J}
        \inf_{x\in X} J(x):=\inf_{x\in X} \left(F(x)+G(Ax)\right).
    \end{equation}
    Invoking the Fenchel--Rockafellar duality, it   enables us to construct a dual problem, namely 
    \begin{equation}\label{eq problem J*}
        \sup_{y\in Y} J^*(y):=\sup_{y\in Y} \left(-G^*(-y)-F^*(A^*y)\right),
    \end{equation}
    where $F^*,G^*$ are the corresponding Fenchel conjugates, and $A^*\in\mathcal{L}(Y;X)$ is the dual of $A$.

    With these settings, the Fenchel--Rockafellar duality theorem is stated as follows; see e.g. \cite{Rockafellar-67,ET-99} for more general results.
    \begin{proposition}\label{prop F-R duality}
        Assume there exists $x\in X$ such that $G$ is finite and continuous at $Ax$, and
        \begin{equation*}
            \lim_{\|x\|_{X}\rightarrow\infty}J(x)=+\infty.
        \end{equation*}
       Then problems \eqref{eq problem J} and \eqref{eq problem J*} each have least one solution, and
        \begin{equation*}
            \inf_{x\in X}J(x)=\sup_{y\in Y}J^*(y).
        \end{equation*}
    \end{proposition}

    \subsection{A priori moment estimates}\label{Sec B}

    This appendix collects some estimates for the Allen--Cahn equation and its Jacobian flow.  

    \begin{lemma}\label{lemma L bdd}
    For the equation give by \eqref{AC equation}-\eqref{eq psi-def}, assume that the sequence $\{b_j\}_{j\in\N^+}$ satisfies
    \begin{equation*}
       \sum_{j\in\N^+} j^2b_j^2<\infty.       
    \end{equation*}    
    Then there exist constants $c,C,\gamma_0>0$ such that the solution $u_t=u(t,u_0)$ of \eqref{AC equation} satisfies the following estimates.  
    \begin{itemize}
        \item [(1)] For any $\gamma\in(0,\gamma_0]$, $u_0\in H$ and $t\geq0$, 
         \begin{align}      \E\exp\left(\gamma\|u_t\|^2\right)&\leq C\exp\left(\gamma e^{-\nu t}\|u_0\|^2\right).\label{eq L2 bdd}
    \end{align}
        \item [(2)] For any $\gamma\in(0,\gamma_0]$, $u_0\in H$ and $n\in\N^+$, 
        \begin{align}      \E\exp\left(\gamma\sum_{0\leq k\leq n}\|u_k\|^2\right)\leq \exp\left(c\gamma \|u_0\|^2\right)\exp\left(Cn\right).\label{eq L2 bdd3} 
    \end{align}  
        \item [(3)]  For any $\gamma\in(0,\gamma_0]$, $u_0\in H$ and $n\in\N^+$, 
         \begin{equation}\label{eq L-inf bdd}        \E\left(\exp\left(\gamma\|u^2\|_{L^\infty((n-1/2,n)\times(0,\pi))}\right)|\mathcal{F}_{n-1}\right)\leq C\exp\left(c\gamma\|u_{n-1}\|^2\right).
    \end{equation}
    \end{itemize}    
    \end{lemma}

    \begin{proof}[Proof of Lemma \ref{lemma L bdd}]
    For $k=0,1$, let us set
    \begin{equation*}        \mathcal{E}_k:=\sum_{j\in\N^+}j^{2k}b_j^2.
     \end{equation*}
    For ease of notation, we use $C\geq 1$  to denote generic constants that may vary from line to line below, which are independent of $u,t,n$.

    \vspace{0.6em}
    
    \noindent {\it Point (1).} Applying the Itô formula to $\varphi(u)=\|u\|^2$, one has
    \begin{equation*}
        d\|u_t\|^2=-2\nu\|u_t\|^2_{H^1}dt-2\|u_t^2\|^2dt+2\lambda\|u_t\|^2dt+\mathcal{E}_0dt+2\sum_{j\in\N^+}b_j\langle u_t,\psi_j\rangle d\beta_j(t).
    \end{equation*}
    It then follows that
        \begin{align*}
        d\|u_t\|^2\leq -2\nu\|u_t\|^2dt+Cdt+2\sum_{j\in\N^+}b_j\langle u_t,\psi_j\rangle d\beta_j(t).
    \end{align*}    
    This implies that 
    \begin{align*}    
      d(\|u_t\|^2e^{\nu t})\leq  Ce^{\nu t}dt+ 2\sum_{j\in\N^+}b_je^{\nu t}\langle u_t,\psi_j\rangle d\beta_j(t)-\nu e^{\nu t}\|u_t\|^2dt,
    \end{align*}
    and 
    \begin{align*}
        \|u_t\|^2-e^{-\nu t}\|u_0\|^2-C\nu^{-1}\leq2\sum_{j\in\N^+}b_j\int_0^te^{-\nu (t-s)}\langle u_s,\psi_j\rangle d\beta_j(s)-\nu\int_0^te^{-\nu(t-s)}\|u_s\|^2ds.
    \end{align*}

    Let 
    \begin{equation*}
        M_t=2\sum_{j\in\N^+}b_j\int_0^t\langle u_s,\psi_j\rangle d\beta_j(s),
    \end{equation*}
    then 
    \begin{equation*}
        \langle M\rangle_t=4\sum_{j\in\N^+}b_j^2\int_0^t\langle u_s,\psi_j\rangle^2ds\leq 4\mathcal{E}_0\int_0^t\|u_s\|^2ds.
    \end{equation*}
    
    Consequently, for any
    \begin{equation}\label{eq H1-8}
        0<\gamma\leq \frac{\nu}{4\mathcal{E}_0}:=\gamma_1,
    \end{equation}
     we derive that
    \begin{align*}
        \|u_t\|^2-e^{-\nu t}\|u_0\|^2-C\nu^{-1}\leq \int_0^te^{-\nu (t-s)}dM_s-\gamma\int_0^te^{-\nu (t-s)}d\langle M\rangle_s.
    \end{align*}
    Thus by \cite[Lemma A.1]{Mattingly-02b}, the desired inequality \eqref{eq L2 bdd}  follows from that for any $K>0$,
    \begin{equation*}
        \Pb\left(\|u_t\|^2-e^{-\nu t}\|u_0\|^2-C\nu^{-1}\geq K\right)\leq e^{-2\gamma K}
    \end{equation*}
    and the fact that a random variable $X$ satisfies $\E X\leq 2$ provided $\Pb(X\geq K)\leq K^{-2}$ for any $K\geq 0$.

    \vspace{0.6em}

    \noindent {\it Point (2).} Using \eqref{eq L2 bdd}, there exists $C_1>0$ such that
        \begin{align*}
            \E\exp\left(\gamma\sum_{0\leq k\leq n}\|u_k\|^2\right)&=\E\left(\exp\left(\gamma\sum_{0\leq k\leq n-1}\|u_k\|^2\right)\E( e^{\gamma\|u_n\|^2}|\mathcal{F}_{n-1})\right)\\
            &\leq C_1\E\exp\left(\gamma\sum_{0\leq k\leq n-1}\|u_k\|^2+\gamma e^{-\nu}\|u_{n-1}\|^2\right).
        \end{align*}
        Applying this procedure repeatedly, one has
         \begin{align*}
            \E\exp\left(\gamma\sum_{0\leq k\leq n}\|u_k\|^2\right)\leq C_1^n\E\exp\left(c_1\gamma\|u_0\|^2\right),
        \end{align*}
        where 
        \begin{equation*}
            c_1:=\sum_{k\in\N}e^{-\nu k}=(1-e^{-\nu})^{-1},
        \end{equation*}
        provided that
        \begin{equation*}
            0<\gamma\leq \gamma_2:=c_1^{-1}\gamma_1.
        \end{equation*}
        In particular, \eqref{eq L2 bdd3} is thus derived with
        \begin{equation*}
            C=C_2:=\ln (C_1+1)\quad \text{and }\quad c=c_1=(1-e^{-\nu})^{-1}.
        \end{equation*}

    \vspace{0.6em}
    
    \noindent {\it Point (3).}   For any $n\in\N^+$, $n-1\leq t\leq n$ and  $\varphi(t,u)=(t-n+1)\|u\|^2_{H^1}$, applying the Itô formula, we have
    \begin{equation}\label{eq H1-1}
    \begin{aligned}
        d\varphi(t,u_t)&=\|u_t\|^2_{H^1}dt+2(t-n+1)\langle -\partial_x^2u_t,du_t\rangle+(t-n+1)\sum_{j\in\N^+}b_j^2\|\psi_j\|^2_{H^1}dt\\
        &=\|u_t\|^2_{H^1}dt-2\nu (t-n+1)\|u_t\|^2_{H^2}dt-2(t-n+1)\langle \partial_xu_t,(3u_t^2-\lambda)\partial_xu_t\rangle dt\\
        &\quad+\pi^2(b-a)^{-2}\mathcal{E}_1(t-n+1)dt+ 2(t-n+1)\sum_{j\in\N^+}b_j\langle -\partial_x^2u_t,\psi_j\rangle d\beta_j(t).
    \end{aligned}
    \end{equation}  
    Set 
    \begin{equation*}
        M_{n,t}=2\sum_{j\in\N^+}b_j\int_{n-1}^t(s-n+1)\langle -\partial_x^2u_s,\psi_j\rangle d\beta_j(s),\quad  n-1\leq t\leq n.
    \end{equation*}
    
    Then for any $n-1\leq t\leq n$,
    \begin{align*}
        \langle M_n\rangle_t&=4\sum_{j\in\N^+}b_j^2\int_{n-1}^t(s-n+1)^2\langle \partial_x^2u_s,\psi_j\rangle ^2ds\leq 4\sum_{j\in\N^+}b_j^2\int_{n-1}^t(s-n+1)\|u_s\|^2_{H^2}ds\\
        &\leq 4\mathcal{E}_0\int_{n-1}^t(s-n+1)\|u_s\|^2_{H^2}ds. 
    \end{align*}

    By \eqref{eq H1-1}, it follows that there exist generic positive constants 
    \begin{equation*}
        C_*=1+2\lambda,\quad C_1=4\mathcal{E}_0,\quad C_2=\pi^2(b-a)^{-2}\mathcal{E}_1,
    \end{equation*}
    such that for any $\gamma>0$, 
    \begin{equation}\label{eq H1-2}
    \begin{aligned}
        &(t-n+1)\|u_t\|_{H^1}^2-C_*\int_{n-1}^t\|u_s\|_{H^1}^2ds+(2\nu-\gamma C_1)\int_{n-1}^t(s-n+1)\|u_s\|_{H^2}^2ds-C_2\\
        &\quad\leq M_{n,t}-\gamma\langle M_n\rangle_t.      
    \end{aligned}
    \end{equation}

    This implies that by taking
    \begin{equation}\label{eq H1-7}
        \gamma_3:=\frac{\nu}{C_1},
    \end{equation}
    it follows that for $n\in\N^+$ and $\gamma\in(0,\gamma_3]$,
    \begin{align*}        \Pb\left(\sup_{t\in[n-1,n]}\left\{(t-n+1)\|u_t\|_{H^1}^2-C_*\int_{n-1}^t\|u_s\|_{H^1}^2ds+\nu\int_{n-1}^t(s-n+1)\|u_s\|_{H^2}^2ds\right\}\geq K\right)\\
    \leq Ce^{-2\gamma K},\quad\;
    \end{align*}
    which implies    
    \begin{align} \label{eq H1-5}       \E\left(\sup_{t\in[n-1,n]}\left\{(t-n+1)\|u_t\|_{H^1}^2-C_*\int_{n-1}^t\|u_s\|_{H^1}^2ds+\nu\int_{n-1}^t(s-n+1)\|u_s\|_{H^2}^2ds\right\}\right)\leq C.
    \end{align}

    \vspace{0.6em}

    Applying Itô formula to $\varphi(u)=\|u\|^2$, one has
    \begin{equation*}
        d\|u_t\|^2=-2\nu\|u_t\|^2_{H^1}+2\lambda\|u_t\|^2dt-2\|u_t^{2}\|^2dt+\mathcal{E}_0dt+d\hat{M}_{s,t},
    \end{equation*}
         where
    \begin{align*}
        \hat{M}_{s,t}=2\sum_{j\in\N^+}b_j\int_s^t\langle u_r,\psi_j\rangle d\beta_j(r),\quad\langle \hat{M}_{s}\rangle_t=4\sum_{j\in\N^+}b_j^2\int_s^t\langle u_r,\psi_j\rangle^2 dr\leq 4\mathcal{E}_0\int_s^t\|u_r\|_{H^1}^2dr.
    \end{align*}    
    Note that there exists a constant  $C_3=\mathcal{E}_0+\pi\lambda^2$ such that
        \begin{align*}
        d\|u_t\|^2&=-2\nu\|u_t\|^2_{H^1}+2\lambda\|u_t\|^2dt-2\|u_t^{2}\|^2dt+\mathcal{E}_0dt+dM_{s,t}\\
        &\leq -2\nu\|u_t\|^2_{H^1}-\|u_t^{2}\|^2dt+C_3dt+dM_{s,t}.
    \end{align*}
    
    This implies that for any
    \begin{equation*}
        0<\gamma\leq \frac{\nu}{4\mathcal{E}_0}=\gamma_1,
    \end{equation*}
    we have
    \begin{equation*}
        \|u_t\|^2-\|u_s\|^2+\nu\int_s^t\|u_r\|^2_{H^1}dr+\int_s^t\|u_r^2\|^2dr-C_3(t-s)\leq \hat{M}_{s,t}-\gamma\langle \hat{M}_{s}\rangle_t. 
    \end{equation*}

    Therefore, for any $K>0$,
    \begin{equation*}
        \Pb\left(\sup_{t\geq s}\left\{\|u_t\|^2-\|u_s\|^2+\nu\int_s^t\|u_r\|^2_{H^1}dr+\int_s^t\|u_r^2\|^2dr-C_3(t-s)\right\}\geq K\right)\leq e^{-2\gamma K},
    \end{equation*}
    which implies that for any $n\in\N^+$ and $\gamma\in(0,\gamma_1]$,    \begin{equation}\label{eq H1-6} 
        \E\left(\exp\left(\gamma\nu\int_{n-1}^{n}\|u_t\|^2_{H^1}dt\right)|\mathcal{F}_{n-1}\right)\leq C\exp\left(\gamma\|u_{n-1}\|^2\right).
    \end{equation}

    Note that 
    \begin{align*}
        \E\exp\left(\gamma\|u\|^2_{C([n-1/2,n];H^1)}\right)&\leq \E\exp\left(2\gamma\sup_{t\in[n-1,n]}\{(t-n+1)\|u_t\|^2_{H^1}\}\right)\\
        &=\E\Biggl(\exp\left(2\gamma\left(\sup_{t\in[n-1,n]}\{(t-n+1)\|u_t\|^2_{H^1}\}-C_*\int_{n-1}^n\|u_t\|^2_{H^1}dt\right)\right)\\
        &\qquad \times\exp\left(2\gamma C_*\int_{n-1}^{n}\|u_t\|^2_{H^1}dt\right)\Biggr)\\
        &\leq \left(\E\exp\left(4\gamma\left(\sup_{t\in[n-1,n]}\{(t-n+1)\|u_t\|^2_{H^1}\}-C_*\int_{n-1}^{n}\|u_t\|^2_{H^1}dt\right)\right)\right)^{\frac{1}{2}}\\
        &\quad \times \left(\E\exp\left(4\gamma C_*\int_{n-1}^{n}\|u_t\|^2_{H^1}dt\right)\right)^{\frac{1}{2}}.
    \end{align*}

    Thus, taking \eqref{eq H1-8}, \eqref{eq H1-7}, \eqref{eq H1-5} and \eqref{eq H1-6} into account, for any 
    \begin{equation*}
        0<\gamma\leq \frac{\gamma_3}{4}\wedge \frac{\gamma_1}{4C_*}, 
    \end{equation*}
    we conclude that
    \begin{align*}        \E\left(\exp\left(\gamma\|u\|^2_{C([n-1/2,n];H^1)}\right)|\mathcal{F}_{n-1}\right)&\leq C\exp\left(2C_*\nu^{-1}\gamma\|u_n\|^2\right).
    \end{align*}

    Note that for any $w\in H_0^1(0,\pi)$,
    \begin{align*}
        \|w^2\|_{L^\infty(0,\pi)}\leq \|w\|_{L^\infty(0,\pi)}^2 \leq C\|w\|_{H^1}^2.
    \end{align*}
    
    Thus, one has 
    \begin{equation*}
        \|u^2\|_{L^\infty((n-1/2,n)\times(0,\pi))}\leq C\|u\|_{C([n-1/2,n];H^1)}^2,
    \end{equation*}
    which implies the desired inequality \eqref{eq L-inf bdd}.

    \end{proof}

    Next we also derive some bounds for the linearized equations, summarized in the following.

     \begin{lemma}\label{lemma J bounds}
      For any $u_0\in H$ and $0\leq s<t$, the solutions of \eqref{J equation} and \eqref{J* equation} satisfy that
        \begin{equation}\label{eq J bound}
            \|\mathcal{J}_{s,t}\xi\|\leq e^{\lambda(t-s)}\|\xi\|,\quad \|\mathcal{J}^*_{s,t}\xi\|\leq e^{\lambda(t-s)}\|\xi\| \quad\forall\,\xi\in H.
        \end{equation}
        Moreover, let $\mathcal{J}^{(2)}_{s,t}(\xi,\xi')$ be the solution of
        \begin{equation}\label{J2 equation}
        \begin{cases}
         \partial_t\mathcal{J}^{(2)}_{s,t}(\xi,\xi')-\nu \partial_x^2\mathcal{J}^{(2)}_{s,t}(\xi,\xi')+(3u_t^2-\lambda)\mathcal{J}^{(2)}_{s,t}(\xi,\xi')+6u_t\mathcal{J}_{s,t}\xi\mathcal{J}_{s,t}\xi'=0,\\         
        \mathcal{J}^{(2)}_{s,s}(\xi,\xi')=0,
        \end{cases}
        \end{equation}
        then there exists a constant $C=C(\nu)>0$ such that        \begin{equation}\label{eq J_2 bound}
            \|\mathcal{J}^{(2)}_{s,t}(\xi,\xi')\|\leq Ce^{3\lambda(t-s)}\|\xi\|\|\xi'\|\quad\forall\,\xi,\xi'\in H.
        \end{equation}
    \end{lemma}

        \begin{proof}[Proof of Lemma \ref{lemma J bounds}]
        To derive the first estimate, we have
         \begin{align*}
        \partial_t\|\mathcal{J}_{s,t}\xi\|^2&=-2\nu\|\mathcal{J}_{s,t}\xi\|_{H^1}^2+2\lambda\|\mathcal{J}_{s,t}\xi\|^2-6\|u_t\mathcal{J}_{s,t}\xi\|^2\\
        &\leq -2\nu\|\mathcal{J}_{s,t}\xi\|_{H^1}^2+2\lambda\|\mathcal{J}_{s,t}\xi\|^2,
    \end{align*}
        which implies that the first estimate in \eqref{eq J bound}. The estimate for $\mathcal{J}^*_{s,t}\xi$ can be calculated similarly. Meanwhile, note that        
        \begin{equation*}
            \int_s^t\|\mathcal{J}_{s,r}\xi\|^2_{H^1}dr\leq \frac{\lambda}{\nu}\int_s^t\|\mathcal{J}_{s,r}\xi\|^2dr\leq\nu^{-1}e^{2\lambda(t-s)}\|\xi\|^2.
        \end{equation*}

        For the second,  using the Young’s inequality and Sobolev embedding theorem,
    \begin{align*}
        \partial_t\|\mathcal{J}^{(2)}_{s,t}(\xi,\xi')\|^2&=-2\nu\|\mathcal{J}^{(2)}_{s,t}(\xi,\xi')\|^2_{H^1}+2\lambda\|\mathcal{J}^{(2)}_{s,t}(\xi,\xi')\|^2-6\|u_t\mathcal{J}^{(2)}_{s,t}(\xi,\xi')\|^2\\
        &\quad+6\langle u_t\mathcal{J}^{(2)}_{s,t}(\xi,\xi'),\mathcal{J}_{s,t}\xi\mathcal{J}_{s,t}\xi'\rangle\\
        &\leq 2\lambda\|\mathcal{J}^{(2)}_{s,t}(\xi,\xi')\|^2+3\|\mathcal{J}_{s,t}\xi\mathcal{J}_{s,t}\xi'\|^2\\
        &\leq 2\lambda\|\mathcal{J}^{(2)}_{s,t}(\xi,\xi')\|^2+3\|\mathcal{J}_{s,t}\xi\|_{L^\infty(0,\pi)}^2\|\mathcal{J}_{s,t}\xi'\|^2\\
        &\leq 2\lambda\|\mathcal{J}^{(2)}_{s,t}(\xi,\xi')\|^2+C\|\mathcal{J}_{s,t}\xi\|_{H^1}^2\|\mathcal{J}_{s,t}\xi'\|^2.
    \end{align*}
    Thus applying the Gronwall's ineqaulity, we obtain
    \begin{align*}
        \|\mathcal{J}^{(2)}_{s,t}(\xi,\xi')\|^2\leq Ce^{2\lambda(t-s)}\int_s^t\|\mathcal{J}_{s,r}\xi\|_{H^1}^2\|\mathcal{J}_{s,r}\xi'\|^2dr\leq C\nu^{-1}e^{6\lambda(t-s)}\|\xi\|^2\|\xi'\|^2.
    \end{align*}

    \end{proof}

    \subsection{Verification of the asymptotic strong Feller}\label{Sec C} Below we present a detailed proof of Lemma \ref{prop rho_n 2}. The proof is based on the construction of control $v$ by \eqref{eq vn def}, \eqref{eq v def}, an application of the Malliavin calculus and some energy estimates, which consists of two parts.  The proof combines the strategies of \cite[Proposition 4.15]{HM-06} and \cite[Proposition 2.6]{FGRT-15}. 
    
    In this appendix, let $B_0>0$ be arbitrarily given, and we use $C=C(B_0)>0$ to denote constants which are independent of $u,n,s,t,N,\delta,\beta$ and may vary form line to line.  

    \subsubsection{ Proof of inequality \eqref{eq rho_n 2} } 
     We first show that, under the assumptions of Lemma \ref{prop rho_n 2},  for any $\gamma,\varepsilon>0$, there exist constants $N\in \N^+$, $\delta\in(0,1/2)$ and $\beta>0$ such that,
        \begin{equation}\label{eq rho_n 3}
            \E\|\rho_n\|^{2}\leq \varepsilon^n\exp\left(\gamma\|u_0\|^2\right)\quad \forall\,u_0\in H,\;n\in\N^+.
        \end{equation}

    \vspace{0.3em}

       Recall that $\rho_0=\xi\in H$, $\|\xi\|=1$. Let $\gamma,\varepsilon\in(0,1)$ be given. We further introduce two parameters $\bar{\gamma}\in(0,\gamma)$ and $\bar{\varepsilon}\in(0,\varepsilon)$ to be fixed.  Define a function $\Phi$ as
       \begin{equation*}
           \Phi(x,y):=\begin{cases}
               x/y\quad&\text{for }y\neq 0,\\
               0\quad&\text{for }y= 0.
           \end{cases}
       \end{equation*}

        Invoking \eqref{eq rho_n 1}, there exists $N=N(\bar{\gamma},\bar{\varepsilon})$ $\delta=\delta(\bar{\gamma},\bar{\varepsilon})$ and $\beta=\beta(\bar{\gamma},\bar{\varepsilon})$ such that 
        \begin{equation}\label{eq rho_n 2-1}
            \E\left(\|\rho_n\|^{4}|\mathcal{F}_{n-1}\right)\leq \bar{\varepsilon}e^{\bar{\gamma}\|u_{n-1}\|^2}\|\rho_{n-1}\|^{4}\quad \forall\,u_0\in H,\;n\in\N^+.
        \end{equation}       
        For each $k\in\N^+$, set
        \begin{equation*}
            X_k:=\Phi(\|\rho_k\|,\|\rho_{k-1}\|)^2\exp\left(-\tfrac{\bar{\gamma}}{2}\|u_{k-1}\|^2\right),\quad Y_k:=\exp\left(\tfrac{\bar{\gamma}}{2}\|u_{k-1}\|^2\right).
        \end{equation*}

       Repeated use of conditional expectation with \eqref{eq rho_n 2-1} implies
    \begin{align*}
        \E\left(\prod_{1\leq k\leq n}X_k^2\right)&=\E\left(\E\left(\prod_{1\leq k\leq n}X_k^2|\mathcal{F}_{n-1}\right)\right)=\E\left(\prod_{1\leq k\leq n-1}X_k^2\E\left(X_n^2|\mathcal{F}_{n-1}\right)\right)\\
       &\leq \bar{\varepsilon}\E\left(\prod_{1\leq k\leq n-1}X_k^2\right)\leq\cdots\leq \bar{\varepsilon}^n.
    \end{align*}
    On the other hand, noting that for any $n\in\N$ and $t\geq n$,  $\mathbf{1}_{\{\|\rho_n\|=0\}}\rho_t=0$, we have
    \begin{equation*}
        \prod_{1\leq k\leq n}X_kY_k=\|\rho_n\|^2.
    \end{equation*}
    
    Consequently, using a priori bound \eqref{eq L2 bdd3}, there exist constants $C_1,c_1,\gamma_0>0$ such that for any $\bar{\gamma}\in(0,\gamma_0]$,
    \begin{align*}
        \E\|\rho_n\|^2\leq \left(\E\prod_{1\leq k\leq n}X_k^2\right)^{\frac{1}{2}}\left(\E\prod_{1\leq k\leq n}Y_k^2\right)^{\frac{1}{2}}\leq \bar{\varepsilon}^n\exp\left(c_1\bar{\gamma}\|u_0\|^2\right)\exp\left(C_1n\right).
    \end{align*}
    Therefore, the desired inequality \eqref{eq rho_n 2} at integer times follows by taking
    \begin{equation*}
        \bar{\varepsilon}=\exp(-C_1)\varepsilon,\quad \bar{\gamma}=(c_1+1)^{-1}(\gamma\wedge\gamma_0).
    \end{equation*}

    \vspace{0.6em}

    Making use of relation \eqref{eq rho_n 3}, we are now able to prove \eqref{eq rho_n 2} for any $t>0$. Let $\gamma,\alpha>0$ be arbitrarily given, and set $\varepsilon=e^{-2\alpha}$. In view of the definition of $v$ given by \eqref{eq vn def},\eqref{eq v def}, it follows that for each $n\in\N^+$,
    \begin{equation*}
        \rho_t=\begin{cases}
            \mathcal{J}_{n-1,t}\rho_{n-1},\quad &\text{for }n-1\leq t\leq n-\delta,\\
            \mathcal{J}_{n-\delta,t}\rho_{n-\delta}-\mathcal{A}_{n-\delta,t}v_n,\quad &\text{for }n-\delta\leq t\leq n.
        \end{cases}
    \end{equation*}
    This implies that for any $t\in[n-1,n-\delta]$, by \eqref{eq J bound},
    \begin{equation*}
        \|\rho_t\|\leq  e^{\lambda}\|\rho_{n-1}\|,
    \end{equation*}
    and for any $t\in[n-\delta,n]$,
    \begin{align*}
        \|\rho_t\|&\leq  e^{\lambda}\|\rho_{n-1}\|+\|\mathcal{A}_{n-\delta,t}\|_{L^2(\mathcal{H};H)}\|v_n\|_{\mathcal{H}}\\
        &\leq e^{\lambda}\|\rho_{n-1}\|+\max_{j\in\N^+}\{|b_j|\}\left(\int_{n-\delta}^n\|\mathcal{J}_{n-\delta,t}\|^2dt\right)^{1/2}\|v_n\|_{\mathcal{H}}\\
        &\leq C\left(\beta^{-1/2}+1\right)\|\rho_{n-1}\|,
    \end{align*}
    which completes the proof of \eqref{eq rho_n 2} by using \eqref{eq rho_n 3}. Here recall that the random operator $\mathcal{A}_{s,t}$ is given by \eqref{eq A def}, and we use the estimate on $v_n$ established in \eqref{eq vn estimate}.

    \subsubsection{ Proof of inequality \eqref{eq v} }
    To begin with, we first have the following estimates, which can be proved using arguments similar to those in \cite{HM-06}.

    \begin{lemma}\label{lemma AM}
        For any $u_0\in H$, $0\leq s<t$,  $N\in\N^+$ and $\beta>0$, it follows that
        \begin{align*}
            \|\mathcal{A}^*_{s,t,N}(\mathcal{M}_{s,t,N}+\beta)^{-1/2}\|_{\mathcal{L}(H;\mathcal{H}_{s,t})}&\leq 1,\\
            \|(\mathcal{M}_{s,t,N}+\beta)^{-1/2}\mathcal{A}_{s,t,N}\|_{\mathcal{L}(\mathcal{H}_{s,t};H)}&\leq 1,\\
            \|(\mathcal{M}_{s,t,N}+\beta)^{-1/2}\|_{\mathcal{L}(H;H)}&\leq \beta^{-1/2}.
        \end{align*}
    \end{lemma}

    \begin{proof}[Proof of \eqref{eq v}]
      Recall that $\mathcal{D}$ denotes the Malliavin derivative operator. By the Riesz representation theorem, for any $j\in\N^+$ and $t>0$, $\mathcal{D}$ can be formulated as
    \begin{equation*}
        \langle \mathcal{D}F,v\rangle_{\mathcal{H}}=\int_{\R^+}\langle \mathcal{D}_tF,v(t)\rangle_{\mathcal{H}}dt=\sum_{j\in\N^+}\int_{\R^+}\langle \mathcal{D}_t^jF,\langle v(t),\psi_j\rangle\psi_j\rangle_{\mathcal{H}}dt\quad \text{for }F\in L^2(\Omega;H),\;v\in\mathcal{H}.
    \end{equation*}
    In particular, we have 
    \begin{equation*}
        \mathcal{D}_s^ju_t=b_j\mathcal{J}_{s,t}\psi_j,\quad \forall\,0<s<t,\;j\in\N^+. 
    \end{equation*}
    Additionally, for any $F\in\mathcal{F}_{s}$,
    \begin{equation*}
        \mathcal{D}_tF=0\quad a.s.\quad \forall\,t>s.
    \end{equation*}

    With these preparations, we are now ready to show \eqref{eq v}. In view of the construction of the control by \eqref{eq vn def},\eqref{eq v def}, by the generalized Itô isometry, see e.g. \cite{Nualart-06}, it follows that
        \begin{align*}
            \E\left(\int_{\R^+}v(t)dW(t)\right)^2&=\E\|v\|^2_{\mathcal{H}}+\E\int_{\R^+}\int_{\R^+}\text{Tr}\left(\mathcal{D}_sv(t)\mathcal{D}_tv(s)\right)dsdt\\
            &=\sum_{n\in\N^+}\E\|v_n\|^2_{\mathcal{H}}+\sum_{n\in\N^+}\E\int_{n-\delta}^n\int_{n-\delta}^n\text{Tr}\left(\mathcal{D}_sv_n(t)\mathcal{D}_tv_n(s)\right)dsdt\\
            &\leq \sum_{n\in\N^+}\E\|v_n\|^2_{\mathcal{H}}+\sum_{n\in\N^+}\E\int_{n-\delta}^n\int_{n-\delta}^n\|\mathcal{D}_sv_n(t)\|^2_{\R^{N\times N}}dsdt.
        \end{align*}

        To bound the first term, invoking Lemma \ref{lemma AM} and \eqref{eq J bound}, we compute that
        \begin{equation}\label{eq vn estimate}
        \begin{aligned}
            \|v_n\|_{\mathcal{H}}&=\|\mathcal{A}_{n-\delta,n,N}^*(\mathcal{M}_{n-\delta,n,N}+\beta)^{-1}\mathcal{J}_{n-\delta,n}\rho_{n-\delta})\|_{\mathcal{H}}\\
            &\leq \|\mathcal{A}^*_{n-\delta,n,N}(\mathcal{M}_{n-\delta,n,N}+\beta)^{-1/2}\|_{\mathcal{L}(H;\mathcal{H}_{n-\delta,n})}\|(\mathcal{M}_{n-\delta,n,N}+\beta)^{-1/2}\|_{\mathcal{L}(H;H)}\\
            &\quad\times\|\mathcal{J}_{n-\delta,n}\|_{\mathcal{L}(H;H)}\|\rho_{n-\delta}\|\\
            &\leq \beta^{-1/2}\|\mathcal{J}_{n-\delta,n}\|_{\mathcal{L}(H;H)}\|\mathcal{J}_{n-1,n-\delta}\|_{\mathcal{L}(H;H)}\|\rho_{n-1}\|\\
            &\leq e^\lambda\beta^{-1/2}\|\rho_{n-1}\|.
        \end{aligned}
        \end{equation}    
    As a result, taking \eqref{eq rho_n 2} into account, we obtain
    \begin{align*}
        \sum_{n\in\N^+}\E\|v_n\|^2_{\mathcal{H}}\leq e^{2\lambda}\beta^{-1}\sum_{n\in\N}\E\|\rho_n\|^2\leq C\beta^{-1}\exp\left(\gamma\|u_0\|^2\right).
    \end{align*}

    To bound the second term, we compute an explicit expression for $\mathcal{D}_sv(t)$. By the Malliavin product rule, for any $1\leq j\leq N$ and $s\in[n-\delta,n]$, there holds
    \begin{align*}
        \mathcal{D}_s^jv_n&=\mathcal{A}_{n-\delta,n,N}^*(\mathcal{M}_{n-\delta,n,N}+\beta)^{-1}(\mathcal{D}_s^j\mathcal{J}_{n-\delta,n})\rho_{n-\delta}\\
        &\quad +\mathcal{A}_{n-\delta,n,N}^*(\mathcal{D}_s^j(\mathcal{M}_{n-\delta,n,N}+\beta)^{-1})\mathcal{J}_{n-\delta,n}\rho_{n-\delta}\\
        &\quad +(\mathcal{D}_s^j\mathcal{A}_{n-\delta,n,N}^*)(\mathcal{M}_{n-\delta,n,N}+\beta)^{-1}\mathcal{J}_{n-\delta,n}\rho_{n-\delta}.  \end{align*}
    In addition, one has
    \begin{align*}
        \mathcal{D}_s^j(\mathcal{M}_{n-\delta,n,N}+\beta)^{-1}&=-(\mathcal{M}_{n-\delta,n,N}+\beta)^{-1}((\mathcal{D}_s^j\mathcal{A}_{n-\delta,n,N})\mathcal{A}_{n-\delta,n,N}^*\\
        &\quad +\mathcal{A}_{n-\delta,n,N}(\mathcal{D}_s^j\mathcal{A}_{n-\delta,n,N}^*))(\mathcal{M}_{n-\delta,n,N}+\beta)^{-1}.
    \end{align*}

    To calculate the expressions of $\mathcal{D}_s^j\mathcal{A}_{n-\delta,n,N}$ and $\mathcal{D}_s^j\mathcal{J}_{n-\delta,n}$, let us note that $\mathcal{D}_r^j\mathcal{J}_{s,t}\xi$ with $r\in[0,t]$ satisfies:
     \begin{equation*}
    \begin{cases}
         \partial_t\mathcal{D}_r^j\mathcal{J}_{s,t}\xi-\nu \partial_x^2\mathcal{D}_r^j\mathcal{J}_{s,t}\xi+(3u_t^2-\lambda)\mathcal{D}_r^j\mathcal{J}_{s,t}\xi+6u_tb_j\mathcal{J}_{r,t}\psi_j\mathcal{J}_{s,t}\xi=0,\quad t>s, \\         
        \mathcal{D}_r^j\mathcal{J}_{s,r\wedge s}\xi=0.
    \end{cases}
    \end{equation*}
    From the variation of constants formula and the expression \eqref{J2 equation} for the process $\mathcal{J}^{(2)}_{s,t}$, we get
    \begin{equation*}
        \mathcal{D}_r^j\mathcal{J}_{s,t}\xi=\begin{cases}
            \mathcal{J}^{(2)}_{r,t}(b_j\psi_j,\mathcal{J}_{s,r}\xi)\quad &\text{for }s\leq r\leq t,\\
            \mathcal{J}^{(2)}_{s,t}(b_j\mathcal{J}_{r,s}\psi_j,\xi)\quad &\text{for }0<r\leq s.
        \end{cases}
    \end{equation*}
    The term  $\mathcal{D}_s^j\mathcal{A}_{n-\delta,n,N}$ thus has the form of, for $h\in\mathcal{H}_{s,t}$, 
    \begin{align*}
        \mathcal{D}_r^j\mathcal{A}_{s,t,N}h&=\int_{s}^t\mathcal{D}_r^j\mathcal{J}_{\tau,t}B\mathsf{P}_Nh(\tau)d\tau\\
        &=\int_{s}^r\mathcal{J}^{(2)}_{r,t}(b_j\psi_j,\mathcal{J}_{\tau,r}B\mathsf{P}_Nh(\tau))d\tau+\int_{r}^t\mathcal{J}^{(2)}_{\tau,t}(b_j\mathcal{J}_{r,\tau}\psi_j,B\mathsf{P}_Nh(\tau))d\tau.
    \end{align*}

   Meanwhile, for the term $\mathcal{D}_r^j\mathcal{A}^*_{s,t,N}$, one has note that
    \begin{equation*}
        \mathcal{D}_r^j\mathcal{A}^*_{s,t,N}=(\mathcal{D}_r^j\mathcal{A}_{s,t,N})^*,\quad\|\mathcal{D}_r^j\mathcal{A}^*_{s,t,N}\|_{\mathcal{L}(H;\mathcal{H}_{s,t})}=\|\mathcal{D}_r^j\mathcal{A}_{s,t,N}\|_{\mathcal{L}(\mathcal{H}_{s,t};H)}.
    \end{equation*}
    
    \vspace{0.6em}
   Summing up these relations and using \eqref{eq J_2 bound}, we derive that, for any $s\in[n-\delta,n]$, 
    \begin{align*}
      \|\mathcal{D}_s^jv_n\|_{L^2(n-\delta,n;\R^{N})}&\leq \beta^{-1/2}\|\mathcal{D}_s^j\mathcal{J}_{n-\delta,n}\|\|\rho_{n-\delta}\|+\beta^{-1}\|\mathcal{D}_s^j\mathcal{A}_{n-\delta,n,N}\|_{\mathcal{L}(\mathcal{H}_{s,t};H)}\|\mathcal{J}_{n-\delta,n}\|\|\rho_{n-\delta}\|\\
      &\quad +2\beta^{-1}\|\mathcal{D}_s^j\mathcal{A}^*_{n-\delta,n,N}\|_{\mathcal{L}(H;\mathcal{H}_{s,t})}\|\mathcal{J}_{n-\delta,n}\|\|\rho_{n-\delta}\|\\
      &\leq \beta^{-1/2}\|\mathcal{D}_s^j\mathcal{J}_{n-\delta,n}\|\|\rho_{n-\delta}\|+3\beta^{-1}\|\mathcal{D}_s^j\mathcal{A}_{n-\delta,n,N}\|_{\mathcal{L}(\mathcal{H}_{s,t};H)}\|\mathcal{J}_{n-\delta,n}\|\|\rho_{n-\delta}\|\\
      &\leq C|b_j|\beta^{-1/2}e^{3\lambda\delta}\|\rho_{n-\delta}\|+C|b_j|\beta^{-1}e^{5\lambda\delta}\|\rho_{n-\delta}\|.
    \end{align*}

    Finally, we conclude that
    \begin{align*}
        \sum_{n\in\N^+}\E\int_{n-\delta}^n\int_{n-\delta}^n\|\mathcal{D}_sv_n(t)\|^2_{\R^{N\times N}}dsdt&=\sum_{n\in\N^+}\sum_{1\leq j\leq N}\int_{n-\delta}^n\E\|\mathcal{D}^j_sv_n(t)\|^2_{L^2(n-\delta,n;\R^N)}ds\\
        &\leq C\delta\beta^{-2}e^{10\lambda\delta}\sum_{j\in\N^+}b_j^2\sum_{n\in\N^+}\E\|\rho_{n-\delta}\|^2\\
        &\leq C\delta\beta^{-2}e^{10\lambda\delta}\sum_{j\in\N^+}b_j^2\sum_{n\in\N}\E\|\rho_{n}\|^2,
    \end{align*}
    which implies the desired inequality \eqref{eq v} by \eqref{eq rho_n 2}.

    \end{proof}

    \subsection{Support lemma}\label{Sec D}

    To establish the irreducibility, let us recall some useful results regarding general Gaussian measures. Let $X$ be a separable Banach space and $X^*$ be its dual space (the space of all continuous linear functionals on $X$). 
    \begin{definition} {\rm(\hspace{-0.01mm}\cite[Definition 2.2.1]{Bogachev-98})} 
        A probability measure $\mu\in \mathcal{P}(X)$ is called a symmetric Gaussian measure if for any $F\in X^*$, the measure $\mu\circ F^{-1}\in\mathcal{P}(\R)$ is a Gaussian measure with zero mean.  The Cameron--Martin space (also called reproducing kernel Hilbert space) of $\mu$ is defined by
        \begin{equation*}
            H_{\mu}:=\{h\in X:\|h\|_{H_{\mu}}<\infty\},\quad 
            \|h\|_{H_{\mu}}:=\sup\{F(h):F\in X^*,\quad\int_XF(x)^2\mu(dx)\leq 1\}.
        \end{equation*}
        Additionally, $H_{\mu}$ is a separable Hilbert space endowed with inner product $L^2(X,\mu)$.
    \end{definition}
    \begin{example}\label{ex BM}
        Let $T>0$, $X=C([0,T])$ and $\beta(t)$ be an one-dimensional stand Brownian. Then for any $c\neq 0$, $\mu=\mathcal{D}(c\beta)\in\mathcal{P}(X)$ is a symmetric Gaussian measure, and the corresponding Cameron--Martin space is given by
        \begin{equation*}
            H_{\mu}=\{f\in C([0,T]): f\text{ is absolutely
 continuous},\quad f(0)=0,\quad f'\in L^2(0,T)\},
        \end{equation*}
        and 
        \begin{equation*}
            \|f\|_{ H_{\mu}}=|c|^{-1}\|f'\|_{L^2(0,T)}.
        \end{equation*}
    \end{example}

    \begin{lemma}\label{lemma support2}{\rm(\hspace{-0.01mm}\cite[Theorem 3.6.1]{Bogachev-98})}        
    Let $\mu\in\mathcal{P}(X)$ be a symmetric Gaussian measure. Then the topological support of $\mu$ coincides with $\overline{H}_{\mu}^X$, that is, for any $h\in \overline{H}_{\mu}^X$ and $\varepsilon>0$,
        \begin{equation*}
            \mu(B_X(h,\varepsilon))>0.
        \end{equation*}
    \end{lemma}

     \subsection{Auxiliary proofs for control problems}\label{Sec E} In this appendix, we collect some technical proofs for the steady-state controllability, including the proofs of Lemma \ref{lemma Kalman} and Lemma \ref{lemma L-functional}.

    \begin{proof}[Proof of Lemma \ref{lemma Kalman}]
        We first verify \eqref{eq Kalman}. Noting that
        \begin{equation*}
            (E,D(\tau)E)=\left(\begin{matrix}
            \mathbf{0}_m &  B(\tau)\\
            I_m & \mathbf{0}_m
        \end{matrix}\right),
        \end{equation*}
        thus it suffices to show
        \begin{equation}\label{eq Kalman2}
            {\rm rank }\left (B(\tau) \right)=m.
        \end{equation}

        To this end, let us recall a standard unique continuation result. Let $\theta\in \R$ and $\varphi\in H_0^1(0,\pi)\cap C^2([0,\pi])$ be the solution of 
        \begin{equation*}
            A(\tau)\varphi=\theta\varphi,\quad \varphi(0)=\varphi(\pi)=0,
        \end{equation*}
        such that
        \begin{equation*}
            \varphi(x)=0\quad\text{for }x\in(a,b).
        \end{equation*}
        Then 
        \begin{equation*}
            \varphi\equiv0\quad\text{on }[0,\pi].
        \end{equation*}

        Invoking this observation, we now apply the arguments as in \cite{BT-04} to deduce \eqref{eq Kalman2}. We claim that the eigenfunctions $\{e_j(\tau,\cdot)\}_{1\leq j\leq m}$ are linearly independent on $(a,b)$. Otherwise, there exist constants $\beta_1(\tau),\cdots,\beta_m(\tau)$ which are not all zero such that
        \begin{equation*}
            \varphi(\tau,\cdot):=\sum_{1\leq j\leq m}\beta_j(\tau) e_j(\tau,\cdot)=0\quad \text{on }(a,b).
        \end{equation*}
        Suppose that $\beta_k(\tau)\neq 0$ for some $1\leq k\leq m$. It follows that
        \begin{equation*}
           A(\tau)\left(\varphi(\tau,\cdot)-\sum_{1\leq j\neq k\leq m}\beta_j(\tau)e_j(\tau,\cdot)\right)=\lambda_k(\tau)\left(\varphi(\tau,\cdot)-\sum_{1\leq j\neq k\leq m}\beta_j(\tau)e_j(\tau,\cdot)\right).
        \end{equation*}
        Thus  by the unique  continuation property, we derive that
        \begin{equation*}
            \varphi(\tau,\cdot)-\sum_{1\leq j\neq k\leq m}\beta_j(\tau)e_j(\tau,\cdot)=\beta_k(\tau)e_k(\tau,\cdot)\equiv0,
        \end{equation*}
        which contradicts the definition of $e_k(\tau,\cdot)$. This implies our claim.

        To complete the proof of \eqref{eq Kalman2}, it suffices to show that the vectors $(b_{1j}(\tau),\cdots,b_{mj} (\tau))^{\rm T}$, $1\leq j\leq m$, are linearly independent. Let $\beta_1(\tau),\cdots,\beta_m(\tau)\in\R$ be such that
        \begin{equation*}
            \sum_{1\leq j\leq m}\beta_j(\tau) \left(b_{1j}(\tau),\cdots,b_{mj} (\tau)\right)^{\rm T}=0,
        \end{equation*} 
        which gives 
        \begin{equation*}
            \sum_{1\leq k\leq m}\beta_k(\tau)\langle\mathbf1_{[a,b]}(\cdot) e_k(\tau,\cdot),e_j(\tau,\cdot)\rangle=0\quad\text{for }1\leq j\leq m.
        \end{equation*}
        By multiplying each of these identities by $\beta_j(\tau)$ and summing up over $j=1,\cdots,m$, we obtain 
        \begin{equation*}
             \left\langle\mathbf1_{[a,b]}(\cdot)\sum_{1\leq j\leq m}\beta_j(\tau)e_j(\tau,\cdot),\sum_{1\leq j\leq m}\beta_j(\tau)e_j(\tau,\cdot)\right\rangle=0.
        \end{equation*}
        Therefore, this implies that
        \begin{equation*}
            \sum_{1\leq j\leq m}\beta_j(\tau)e_j(\tau,\cdot)=0\quad \text{on }(a,b),
        \end{equation*}
        and hence
        \begin{equation*}
            \beta_j(\tau)=0\quad\forall\, 1\leq j\leq m,\;\tau\in[0,1].
        \end{equation*}
        This completes the proof of \eqref{eq Kalman2}.

        \vspace{0.3em}

        Consequently, the Kalman condition \eqref{eq Kalman} implies a pole shifting result; see e.g. \cite[Theorem 10.1]{Coron-07}. Specifically, there exists $K(\tau)\in \R^{m\times 2m}$ such that the matrix $D(\tau)+EK(\tau)$ admits $-1$ as an eigenvalue with order $2m$. Moreover, by \cite[Theorem 4.6]{Khalil-92}, for each $\tau\in[0,1]$, there exists a symmetric and positive definite matrix $Q(\tau)\in \R^{2m\times 2m}$ such that \eqref{eq Q def} holds. In addition, given that the mapping $\tau\mapsto D(\tau)$ is $C^1$-smooth, we can choose suitable controllers $K(\tau)$ such that the mapping $\tau\mapsto Q(\tau)$ is $C^1$-smooth.

    \end{proof}

    \vspace{0.6em}

    \begin{proof}[Proof of Lemma \ref{lemma L-functional}]
        To establish \eqref{eq V estimate1}, for any $v(\cdot)=\sum_{j\in\N^+}v_j(t)e_j(\varepsilon t,\cdot)\in H_0^1(0,\pi)\cap H^2(0,\pi)$, one has
        \begin{equation*}
            \|v\|_{H^1}^2=\sum_{j,k\in\N^+}v_j(t)v_k(t)\langle \partial_xe_j(\varepsilon t,\cdot),\partial_x e_k(\varepsilon t,\cdot)  \rangle.
        \end{equation*}
        Integrating by parts and using the definition of $e_j$, we compute
        \begin{align*}
            \langle \partial_xe_j(\varepsilon t,\cdot),\partial_x e_k(\varepsilon t,\cdot)  \rangle=-\langle \partial^2_xe_j(\varepsilon t,\cdot),e_k(\varepsilon t,\cdot)  \rangle=\nu^{-1}\langle (\lambda-3\bar{y}(\varepsilon t)^2-\lambda_j(\varepsilon t))e_j(\varepsilon t,\cdot),e_k(\varepsilon t,\cdot)\rangle,
        \end{align*}
        and thus
        \begin{equation*}
            \|v\|_{H^1}^2=\frac{3}{\nu}\|\bar{y}(\varepsilon t)v\|^2+\frac{\lambda}{\nu}\|v\|^2-\frac{1}{\nu}\sum_{j\in\N^+}\lambda_j(\varepsilon t)v_j(t)^2.
        \end{equation*}
    Since $\bar{y}$ is bounded on $[0,\varepsilon^{-1}]\times [0,\pi]$ uniformly in $\varepsilon\in(0,1]$, there exist generic constants $C_1,C_2>0$ such that 
    \begin{equation*}
        \|v\|_{H^1}^2\leq C_1\left(\|v\|^2-\sum_{j>m}\lambda_j(\varepsilon t)v_j(t)^2\right)\leq C_2 V(t,\xi,v).
    \end{equation*}
    Conversely, there exist $C_3,C_4>0$ such that 
    \begin{align*}
      -\sum_{j>m}\lambda_j(\varepsilon t)v_j(t)^2&=\nu \|v\|_{H^1}^2+\sum_{1\leq j\leq m}\lambda_j(\varepsilon t)v_j(t)^2-3\|\bar{y}(\varepsilon t)v\|^2-\lambda\|v\|^2\\
      &\leq \nu\|v\|_{H^1}^2+C_3\|v\|^2\leq C_4\|v\|_{H^1}^2.
    \end{align*}
    Applying \eqref{eq V estimate}, we conclude easily that \eqref{eq V estimate1} holds.

   Meanwhile, there exist generic constants $C_5,C_6>0$ such that
    \begin{align*}
        \|v\|_{H^1}^2&\leq C_5\left(\sum_{1\leq j\leq m}v_j(t)^2-\sum_{j>m}\lambda_j(\varepsilon t)v_j(t)^2\right)\leq C_6\left( \sum_{1\leq j\leq m }v_j(t)^2+\sum_{j\in \N^+}\lambda_j(\varepsilon t)^2v_j(t)^2\right)\\
        &=C_6\left(\sum_{1\leq j\leq m }v_j(t)^2+\|A(\varepsilon t)v\|^2\right),
    \end{align*}
    which implies \eqref{eq V estimate2}.
        
    \end{proof}

    \normalem
    \bibliographystyle{plain}
    \bibliography{References}

\begin{thebibliography}{10}

\bibitem{AM-22}
P.~Alphonse and J.~Martin.
\newblock Stabilization and approximate null-controllability for a large class of diffusive equations from thick control supports.
\newblock {\em ESAIM Control Optim. Calc. Var.}, 28(16):30 pp, 2022.

\bibitem{BT-04}
V.~Barbu and R.~Triggiani.
\newblock Internal stabilization of {N}avier-{S}tokes equations with finite-dimensional controllers.
\newblock {\em Indiana Univ. Math. J.}, 53(5):1443--1494, 2004.

\bibitem{BBPS-22d}
J.~Bedrossian, A.~Blumenthal, and S.~Punshon-Smith.
\newblock Almost-sure exponential mixing of passive scalars by the stochastic {Navier-Stokes} equations.
\newblock {\em Ann. Probab.}, 50(1):241–303, 2022.

\bibitem{BBPS-22c}
J.~Bedrossian, A.~Blumenthal, and S.~Punshon-Smith.
\newblock Lagrangian chaos and scalar advection in stochastic fluid mechanics.
\newblock {\em J. Eur. Math. Soc. (JEMS)}, 24(6):1893--1990, 2022.

\bibitem{Bogachev-98}
V.~I. Bogachev.
\newblock {\em Gaussian measures}.
\newblock American Mathematical Society, Providence, RI, 1998.

\bibitem{BKL-02}
J.~Bricmont, A.~Kupiainen, and R.~Lefevere.
\newblock Exponential mixing of the 2{D} stochastic {N}avier-{S}tokes dynamics.
\newblock {\em Comm. Math. Phys.}, 230(1):87–132, 2002.

\bibitem{CLXZ-24}
Y.~Chen, Z.~Liu, S.~Xiang, and Z.~Zhang.
\newblock Local large deviations for randomly forced nonlinear wave equations with localized damping.
\newblock {\em arXiv:2409.11717}, 2024.

\bibitem{CX-26}
Y.~Chen and S.~Xiang.
\newblock {Donsker-Varadhan} large deviation principle for locally damped and randomly forced {NLS} equations.
\newblock {\em Ann. Henri Poincaré}, 2026, early access.

\bibitem{CXZZ-25}
Y.~Chen, S.~Xiang, Z.~Zhang, and J.-C. Zhao.
\newblock Exponential mixing for the randomly forced {NLS} equation.
\newblock {\em arXiv:2506.10318}, 2025.

\bibitem{CR-25}
W.~Cooperman and K.~Rowan.
\newblock Exponential scalar mixing for the {2D Navier–Stokes} equations with degenerate stochastic forcing.
\newblock {\em Invent. math.}, To appear.

\bibitem{Coron-02}
J.-M. Coron.
\newblock Local controllability of a 1-d tank containing a fluid modeled by the shallow water equations.
\newblock {\em ESAIM Control Optim. Calc. Var.}, 8:513--554, 2002.

\bibitem{Coron-07}
J.-M. Coron.
\newblock {\em Control and nonlinearity}.
\newblock American Mathematical Society, Providence, RI, 2007.

\bibitem{CKX-25}
J.-M. Coron, J.~Krieger, and S.~Xiang.
\newblock Wave maps from circle to {R}iemannian manifold: global controllability is equivalent to homotopy.
\newblock {\em arXiv:2509.12779}, 2025.

\bibitem{CT-04}
J.-M. Coron and E.~Trélat.
\newblock Global steady-state controllability of one-dimensional semilinear heat equations.
\newblock {\em SIAM J. Control Optim.}, 43(2):549--569, 2004.

\bibitem{CX-25}
J.-M. Coron and S.~Xiang.
\newblock Global controllability to harmonic maps of the heat flow from a circle to a sphere.
\newblock {\em J. Math. Pures Appl. (9)}, 204:Paper No. 103761, 47 pp, 2025.

\bibitem{CZH-21}
M.~Coti~Zelati and M.~Hairer.
\newblock A noise-induced transition in the {L}orenz system.
\newblock {\em Commun. Math. Phys.}, 383:2243--2274, 2021.

\bibitem{DPZ-96}
G.~Da~Prato and J.~Zabczyk.
\newblock {\em {E}rgodicity for infinite dimensional systems}.
\newblock Cambridge University Press, Cambridge, 1996.

\bibitem{DZZ-08}
T.~Duyckaerts, X.~Zhang, and E.~Zuazua.
\newblock On the optimality of the observability inequalities for parabolic and hyperbolic systems with potentials.
\newblock {\em Ann. Inst. H. Poincaré C Anal. Non Linéaire}, 25(1):1--41, 2008.

\bibitem{EM-01}
W.~E and J.~C. Mattingly.
\newblock Ergodicity for the {N}avier-{S}tokes equation with degenerate random forcing: finite-dimensional approximation.
\newblock {\em Comm. Pure Appl. Math.}, 54(11):1386–1402, 2001.

\bibitem{EMS-01}
W.~E, J.~C. Mattingly, and Y.~Sinai.
\newblock Gibbsian dynamics and ergodicity for the stochastically forced {N}avier-{S}tokes equation.
\newblock {\em Comm. Math. Phys.}, 224(1):83–106, 2001.

\bibitem{ET-99}
I.~Ekeland and R.~Témam.
\newblock {\em Convex analysis and variational problems}.
\newblock Society for Industrial and Applied Mathematics (SIAM), Philadelphia, PA, 1999.

\bibitem{Emanuilov-95}
O.~Yu. {\`E}manuilov.
\newblock Controllability of parabolic equations.
\newblock {\em Mat. Sb.}, 186(6):879--900, 1995.

\bibitem{FCZ-00}
E.~Fernández-Cara and E.~Zuazua.
\newblock The cost of approximate controllability for heat equations: the linear case.
\newblock {\em Adv. Differential Equations}, 5(4-6):465--514, 2000.

\bibitem{FM-95}
F.~Flandoli and B.~Maslowski.
\newblock Ergodicity of the 2-{D} {N}avier-{S}tokes equation under random perturbations.
\newblock {\em Comm. Math. Phys.}, 172(1):119–141, 1995.

\bibitem{FGRT-15}
J.~F\"{o}ldes, N.~E. Glatt-Holtz, G.~Richards, and E.~Thomann.
\newblock Ergodic and mixing properties of the {B}oussinesq equations with a degenerate random forcing.
\newblock {\em J. Funct. Anal.}, 269(8):2427–2504, 2015.

\bibitem{FLZ-19}
X.~Fu, Q.~L{\"u}, and X.~Zhang.
\newblock {\em Carleman estimates for second order partial differential operators and applications: a unified approach}.
\newblock Springer, Cham, 2019.

\bibitem{GLLL-25}
F.~Gong, Y.~Liu, Y.~Liu, and Z.~Liu.
\newblock Asymptotic stability for non-equicontinuous {M}arkov semigroups.
\newblock {\em Commun. Math. Stat.}, 2025, early access.

\bibitem{GLLL-24}
F.~Gong, Y.~Liu, Y.~Liu, and Z.~Liu.
\newblock Ergodicity for eventually continuous {Markov--Feller semigroups on Polish spaces}.
\newblock {\em Sci. China Math.}, 2026, early access.

\bibitem{HM-06}
M.~Hairer and J.~C. Mattingly.
\newblock Ergodicity of the 2{D} {N}avier-{S}tokes equations with degenerate stochastic forcing.
\newblock {\em Ann. of Math. (2)}, 164(3):993--1032, 2006.

\bibitem{HM-08}
M.~Hairer and J.~C. Mattingly.
\newblock Spectral gaps in {W}asserstein distances and the 2{D} stochastic {N}avier-{S}tokes equations.
\newblock {\em Ann. Probab.}, 36(6):2050--2091, 2008.

\bibitem{HM-11b}
M.~Hairer and J.~C. Mattingly.
\newblock A theory of hypoellipticity and unique ergodicity for semilinear stochastic {PDE}s.
\newblock {\em Electron. J. Probab.}, 16:658--738, 2011.

\bibitem{HZ-24}
M.~Hairer and W.~Zhao.
\newblock Ergodicity of {2D singular stochastic Navier--Stokes} equations.
\newblock {\em Probab. Math. Phys.}, 6(3):777--818, 2025.

\bibitem{Hale-88}
J.~K. Hale.
\newblock {\em Asymptotic behavior of dissipative systems}.
\newblock American Mathematical Society, Providence, RI, 1988.

\bibitem{Henry-81}
D.~Henry.
\newblock {\em Geometric theory of semilinear parabolic equations}.
\newblock Springer-Verlag, Berlin-New York, 1981.

\bibitem{Khalil-92}
H.~K. Khalil.
\newblock {\em Nonlinear systems}.
\newblock Macmillan Publishing Company, New York, 1992.

\bibitem{KNS-20}
S.~Kuksin, V.~Nersesyan, and A.~Shirikyan.
\newblock Exponential mixing for a class of dissipative {PDE}s with bounded degenerate noise.
\newblock {\em Geom. Funct. Anal.}, 30(1):126--187, 2020.

\bibitem{KS-02}
S.~Kuksin and A.~Shirikyan.
\newblock Coupling approach to white-forced nonlinear {PDEs}.
\newblock {\em J. Math. Pures Appl. (9)}, 81(6):567--602, 2002.

\bibitem{Lions-88}
J.-L. Lions.
\newblock {\em Contr\^{o}labilit\'{e} exacte, perturbations et stabilisation de syst\`{e}mes distribu\'{e}s. Tome 1}.
\newblock Masson, Paris, 1988.

\bibitem{Lions-92}
J.-L. Lions.
\newblock Remarks on approximate controllability.
\newblock {\em J. Anal. Math.}, 59:103--116, 1992.

\bibitem{LWXZZ-24}
Z.~Liu, D.~Wei, S.~Xiang, Z.~Zhang, and J.-C. Zhao.
\newblock Exponential mixing for random nonlinear wave equations: weak dissipation and localized control.
\newblock {\em arXiv:2407.15058}, 2024.

\bibitem{MY-02}
N.~Masmoudi and L.-S. Young.
\newblock Ergodic theory of infinite dimensional systems with applications to dissipative parabolic {PDE}s.
\newblock {\em Comm. Math. Phys.}, 227(3):461--481, 2002.

\bibitem{Mattingly-02b}
J.~C. Mattingly.
\newblock The dissipative scale of the stochastics {Navier-Stokes} equation: regularization and analyticity.
\newblock {\em J. Statist. Phys.}, 108(5-6):1157--1179, 2002.

\bibitem{Mattingly-02}
J.~C. Mattingly.
\newblock Exponential convergence for the stochastically forced {N}avier-{S}tokes equations and other partially dissipative dynamics.
\newblock {\em Comm. Math. Phys.}, 230(3):421--462, 2002.

\bibitem{MP-06}
J.~C. Mattingly and É. Pardoux.
\newblock Malliavin calculus for the stochastic 2{D} {N}avier-{S}tokes equation.
\newblock {\em Comm. Pure Appl. Math.}, 59(12):1742--1790, 2006.

\bibitem{NZZ-24}
V.~Nersesyan, D.~Zhang, and C.~Zhou.
\newblock On the chaotic behavior of the {L}agrangian flow of the {2D Navier-Stokes} system with bounded degenerate noise.
\newblock {\em arXiv:2406.17612}, 2024.

\bibitem{NZ-24}
V.~Nersesyan and M.~Zhao.
\newblock Polynomial mixing for the white-forced {Navier-Stokes} system in the whole space.
\newblock {\em arXiv:2410.15727}, 2024.

\bibitem{Nualart-06}
D.~Nualart.
\newblock {\em The {M}alliavin calculus and related topics}.
\newblock Springer-Verlag, Berlin, second edition, 2006.

\bibitem{PZZ-24}
X.~Peng, J.~Zhai, and T.~Zhang.
\newblock Ergodicity for {2D Navier-Stokes} equations with a degenerate pure jump noise.
\newblock {\em arXiv:2405.00414}, 2024.

\bibitem{Rockafellar-67}
R.~T. Rockafellar.
\newblock Duality and stability in extremum problems involving convex functions.
\newblock {\em Pacific J. Math.}, 21:167--187, 1967.

\bibitem{Shi-15}
A.~Shirikyan.
\newblock Control and mixing for 2{D} {N}avier-{S}tokes equations with space-time localised noise.
\newblock {\em Ann. Sci. \'{E}c. Norm. Sup\'{e}r}, 48(2):253--280, 2015.

\bibitem{Shi-21}
A.~Shirikyan.
\newblock Controllability implies mixing {II}. {C}onvergence in the dual-{L}ipschitz metric.
\newblock {\em J. Eur. Math. Soc. (JEMS)}, 23(4):1381--1422, 2021.

\bibitem{TWX-20}
E.~Trélat, G.~Wang, and Y.~Xu.
\newblock Characterization by observability inequalities of controllability and stabilization properties.
\newblock {\em Pure Appl. Anal.}, 59(1):93--122, 2020.

\bibitem{Xiang-23}
S.~Xiang.
\newblock Small-time local stabilization of the two-dimensional incompressible {Navier-Stokes} equations.
\newblock {\em Ann. Inst. H. Poincaré C Anal. Non Linéaire}, 40(6):1487--1511, 2023.

\bibitem{Xiang-24}
S.~Xiang.
\newblock Quantitative rapid and finite time stabilization of the heat equation.
\newblock {\em ESAIM Control Optim. Calc. Var.}, 30:Paper No. 40, 25 pp, 2024.

\end{thebibliography}

    \end{document}